\documentclass[11pt,hidelinks]{article}

\usepackage{orcidlink}

\title{Crossover from subcritical to critical decay:
\\ random walk, self-avoiding walk, percolation
}

\author{
Yucheng Liu\,\orcidlink{0000-0002-1917-8330}\thanks{Beijing International Center for Mathematical Research,
	Peking University,
	Beijing, China 100871.
	\href{mailto:yliu135@pku.edu.cn}{yliu135@pku.edu.cn}.
	}
\and
Gordon Slade\,\orcidlink{0000-0001-9389-9497}\thanks{Department of Mathematics,
	University of British Columbia,
	Vancouver, BC, Canada V6T 1Z2.
	\href{mailto:slade@math.ubc.ca}{slade@math.ubc.ca}.
	}
}
\date{\vspace{-5ex}} 

\usepackage{amsmath, amssymb, amscd, amsthm, amsfonts}
\usepackage{graphicx} 
\usepackage{hyperref} 

\usepackage{booktabs}
\usepackage{xcolor}
\usepackage{ dsfont, bm}
\usepackage[title]{appendix}
\usepackage{comment}
\usepackage{enumerate}
\usepackage{enumitem}
\usepackage{cite}
\usepackage{theoremref}

\usepackage[textwidth=480pt,textheight=650pt,centering]{geometry} 

\usepackage{tikz}
\tikzset{every picture/.style={line width=0.75pt}}

\theoremstyle{plain}
\newtheorem{theorem}{Theorem}[section]
\newtheorem{lemma}[theorem]{Lemma}
\newtheorem{conjecture}[theorem]{Conjecture}
\newtheorem{proposition}[theorem]{Proposition}
\newtheorem{corollary}[theorem]{Corollary}
\newtheorem{definition}[theorem]{Definition}

\theoremstyle{definition}
\newtheorem{remark}[theorem]{Remark}
\newtheorem{example}[theorem]{Example}
\newtheorem{assumption}{Assumption}

\numberwithin{equation}{section}

\newcommand{\D}{\mathrm{d}}

\newcommand{\ie}{i.e.}

\newcommand{\eps}{\varepsilon}

\newcommand{\Z}{\mathbb{Z}}

\newcommand{\R}{\mathbb{R}}
\newcommand{\C}{\mathbb{C}}

\renewcommand{\P}{\mathbb{P}}
\newcommand{\T}{\mathbb{T}}

\newcommand{\mG}{\mathbb{G}}
\newcommand{\bg}{\mathbb{C}}

\newcommand{\Dcal}{\mathcal{D}}

\newcommand{\Ical}{\mathcal{I}}

\newcommand{\Qcal}{\mathcal{Q}}
\newcommand{\Lcal}{\mathcal{L}}
\newcommand{\Rcal}{\mathcal{R}}

\newcommand{\Ucal}{\mathcal{U}}

\newcommand{\Wcal}{\mathcal{W}}
\newcommand{\Zcal}{\mathcal{Z}}

\newcommand{\Rd}{{\mathbb R^d}}
\newcommand{\Zd}{{\mathbb Z^d}}
\newcommand{\Td}{{\mathbb T^d}}

\DeclareMathOperator{\arccosh}{arccosh}
\DeclareMathOperator{\arcsinh}{arcsinh}

\newcommand{\del}{\partial}
\newcommand{\grad}{\nabla}
\newcommand{\inv}{^{-1}}

\newcommand{\half}{\frac{1}{2}}

\newcommand{\1}{\mathds{1}}
\newcommand{\nl}{\nonumber \\}

\renewcommand{\Re}{\mathrm{Re}\,}
\renewcommand{\Im}{\mathrm{Im}\,}

\providecommand{\abs}[1]{\lvert#1\rvert}
\providecommand{\bigabs}[1]{\bigl\lvert#1\bigr\rvert}
\providecommand{\Bigabs}[1]{\Bigl\lvert#1\Bigr\rvert}
\providecommand{\biggabs}[1]{\biggl\lvert#1\biggr\rvert}

\providecommand{\norm}[1]{\lVert#1\rVert}
\providecommand{\bignorm}[1]{\bigl\lVert#1\bigr\rVert}
\providecommand{\Bignorm}[1]{\Bigl\lVert#1\Bigr\rVert}
\providecommand{\biggnorm}[1]{\biggl\lVert#1\biggr\rVert}

\newcommand{\mz}{m_z}
\newcommand{\mS}{m_S}
\renewcommand{\mp}{m_p}

\newcommand{\mux}{{\mu_{\hat x}}}
\newcommand{\muxz}{{\mu_{\hat x,z}}}
\newcommand{\etax}{{\eta_{\hat x}}}

\newcommand{\supmu}{^{(\mu)}}
\newcommand{\supmux}{^{(\mu_{\hat x})}}
\newcommand{\supmuxz}{^{(\mu_{\hat x,z})}}

\newcommand{\supk}[1]{^{(#1)}}

\newcommand{\supzero}{^{(0)}}

\newcommand{\hatx}{{\hat x}}
\newcommand{\hatxz}{{\hat x,z}}
\newcommand{\conv}{\mathrm{Conv}}
\newcommand{\rad}{\mathrm{rad}}

\newcommand{\lam}{\lambda}
\newcommand{\Lam}{\Lambda}

\newcommand{\bubble}{{\sf B}}
\newcommand{\B}{\bubble}
\newcommand{\Bsupmu}{\bubble^{(2\mu)}}
\newcommand{\Triangle}{{\sf T}}
\newcommand{\w}{\hat{b}}

\definecolor{darkmagenta}{rgb}{0.55, 0.0, 0.55}

\newcommand{\nnb}{\nonumber\\}

\newcommand{\veee}[1]{|\!|\!|#1|\!|\!|}

\providecommand{\nnnorm}[1]{\veee {#1}}
\newcommand{\const}{\mathrm{const}}

\providecommand{\floor}[1]{\lfloor #1 \rfloor}

\providecommand{\integer}[1]{\floor{#1}}
\providecommand{\Biginteger}[1]{\Bigl\lfloor #1 \Bigr\rfloor}
\providecommand{\fractional}[1]{\mathrm{frac}(#1)}

\newcommand{\Dnn}{P}
\newcommand{\KIR}{K_{\mathrm{IR}}}

\newcommand{\GL}{S}
\newcommand{\GQ}{ G_{\Jsupmu,g} }
\newcommand{\Jsupmu}{Q}
\newcommand{\g}{h}
\newcommand{\so}{s_0}
\newcommand{\zetao}{{\zeta_0}}
\newcommand{\sym}{{p,\mathrm{sym}}}
\newcommand{\sat}{{{\rm sat}}}

\newcommand{\Casy}{B}

\newcommand{\SAWOne}{
\begin{tikzpicture}[x=0.75pt,y=0.75pt,yscale=-1,xscale=1]

\draw    (140,120) -- (180,120) ;
\filldraw[fill=white] (140,120) circle (0pt) node[left]{$0$};
\draw    (140,120) -- (160.4,85.4) ;
\draw    (160.4,85.4) -- (200.4,85.4) ;
\draw    (160.4,85.4) -- (180,120) ;
\draw    (180,120) .. controls (179.41,117.72) and (180.26,116.28) .. (182.54,115.69) .. controls (184.82,115.1) and (185.67,113.67) .. (185.08,111.39) .. controls (184.49,109.11) and (185.34,107.67) .. (187.62,107.08) .. controls (189.9,106.49) and (190.75,105.05) .. (190.16,102.77) .. controls (189.57,100.49) and (190.42,99.05) .. (192.7,98.46) .. controls (194.98,97.87) and (195.83,96.44) .. (195.24,94.16) .. controls (194.65,91.88) and (195.5,90.44) .. (197.78,89.85) .. controls (200.06,89.26) and (200.91,87.82) .. (200.32,85.54) -- (200.4,85.4) -- (200.4,85.4) ;
\filldraw[fill=white] (160.4,85.4) circle (0pt) node[above]{$a$};
\filldraw[fill=white] (200.4,85.4) circle (0pt) node[above]{$c$};
\draw  [draw opacity=0] (140,119.84) .. controls (140.06,105) and (148.19,92.07) .. (160.24,85.21) -- (180,120) -- cycle ; \draw   (140,119.84) .. controls (140.06,105) and (148.19,92.07) .. (160.24,85.21) ;
\draw  [draw opacity=0] (280.4,85.4) .. controls (280.36,100.36) and (272.11,113.39) .. (259.91,120.22) -- (240.4,85.3) -- cycle ; \draw   (280.4,85.4) .. controls (280.36,100.36) and (272.11,113.39) .. (259.91,120.22) ;
\draw    (183.2,81.8) -- (178.8,89.4) ;
\draw    (162.8,116.4) -- (158.4,124) ;
\draw    (180,120) .. controls (181.67,118.33) and (183.33,118.33) .. (185,120) .. controls (186.67,121.67) and (188.33,121.67) .. (190,120) .. controls (191.67,118.33) and (193.33,118.33) .. (195,120) .. controls (196.67,121.67) and (198.33,121.67) .. (200,120) .. controls (201.67,118.33) and (203.33,118.33) .. (205,120) .. controls (206.67,121.67) and (208.33,121.67) .. (210,120) .. controls (211.67,118.33) and (213.33,118.33) .. (215,120) .. controls (216.67,121.67) and (218.33,121.67) .. (220,120) -- (220,120) ;
\filldraw[fill=white] (180,120) circle (0pt) node[below]{$b$};
\filldraw[fill=white] (220,120) circle (0pt) node[below]{$d$};
\draw    (220,120) -- (240.4,85.4) ;
\draw    (200.4,85.4) -- (240.4,85.4) ;
\draw    (200.4,85.4) -- (220,120) ;
\draw    (223.2,81.8) -- (218.8,89.4) ;
\draw    (240.4,85.4) -- (280.4,85.4) ;
\filldraw[fill=white] (280.4,85.4) circle (0pt) node[right]{$x$};
\draw    (259.91,120.22) -- (280.31,85.62) ;
\draw    (220,120) -- (260,120) ;
\draw    (240.4,85.3) -- (260,119.9) ;
\draw    (262.2,81.8) -- (257.8,89.4) ;
\draw    (243.2,115.8) -- (238.8,123.4) ;
\end{tikzpicture}
}

\newcommand{\SAWTwo}{
\begin{tikzpicture}[x=0.75pt,y=0.75pt,yscale=-1,xscale=1]

\draw    (140,120) -- (180,120) ;
\filldraw[fill=white] (140,120) circle (0pt) node[left]{$0$};
\draw    (140,120) -- (160.4,85.4) ;
\draw    (160.4,85.4) -- (200.4,85.4) ;
\draw    (160.4,85.4) -- (180,120) ;
\draw    (180,120) -- (200.4,85.4) ;
\draw  [draw opacity=0] (140,119.84) .. controls (140.06,105) and (148.19,92.07) .. (160.24,85.21) -- (180,120) -- cycle ; \draw   (140,119.84) .. controls (140.06,105) and (148.19,92.07) .. (160.24,85.21) ;
\draw  [draw opacity=0] (280.4,85.4) .. controls (280.36,100.36) and (272.11,113.39) .. (259.91,120.22) -- (240.4,85.3) -- cycle ; \draw   (280.4,85.4) .. controls (280.36,100.36) and (272.11,113.39) .. (259.91,120.22) ;
\draw    (183.2,81.8) -- (178.8,89.4) ;
\draw    (162.8,116.4) -- (158.4,124) ;
\draw    (180,120) .. controls (181.67,118.33) and (183.33,118.33) .. (185,120) .. controls (186.67,121.67) and (188.33,121.67) .. (190,120) .. controls (191.67,118.33) and (193.33,118.33) .. (195,120) .. controls (196.67,121.67) and (198.33,121.67) .. (200,120) .. controls (201.67,118.33) and (203.33,118.33) .. (205,120) .. controls (206.67,121.67) and (208.33,121.67) .. (210,120) .. controls (211.67,118.33) and (213.33,118.33) .. (215,120) .. controls (216.67,121.67) and (218.33,121.67) .. (220,120) -- (220,120) ;
\filldraw[fill=white] (180,120) circle (0pt) node[below]{$b$};
\filldraw[fill=white] (220,120) circle (0pt) node[below]{$d$};
\draw    (220,120) -- (240.4,85.4) ;
\draw    (200.4,85.4) -- (240.4,85.4) ;
\draw    (200.4,85.4) .. controls (202.67,86.03) and (203.49,87.48) .. (202.86,89.75) .. controls (202.23,92.02) and (203.06,93.47) .. (205.33,94.1) .. controls (207.6,94.73) and (208.42,96.18) .. (207.79,98.45) .. controls (207.16,100.72) and (207.99,102.17) .. (210.26,102.8) .. controls (212.53,103.43) and (213.35,104.88) .. (212.72,107.15) .. controls (212.09,109.42) and (212.92,110.87) .. (215.19,111.5) .. controls (217.46,112.13) and (218.28,113.58) .. (217.65,115.85) -- (220,120) -- (220,120) ;
\filldraw[fill=white] (200.4,85.4) circle (0pt) node[above]{$a$};
\filldraw[fill=white] (240.4,85.4) circle (0pt) node[above]{$c$};
\draw    (223.2,81.8) -- (218.8,89.4) ;
\draw    (240.4,85.4) -- (280.4,85.4) ;
\filldraw[fill=white] (280.4,85.4) circle (0pt) node[right]{$x$};
\draw    (259.91,120.22) -- (280.31,85.62) ;
\draw    (220,120) -- (260,120) ;
\draw    (240.4,85.3) -- (260,119.9) ;
\draw    (262.2,81.8) -- (257.8,89.4) ;
\draw    (243.2,115.8) -- (238.8,123.4) ;
\end{tikzpicture}
}

\newcommand{\PercH}{
\begin{tikzpicture}[x=0.75pt,y=0.75pt,yscale=-1.4,xscale=1.4]

\filldraw[fill=white] (70,110) circle (0pt) node[left]{$z$};
\draw    (70,110) -- (90,110) ;
\filldraw[fill=white] (90,109) circle (0pt) node[below]{$0$};
\draw    (90,110) -- (100,95) ;
\filldraw[fill=white] (110,111) circle (0pt) node[below]{$x$};
\draw    (90,110) -- (110,110) ;
\filldraw[fill=white] (99,94) circle (0pt) node[right]{$w$};
\draw    (100,95) -- (110,110) ;
\filldraw[fill=white] (130,110) circle (0pt) node[right]{$y$};
\draw    (110,110) -- (130,110) ;
\filldraw[fill=white] (100,80) circle (0pt) node[above]{$v$};
\draw    (100,80) -- (100,95) ;
\filldraw[fill=white] (70,80) circle (0pt) node[left]{$z+b$};
\draw    (70,80) .. controls (71.67,78.33) and (73.33,78.33) .. (75,80) .. controls (76.67,81.67) and (78.33,81.67) .. (80,80) .. controls (81.67,78.33) and (83.33,78.33) .. (85,80) .. controls (86.67,81.67) and (88.33,81.67) .. (90,80) .. controls (91.67,78.33) and (93.33,78.33) .. (95,80) .. controls (96.67,81.67) and (98.33,81.67) .. (100,80) -- (100,80) ;
\filldraw[fill=white] (130,80) circle (0pt) node[right]{$y+c$};
\draw    (100,80) .. controls (101.67,78.33) and (103.33,78.33) .. (105,80) .. controls (106.67,81.67) and (108.33,81.67) .. (110,80) .. controls (111.67,78.33) and (113.33,78.33) .. (115,80) .. controls (116.67,81.67) and (118.33,81.67) .. (120,80) .. controls (121.67,78.33) and (123.33,78.33) .. (125,80) .. controls (126.67,81.67) and (128.33,81.67) .. (130,80) -- (130,80) ;
\end{tikzpicture}
}

\begin{document}
\maketitle

\begin{abstract}
The study of the Ornstein--Zernike decay of subcritical two-point functions in equilibrium statistical mechanics has a history going back over a century.
Despite this, the crossover from Ornstein--Zernike decay to critical power-law decay has received scant attention in the literature.
We prove a general theorem which, under appropriate hypotheses, identifies the asymptotic behaviour
of the solution to an Ornstein--Zernike equation
on $\mathbb{Z}^d$
as that of the Green function for Brownian motion with drift,
multiplied by an anisotropic exponentially decaying factor.
The theorem applies to a wide class of
random walks,
to nearest-neighbour self-avoiding walk in dimensions $d \ge 5$,
and to nearest-neighbour percolation in
dimensions $d \ge 15$.
Wide-ranging consequences follow, including details of the crossover
from Ornstein--Zernike to critical decay on the scale of the correlation length, and the fact
that all finite-order correlation lengths are equivalent up to 
universal constants.
The proof is based on a variational characterisation of the direction-dependent rate of exponential decay
and a major extension of Hara's 2008 Gaussian Lemma to noncentred kernels.
\end{abstract}

%
\setcounter{tocdepth}{2}
\tableofcontents

\section{Introduction}
\label{sec:introduction}

\subsection{Ornstein--Zernike versus critical decay}
\label{sec:OZintro}

Let $\g,J,\GL$ be real-valued functions defined on the integer lattice $\Z^d$.
Convolution equations of the form
\begin{equation} \label{eq:OZeq-intro}
    \GL = \g + J*\GL
\end{equation}
are ubiquitous in mathematics and physics.  Our first purpose is to
determine the long-distance asymptotic behaviour of the solution $\GL$
to \eqref{eq:OZeq-intro}, under appropriate assumptions on $\g$ and $J$.

Formally, the solution to \eqref{eq:OZeq-intro} is given by
$\GL = (\delta_0 -J)^{-1}\g$,
where $\delta_0(x) = \delta_{0,x}  = \1_{x=0}$ is the Kronecker delta.
When $\g = \delta_0$,
the solution $G = ( \delta_0-J)^{-1}\delta_0$ is the Green function for the
operator $\delta_0-J$.  The solution to the general case \eqref{eq:OZeq-intro} is then $\GL=G*\g$.  From this perspective, our first purpose is to determine the
asymptotic behaviour of the Green function for $\delta_0-J$.  For example,
if $J(x)=\frac{z}{2d}\1\{\|x\|_1=1\}$ with $z\le 1$, then $G$ is the Green function for the
lattice Laplacian (the massive Laplacian when $z<1$).

Following the pioneering work of Ornstein and
Zernike from the early 1900s \cite{OZ14,Zern16}, we are particularly
interested in determining conditions
on $\g$ and $J$ which guarantee that $\GL$ decays as
\begin{equation} \label{eq:OZdecay-intro}
    \GL(x) \sim
      \frac{ c_{\hat x} }{ \abs x_S^{(d-1)/2} } e^{- m_S \abs x_S } .
\end{equation}
This \emph{Ornstein--Zernike} (OZ) decay in \eqref{eq:OZdecay-intro}
involves a constant factor $c_{\hat x}$ which depends only on the
direction of $x$, a norm $|\cdot|_S$ on $\R^d$ which
satisfies $\|x\|_\infty \le |x|_S \le \|x\|_1$,
and a \emph{mass} $m_S >0$ (also called inverse \emph{correlation length}).

More generally, we consider the case where $\g,J,\GL$ depend on a
parameter $z \in [0,z_c]$, so that
\begin{equation} \label{eq:OZeq-z-intro}
    \GL_z = \g_z + J_z*\GL_z.
\end{equation}
Examples we study in detail are:
\begin{enumerate}
\item[(i)]
    The Green function $\GL_z(x)=\sum_{n=0}^\infty\P(X_n=x) z^n$ for an
     irreducible random walk $(X_n)_{n \ge 0}$ on $\Zd$ with super-exponentially decaying
    transition kernel $D(x)= \P(X_1=x)$.
    The critical value is $z_c=1$.  For $d>2$ and $z\in [0,z_c]$, and
    for $d\le 2$ and $z \in [0,z_c)$,
    $\GL_z$ obeys \eqref{eq:OZeq-z-intro} with $\g_z  = \delta_0$
    and $J_z=zD$, by the Markov property.
\item[(ii)]
    The two-point function $\GL_z(x) = \sum_{n=0}^\infty
    c_n(x)z^n$
    for  the nearest-neighbour self-avoiding walk,
    where $c_n(x)$ is the number of $n$-step self-avoiding walks
    from $0$ to $x$, in dimensions $d \ge 5$.  There is a critical point $z_c$ such that
    $\GL_z$ obeys \eqref{eq:OZeq-z-intro} for $z\in [0,z_c]$,
    with $\g_z = \delta_0$ and $J_z=2dzP+\Pi_z$, where
    $P(x)=\frac{1}{2d}\1\{\|x\|_1=1\}$ and $\Pi_z(x)$ is an explicit function
    given by the lace expansion \cite{BS85,MS93}.
\item[(iii)]
    The two-point function $\GL_z(x) = \P_z(0 \leftrightarrow x)$
    for nearest-neighbour bond percolation, with occupation probability $z$
    in dimensions $d \ge 15$.
    There is a critical point $z_c$ such that
    $\GL_z$ obeys \eqref{eq:OZeq-z-intro} for $z\in [0,z_c]$, with $\g_z$ an explicit
    function given by the lace expansion and $J_z= 2dzP*\g_z$
    \cite{HS90a}.
\end{enumerate}
In all three examples,
the mass $m_{S_z}$ of $S_z$ vanishes at the critical point $z_c$,
and the decay of $\GL_{z_c}$
(with $d>2$ for random walk) is
\begin{equation} \label{eq:Scrit}
    \GL_{z_c}(x) \sim \frac{ \const }{\|x\|_2^{d-2}}.
\end{equation}
In particular, except for a coincidence when $d=3$ (when $\frac{d-1}{2}=d-2$), the critical
decay \eqref{eq:Scrit} does not correspond to setting the mass equal to zero in
the OZ decay \eqref{eq:OZdecay-intro}.
Our second purpose, in addition to determining conditions on $\g_z$ and
$J_z$ which lead to OZ decay for $\GL_z$ when $z<z_c$, is to explain
the crossover between the subcritical OZ decay \eqref{eq:OZdecay-intro} and the critical decay \eqref{eq:Scrit}.

In our main result,
we determine hypotheses on $\g_z$ and $J_z$ under which both
goals are achieved.
These hypotheses are satisfied in each of the above three examples.
In particular, we prove that
\begin{equation}  \label{eq:crossover-intro}
\GL_z(x)  \asymp
	\frac{  \max\{ 1 , m_z \abs x_z \} ^{ (d-3)/2} } { \abs x_z^{d-2} }
	e^{- m_z \abs x_z }
\end{equation}
for all large $\abs x$ in dimensions $d>2$.
We also prove that the subcritical norm $|\cdot|_z$ (anisotropic due to lattice effects) converges to the Euclidean norm $\| \cdot \|_2$  as $z\to z_c$.
This gives the critical decay with power $d-2$ when $\abs x_z$ is below the correlation length $\mz\inv$, and gives the OZ power $\frac{d-1}{2}$ when $\abs x_z$ exceeds the correlation length.
Moreover, our results provide
not just upper and lower bounds with different constants (the meaning
of the notation ``$\asymp$'' above), but also provide exact asymptotic
formulas which give fully detailed information about the crossover.

For example, it follows from our asymptotic result that, for every $\phi > 0$, the correlation length $\xi_\phi(z)$ of order $\phi$ (defined in \eqref{eq:xidef}) obeys the asymptotic formula
\begin{equation}
\xi_\phi(z)   \sim    \frac {a_\phi} {m_z}
	\qquad (z\to z_c)
\end{equation}
with an explicit constant $a_\phi  > 0$ that depends only on $\phi$ and $d$. This implies that all finite-order correlation lengths are equivalent to $m_z\inv$ (``the'' correlation length) and to each other.

There is an extensive literature on OZ decay in the subcritical regime.
In dimensions $d \ge 2$,
the OZ decay \eqref{eq:OZdecay-intro}
has been proved for self-avoiding walk \cite{CC86,CC86b,MS93,Ioff98,IV08},
for percolation \cite{CCC91,CI02},
for the Ising model \cite{CIV03},
and for random cluster models \cite{CIV08,DM26}.
There is also an extensive literature on the critical decay \eqref{eq:Scrit} for high-dimensional statistical-mechanical models.
Most relevant for our results are the fact that \eqref{eq:Scrit} is proved, using the lace expansion, for self-avoiding walk in dimensions $d \ge 5$ \cite{HS92a,Hara08}, and for percolation in dimensions $d \ge 11$ \cite{HS90a,Hara08,FH17}, but the literature is much larger (e.g., \cite{HHS03,Saka07,BHH21,FH21,LS24a,LS24b,DL26}).  In this paper, we give new proofs of OZ
decay for self-avoiding walk in dimensions $d \ge 5$ and percolation in
dimensions $d \ge 15$,
for $z \in [z_c-\delta,z_c)$ for some
$\delta>0$.  These new proofs reveal, for the first time, the crossover present in \eqref{eq:crossover-intro} in these two contexts.
We do not attempt to apply our methods to the high-dimensional Ising
or $|\varphi|^4$ models, or to lattice trees and lattice animals, but we expect that
they should apply when combined with the lace expansions for these models
\cite{HS90b,Saka07,BHH21,Saka15}.

Our results also include a proof of \eqref{eq:crossover-intro} and its more exact asymptotic formula for a large class of random walks.
This class includes the simple random walk, whose Green function is $\GL_z^{\rm SRW}(x)=\sum_{n=0}^\infty z^n P^{*n}(x)$, where $P(x)=\frac{1}{2d}\1\{\|x\|_1=1\}$ is the nearest-neighbour transition kernel.  Explicitly, for $d \ge 1$, $z \in [0,1)$,
and $x\in\Z^d$,
\begin{equation} \label{eq:lattice_Green_fcn}
    \GL_z^{\rm SRW}(x)
    = \int_{[-\pi,\pi]^d} \frac{\D k}{(2\pi)^d}  \frac{e^{ik\cdot x}}{1-z\hat P(k)}
    \quad
    \text{with} \quad
    \hat P(k)=\frac 1d \sum_{j=1}^d \cos k_j .
\end{equation}
A version of \eqref{eq:crossover-intro} was proved in \cite{MS22} for $\GL_z^{\rm SRW}(x)$,
by relying heavily on an explicit representation of
$\GL_z^{\rm SRW}(x)$ as a one-dimensional integral of the modified Bessel function of the first kind, together
with known properties of Bessel functions.
The method of \cite{MS22} does not generalise to our setting.
Our approach is also based on an analysis of Fourier integrals:
we extend Hara's Gaussian Lemma \cite{Hara08}
to noncentred kernels.

We expect, but do not prove, that when the critical
decay has the more general power $\|x\|_2^{-(d-2+\eta)}$,
the crossover bounds
\eqref{eq:crossover-intro} become
\begin{equation} \label{eq:crossover-general}
\GL_z(x)  \asymp
	\frac{  \max\{ 1 , m_z \abs x_z \} ^{ \eta + (d-3)/2} } { \abs x_z^{d-2 + \eta} }
	e^{- m_z \abs x_z }  .
\end{equation}
For the $2$-dimensional Ising model at inverse temperature $z<z_c$, thanks to an exact solution
for the high-temperature spin-spin correlation function, \eqref{eq:crossover-general}  has been proved with $\eta=\frac 14$
and $m_z \sim c(z_c-z)^{1}$  \cite[Theorem~2.7.2]{Palm07} (the norm is as in \eqref{eq:normSRW}, see \cite{Mess06}).
Recently, a formula similar to \eqref{eq:crossover-general} has been established for the 2-dimensional
random cluster model with $1 \le q < 4$ \cite{DM26}.  The results of \cite{DM26} are less explicit than \eqref{eq:crossover-general}, since the critical exponent $\eta$ is not generally known to exist; the crossover is instead captured via the one-arm probability.

It is tempting to expect that OZ decay will apply whenever subcritical $\g_z$ and $J_z$ have a sufficient number of finite moments.  However, this is not the case for some infinite-range models, as has been discovered and studied in \cite{AIOV21,AIOV21PRE,AOV23,AOV24}.  The counterexamples to OZ decay are connected to a \emph{saturation} phenomenon in which the mass $m_z$ becomes constant for small $z$, rather than diverging to infinity as is common for finite-range
models.  The fact that saturation is excluded from our theory
is discussed in  Section~\ref{sec:counterx}.

\smallskip \noindent \textbf{Notation.}
We write $a \vee b = \max \{ a , b \}$ and $a \wedge b = \min \{ a , b \}$.
We write $f = O(g)$ or $f \lesssim g$ to mean that there exists a constant $C> 0$ such that $\abs {f} \le C\abs {g}$,
write $f \asymp g$ to mean that $f \lesssim g$ and $g \lesssim f$,
write $f\sim g$ to mean that $\lim f/g=1$,
and write $f= o(g)$
to mean that $\lim f/g = 0$.
We use the two notations
\begin{equation}
\abs x = \norm x_2
\end{equation}
for the Euclidean norm on $\Rd$,
with $\norm x_2$ favoured
when the context requires more clarity.
The standard basis vectors of $\Z^d$ are denoted $e_1,\ldots,e_d$.
A function $f$ on $\Z^d$ is said to be $\Z^d$-\emph{symmetric} if the value of $f(x)$ is unchanged when a component of $x$ is multiplied by $-1$ or when the components of $x$ are permuted.
The Fourier transform of $f:\Zd \to \R$ is the function $\hat f:\T^d\to \C$ defined by $\hat f(k)=\sum_{x\in \Zd}f(x) e^{-ik\cdot x}$; here $\T^d$ is the $d$-dimensional torus of period $2\pi$.
The convolution of $L^1$ functions $f,g:\Zd\to \R$ is $(f*g)(x)=\sum_{y\in\Zd} f(x-y)g(y)$.

\subsection{The Crossover Theorem}

In this section, we state our general crossover theorem, Theorem~\ref{thm:crossover}.

\subsubsection{Assumptions}

Our goal is to prove a precise asymptotic formula for a
$\Zd$-symmetric function $S : \Zd \to [0, \infty)$
which decays exponentially as $\abs x \to \infty$.
In order to identify the polynomial prefactor in the exponential decay, it is useful to
define the \emph{exponential tilt} of a function $f:\Zd\to \R$,
for $\mu \in \Rd$, by
\begin{equation}
    f\supmu(x) =f(x)e^{\mu\cdot  x} .
\end{equation}
The \emph{tilted susceptibility} $\chi\supmu$ is then defined by
\begin{equation}
    \chi\supmu = \sum_{x\in \Zd} \GL\supmu (x) .
\end{equation}
By symmetry,
\begin{equation}
\label{eq:chisym}
    \chi\supmu
    = \sum_{x\in \Zd} \GL(x)\cosh (\mu\cdot x)
    = \sum_{x\in \Zd} \GL(x) \prod_{j=1}^d \cosh (\mu_j x_j).
\end{equation}
We also define a set $\Omega\subset \R^d$ by
\begin{equation}
\label{eq:Omega-def}
\Omega = \bigl\{\mu\in \R^d :
	\chi\supmu
	< \infty \bigr\} .
\end{equation}
The closure $\overline{\Omega}$ of $\Omega$ is called the \emph{Wulff shape} (the terminology is justified in Section~\ref{sec:Wulff}).
The first assumption
says that $S(x)$ decays exponentially in an averaged sense.

\begin{assumption} \label{ass:Omega}
We assume $S:\Zd \to [0,\infty)$ is a $\Zd$-symmetric function
such that $\Omega$ is an open set containing $0$,
and such that $\mu_1 e_1 \not \in \Omega$ for some $\mu_1 > 0$.
\end{assumption}

The next lemma indicates some
 elementary properties that follow from Assumption~\ref{ass:Omega}.
We define
\begin{equation} \label{eq:def_mS}
\mS  = \sup \bigl\{ t \ge 0 : \chi \supk{ te_1} < \infty \bigr\} .
\end{equation}

\begin{lemma} \label{lem:Omega_convex}
Under Assumption~\ref{ass:Omega},
the set $\Omega$ is bounded and convex.
In addition, $\mS  \in (0,\mu_1]$ and the closure $\overline\Omega$ of $\Omega$
obeys
\begin{equation} \label{eq:Omega_incl}
    \bigl\{\mu\in \R^d : \|\mu\|_1 \le \mS  \bigr\}  \subset
    \overline \Omega
    \subset \bigl\{\mu\in \R^d : \|\mu\|_\infty \le \mS  \bigr\} .
\end{equation}
Also, if $\mu\in \partial \Omega$ then $\|\mu\|_\infty \le \mS  \le
\|\mu\|_1$ and $\chi \supmu = \infty$.
\end{lemma}

\begin{proof}
Since $\GL(x) \ge 0$,
the convexity of $\Omega$ follows from the convexity of the map $\mu \mapsto e^{\mu \cdot x}$ for each $x$.
Since $\Omega$ is open and contains $0$ by Assumption~\ref{ass:Omega}, we have $\mS  > 0$.
The bound $\mS  \le \mu_1$ also follows from Assumption~\ref{ass:Omega}.

For the first inclusion of \eqref{eq:Omega_incl},
since $\mS  e_1 \in \overline \Omega$ by definition,
the $\Zd$-symmetry of $S$ implies that $\pm \mS  e_j \in \overline \Omega$ for all directions $j$. The first inclusion then follows by convexity of $\overline \Omega$.
For the second inclusion,
let $\mu = (\mu_1, \dots, \mu_d) \in \Rd$.
If $|\mu_i| > \mS $ for some $i$, then, by \eqref{eq:chisym},
\begin{equation}
\chi \supk \mu
\ge \sum_{x\in\Z^d} \GL(x) \cosh(\mu_i x_i)
= \infty ,
\end{equation}
so $\mu \not \in \Omega$.
This proves $\Omega \subset \{\mu : \|\mu\|_\infty \le \mS \}$, and the second inclusion follows by taking the closure.

Lastly, if $\mu\in \del \Omega$ then $\norm \mu_1 \ge \mS $ and $\norm \mu_\infty \le \mS $ by \eqref{eq:Omega_incl}.
We also know $\mu \not \in \Omega$, because $\Omega$ is open so $\chi \supmu = \infty$.
\end{proof}

To state our second and more substantial assumption, we first introduce a certain class of functions.
Functions in this class obey specific norm estimates,
as well as an \emph{infrared bound} \eqref{eq:Q_infrared}.

\begin{definition} \label{def:A}
Let $d \ge 1$ and let $M,\KIR,\zeta > 0$ be given.
We define $\Qcal_{M,\KIR,\zeta}$ to be the set of pairs of functions $(\Jsupmu,g)$,
each mapping $\Zd$ into $\R$, such that the following bounds apply:
\begin{gather}
\label{eq:Q_moments}
\norm{ \Jsupmu(y) }_{1} ,\,
\bignorm{ \abs y ^{2+\zeta}  \Jsupmu(y) }_{1}
	\le M ,
\\
\label{eq:Q_infrared}
\Re[ \hat \Jsupmu(0)  - \hat \Jsupmu(k) ]
	\ge \KIR \abs k ^2 	
	\qquad (k\in \Td) ,
\\
\label{eq:g_moments}
\norm{ g(y) }_{1} ,\,
	\bignorm{ \abs y ^{\zeta} g(y) }_{1}  ,\,
    \bignorm{ \abs y ^{2}  g(y) }_{ p_d \wedge 2  }
	\le M ,
\end{gather}
with $p_d=2$ for $d=1$ and $p_d=\frac{d}{d-2+(\zeta\wedge 2)}$ for $d \ge 2$.
For $d \ge 4$, we additionally assume that
\begin{equation} \label{eq:Q_d-1}
\bignorm{ \abs y^{d-1} \Jsupmu(y) }_{ \frac d 3 \wedge 2},\,
\bignorm{ \abs y^{d-1} g(y) }_{ 2}
	\le M .
\end{equation}
\end{definition}

\begin{remark} \label{rmk:Q}
The bounds \eqref{eq:Q_d-1} are invoked only in the proof of
Lemma~\ref{lem:It_decay}.
When $g=\delta_0$, as is the case in our applications to random walk and self-avoiding walk, the hypotheses on $g$ are vacuous.
In our application to percolation, we verify the last inequality in \eqref{eq:g_moments}
by proving that $\norm{ \abs y ^{2}  g(y) }_{1}\le M$; this is stronger since $p_d \ge 1$ and the $L^p(\Zd)$ norms are decreasing in $p$.
\end{remark}

For any function $\Jsupmu$ satisfying \eqref{eq:Q_moments}--\eqref{eq:Q_infrared},
we define a vector $\eta \in \Rd$ and a symmetric matrix matrix $\Lam \in \R^{d \times d}$ as the first and second moments of $\Jsupmu$,
namely
\begin{equation} \label{eq:def_Lambda-intro}
\eta_j = \sum_{y\in \Zd} y_j \Jsupmu(y),
\qquad
\Lambda_{j l } = \sum_{y\in \Zd} y_j y_l \Jsupmu(y) .
\end{equation}
It is a simple consequence of
\eqref{eq:Q_moments} that
$\abs{ \eta_j}, \abs{ \Lam_{jl} } \le M$
and $k\cdot \Lambda k \le M \norm k_1^2$ for all $k \in \R^d$.
By Taylor expansion in  \eqref{eq:Q_infrared},
we also have $\half k \cdot \Lam k \ge \half \KIR \abs k^2$ for small $k$.
This extends to all $k\in \Rd$ by homogeneity.
Thus, the matrix $\Lam$ is positive-definite, and we have
\begin{equation} \label{eq:quadratic_form}
k \cdot \Lam k \asymp \abs k^2 \asymp k \cdot \Lam\inv k
	\qquad (k\in \Rd) ,
\end{equation}
with constants depending only on $d,M,\KIR$.

\begin{assumption} \label{ass:J}
Let $d \ge 1$.  We assume that $\GL$ obeys the generalised OZ equation
\begin{equation} \label{eq:OZeq}
    \GL = \g + J*\GL ,
\end{equation}
with $\Zd$-symmetric functions $J,\g \in L^1(\Zd)$.
We assume further that there are $M,\KIR,\zeta > 0$ such that
the exponential tilts of $J,\g$ obey
$(J\supmu,\g\supmu) \in \Qcal_{M,\KIR,\zeta}$ uniformly in $\mu\in \overline\Omega$,
and that
\begin{equation} \label{eq:gnonzero}
\hat \g\supmu (0)  > 0
	\qquad (\mu \in \overline \Omega) .
\end{equation}
\end{assumption}

The importance of the tilted functions $(J\supmu,\g\supmu)$
in Assumption~\ref{ass:J} can be appreciated from the fact that, upon tilting,
the convolution equation \eqref{eq:OZeq} becomes
\begin{equation}
    \GL\supmu = \g\supmu + J\supmu *\GL\supmu.
\end{equation}
Most of our analysis relates to this tilted convolution equation.

\subsubsection{Geometry of $\Omega$}
\label{sec:Omega-geom-intro}

The following proposition is proved in Section~\ref{sec:rate}.
Its ingredients $\mux,\eta_\hatx, \Lambda_\hatx$
appear in the crossover theorem. See Figure~\ref{fig:geom}.

\begin{proposition} \label{prop:geom-intro}
Let $d\ge 1$.
Suppose that Assumptions~\ref{ass:Omega} and~\ref{ass:J}
hold, but without requiring that $\g\supmu,J\supmu$ satisfy
\eqref{eq:Q_d-1}.
Then the set $\overline \Omega$ is strictly convex, and every $\mu \in \overline \Omega$ obeys $\abs \mu^2 \le M/\KIR$.
For each $\mu \in \overline\Omega$, we have
\begin{equation} \label{eq:J1-intro}
    \sum_{y\in \Zd} J \supmu(y) \le 1
    \quad \text{with equality if and only if $\mu\in\partial\Omega$.}
\end{equation}
For each nonzero $ x \in \Rd$,
there is a unique vector $\mux \in \del \Omega$, depending only on the direction $\hat x = x / \abs x$ of $x$, such that
\begin{equation}
\label{eq:optimal-mu-intro}
\mu_{\hat x}\cdot x = \max_{\mu \in \overline\Omega }\, \mu\cdot x.
\end{equation}
The vector $\eta_\hatx \in \Rd$ with components
\begin{equation} \label{eq:eta_cond-intro}
\eta_{\hatx,j}
= \sum_{y\in \Zd} y_j J\supmux(y)
\end{equation}
satisfies $\eta_{\hat x} = |\eta_{\hat x}|\hat x$.
Also,
with the matrix $\Lam_\hatx$ defined by
$\Lambda_{\hat x, j l } = \sum_{y\in \Zd} y_j y_l J\supmux(y)$,
we have
\begin{equation} \label{eq:crossover-bds-intro}
\mu_\hatx \cdot x  \asymp \mS  \abs x ,
	\qquad
|\eta_{\hat x}| \asymp \mS ,
	\qquad
k \cdot \Lam_\hatx k \asymp \abs k^2
	\quad (k\in \Rd) ,
\end{equation}
with constants depending only on $d,M,\KIR$.
\end{proposition}

\begin{figure}[ht]
\centering{
\includegraphics{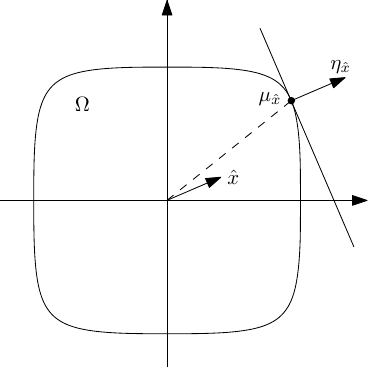}

\caption{
The optimal vector $\mu_{\hat x}$ is the point on $\partial\Omega$
(in the orthant containing $\hat x$)
whose tangent plane is orthogonal to $\hat x$.  Since $\partial\Omega$
is contained in the level set $\sum_{y\in\Zd}J\supmu(y)=1$, the gradient $\eta_{\hat x}$
of $\sum_{y\in\Zd}J\supmu(y)$ points in the same direction as $\hat x$.
}
\label{fig:geom}}
\end{figure}

For $x = e_1$, it follows from Lemma~\ref{lem:Omega_convex} that
$\mu_{e_1} = \mS  e_1$, so the mass $\mS$ obeys
$\hat J^{(m_S e_1)}(0)=1$.

\begin{corollary} \label{cor:norm}
Let $d\ge 1$.
Suppose that Assumptions~\ref{ass:Omega} and~\ref{ass:J}
hold, but without requiring that $\g\supmu,J\supmu$ satisfy
\eqref{eq:Q_d-1}.  Then
the function $|\cdot|_S :  \R^d \to [0,\infty)$ defined by
$\abs 0_S = 0$ and
\begin{equation}
\abs x_S = \frac{ \mu_\hatx \cdot x }{ \mS  }
	\qquad (x\ne 0)
\end{equation}
is a norm, and it satisfies $\|x\|_\infty \le |x|_S \le \|x\|_1$ for every $x \in \R^d$.
\end{corollary}

\begin{proof}
The function $|\cdot|_S$ is positively homogeneous by definition,
and by \eqref{eq:optimal-mu-intro} it satisfies $\abs{-x}_S = |x|_S$ (because $- \Omega = \Omega$).
It is also positive-definite by the first lower bound of \eqref{eq:crossover-bds-intro}.
For the triangle inequality, we take nonzero $x,y\in \Rd$ and use the optimality of $\mux$ and $\mu_{\hat y}$ to get
\begin{equation}
\mu_{\widehat{ x+y }} \cdot (x+y)
\le \mux \cdot x + \mu_{\hat y } \cdot y .
\end{equation}
Division by $\mS $ then gives $\abs{ x+y }_S \le \abs x_S + \abs y_S$.

For the bounds,
using H\"older's inequality we have
$\abs x_S   \le \mS \inv \norm{ \mux }_\infty \norm x_1$,
which is at most $\norm x_1$ by Lemma~\ref{lem:Omega_convex}.
For the lower bound,
we use the $\Zd$ symmetry to assume $x_1 = \norm x_\infty$ without loss of generality, and we take $\mu = \mS e_1 \in \overline \Omega$ in \eqref{eq:optimal-mu-intro} to get
$\abs x_S \ge e_1 \cdot x = x_1 = \norm x_\infty$.
This concludes the proof.
\end{proof}

In Section~\ref{sec:Wulff}, we show that the Wulff shape $\overline\Omega$
is the ball of radius $\mS$ for the norm dual to $|\cdot|_S$.

\subsubsection{The Crossover Theorem}
\label{sec:cross}

Let $d \ge 1$.
Given a vector $\eta \in \R^d$ and a symmetric positive-definite matrix $\Lambda \in \R^{d\times d}$,
the \emph{heat kernel} of the Brownian motion on $\Rd$ with drift $\eta\in\Rd$ and covariance $\Lambda$
is given by
\begin{align} \label{eq:def_rho}
\rho_t(x; \eta, \Lambda)
&= \frac 1 { \sqrt{\det \Lambda } } \frac{ 1 }{(2\pi t)^{d/2} }
	\exp \Bigl\{ - \frac 1 {2t} (x-t\eta) \cdot \Lambda\inv (x - t\eta) \Bigr\} .
\end{align}
Its \emph{Green function} is the integral
\begin{equation} \label{eq:def_C}
\bg(x; \eta, \Lambda) = \int_0^\infty \D t \, \rho_t(x; \eta, \Lambda)
	\qquad(x\in \Rd) .
\end{equation}
The Green function is finite for all $d \ge 1$ and all $x\in\R^d$ when $\eta \ne 0$, and it is finite for all $d>2$ and all $\eta\in \R^d$
when $x \ne 0$.

The following theorem is our main result.

\begin{theorem} [Crossover Theorem]
\label{thm:crossover}
Let $d\ge 1$.
Suppose Assumptions~\ref{ass:Omega} and~\ref{ass:J} hold.
For nonzero $x \in \Zd$, let $\mux,\eta_{\hat x},\Lam_{\hat x}$ be given by Proposition~\ref{prop:geom-intro}.
If $d > 2$ then
\begin{equation} \label{eq:main_asymp-intro}
\GL(x)
= \bg(x; \eta_{\hat x} , \Lambda_{\hat x}) e^{-\mS  \abs x_S} \Bigl[ \hat \g \supmux (0) + O\Bigl(\frac 1 {\abs x^\eps}\Bigr) \Bigr]
\end{equation}
as $\abs x \to \infty$, with some $\eps > 0$
and with the constant depending only on $d, \zeta, M, \KIR$.
If $d \le 2$ then
\eqref{eq:main_asymp-intro} still holds when the limit is taken subject to the condition that $\mS  \abs x \ge \so$ for some $\so > 0$, and then
the constant also depends on $\so$.
\end{theorem}

Theorem~\ref{thm:crossover}
reduces the asymptotic behaviour of $S(x)$ to that of
$\C(x; \eta_{\hat x} , \Lambda_{\hat x})$.
The latter is studied in Appendix~\ref{app:Brown}.
In particular, it
decays polynomially as $\abs x \to \infty$.
Since $ \hat \g \supmux (0) > 0$ by
Assumption~\ref{ass:J}, it follows from Theorem~\ref{thm:crossover} that
\begin{equation} \label{eq:decay_rate}
\lim_{n \to \infty} \frac {- \log \GL (n x) } n
= \mS  \abs x_S = \mux \cdot x
	\qquad (0 \ne x \in \Zd) .
\end{equation}
Thus $S$ decays exponentially in the $x$ direction with rate $\mS  \abs x_S$. This rate/norm can be computed via the variational problem on $\overline \Omega$ in Proposition~\ref{prop:geom-intro}.
For example, as we show in Section~\ref{sec:nnRW},
for the subcritical Green function $\GL_z^{\rm SRW}$ for nearest-neighbour simple random walk with parameter $z\in (0,1)$,
our method gives an effortless
way to compute its norm
\begin{equation}
|x|_z  =  \frac{1}{\mz }\sum_{j=1}^d x_j \arcsinh (u_z(x)x),
\end{equation}
with $u_z(x) \ge 0$ the unique solution to the equation
$\sum_{j=1}^d  \sqrt{1+ u_z(x)^2 x_j^2  } = d/z$.
A more indirect computation was given in \cite{MS22}.

The fact that the exponential rate of decay in
\eqref{eq:main_asymp-intro} involves the optimal tilt $\mux$ from \eqref{eq:optimal-mu-intro}
can be understood intuitively as follows.
We exponentially tilt $S$ in all possible directions such that $\chi\supmu$ remains finite.
Since the direction $x$ ``feels'' rate $\mu \cdot x$ in such a tilting,
the heaviest possible tilt corresponds to the decay rate of $\GL$ in that direction.

In much of the literature, proofs of OZ decay involve analysing
typical configurations (of self-avoiding walks, percolation clusters, etc.) which realise connections beyond the correlation length.  These
configurations are highly stretched-out compared to critical connections, and display a renewal structure (e.g., \cite{Ioff98,CI02,CIV08,DM26}).
In Theorem~\ref{thm:crossover}, this feature presents itself in the Brownian Green function.  Its drift is aligned with $x$, which causes
typical configurations to be stretched towards $x$.

\begin{remark}
The fact that the decay rate \eqref{eq:decay_rate} defines a norm is usually proven using an FKG-type inequality
\begin{equation}
    \GL(x) \ge \const\, \GL(y)\GL(x-y)
	\qquad( x,y\in \Zd)
\end{equation}
for some positive constant (e.g., \cite{MS22,Pfis91}, an exception is \cite{Ott22}).
We do not need this assumption.
\end{remark}

To interpret Theorem~\ref{thm:crossover}, we apply standard
asymptotic properties of the Brownian Green function.
For odd dimensions, the Brownian Green function can be expressed in terms of elementary functions.
This is simplest for $d=1$ and $d=3$, where for all $x \neq 0$, and with $\eta \neq 0$ for $d=1$, we have
\begin{alignat}2
\bg(x; \norm\eta_2  \hat x, \Lambda)
&=
\frac{ 1 } { \norm\eta_2 }
    && (d=1),
    \\
    \bg(x;  \norm\eta_2 \hat x,  \Lambda)
&= \frac { 1} { 2\pi  \sqrt{\det \Lambda } }
\frac{ 1 } {\sqrt{ \hat x \cdot \Lam\inv \hat x} } \frac{1}{\norm x_2}
    \qquad &&(d=3)
\end{alignat}
(see Lemma~\ref{lem:CBess}).
The  following two corollaries of Theorem~\ref{thm:crossover}
are proved in Appendix~\ref{app:Brown}, by inserting the asymptotic behaviour of $\bg$ into \eqref{eq:main_asymp-intro} to yield the asymptotic behaviour of $S(x)$.

\begin{corollary} [Ornstein--Zernike decay]
\label{cor:OZ}
Let $d\ge 1$ and
suppose that Assumptions~\ref{ass:Omega} and~\ref{ass:J} hold.
For nonzero $x \in \Zd$, let $\mux,\eta_{\hat x},\Lam_{\hat x}$ be given by Proposition~\ref{prop:geom-intro}.
Then as $\abs x\to \infty$, we have
\begin{equation}
S(x) =
	\frac 1 { (2\pi)^{(d-1)/2} \sqrt{\det\Lambda_{\hat x}} }
	\frac{1}{ (\hat x \cdot \Lam_{\hat x}\inv \hat x)^{1/2} }
	\frac{ \norm {\eta_{\hat x}}_2^{(d-3)/2}} { \norm  x_2^{(d-1)/2}}
	e^{-\mS  \abs x_S}
	[ \hat \g \supmux (0) + o(1)] ,
\end{equation}
with the $o(1)$ depending only on $d,M,\KIR,\zeta,\mS$.
\end{corollary}

\begin{corollary} \label{cor:crossover_weak}
Let $d\ge 1$  and suppose that Assumptions~\ref{ass:Omega} and~\ref{ass:J} hold.
Suppose also that $\hat \g\supmu (0)  \ge M\inv > 0$ for all $\mu \in \del \Omega$.
\begin{enumerate}
\item[(i)]
For $d >2$, there exists $R > 0$ such that
\begin{equation} \label{eq:GLasy}
\GL(x) \asymp
	\frac{  \max\{ 1 , \mS  \abs x_S \} ^{ (d-3)/2} } { \abs x_S^{d-2} }
	e^{- \mS  \abs x_S }
	\qquad (\abs x_S \ge R)
\end{equation}
with constants that depend only on $d,M,\KIR,\zeta$.
\item[(ii)]
For $d\le 2$ and $\so > 0$, there exists $R > 0$ such that
\begin{equation}
\GL(x) \asymp
	\frac{  1  } { \mS   ^{ (3-d)/2}  \abs x_S^{ (d-1)/2} }
	e^{- \mS  \abs x_S }
	\qquad (\abs x_S \ge R \vee \frac{ \so }{ \mS  })
\end{equation}
with constants that depend only on $d,M,\KIR,\zeta,\so$.
\end{enumerate}
\end{corollary}

We call Theorem~\ref{thm:crossover} the Crossover Theorem because
the uniform estimates in \eqref{eq:GLasy} encompass both the critical decay ($\max=1$)
and the OZ
($\max=\mS  \abs x_S$) decay.
However, we stress that the explicit nature of the Brownian Green function gives an exact asymptotic formula for $\GL(x)$ in all regimes.

Applications of the Crossover Theorem are presented
for random walk in Section~\ref{sec:RW},
for self-avoiding walk in Section~\ref{sec:SAW},
and for percolation in Section~\ref{sec:perc}.

\subsection{The noncentred Gaussian Lemma}

The next theorem, which we call the noncentred Gaussian Lemma, extends Hara's Gaussian Lemma \cite{Hara08} from the case $\eta=0$ to arbitrary drift aligned with $x$.  As we will see, the Crossover Theorem
quickly follows from Theorem~\ref{thm:deconv-intro}.

\begin{theorem} \label{thm:deconv-intro}
Let $d \ge 1$ and
let $M,\KIR,\zeta>0$ be given.
Suppose $(\Jsupmu,g)\in \Qcal_{M,\KIR,\zeta}$.
Let $\eta$ and $\Lambda$ be given by \eqref{eq:def_Lambda-intro},
and let $x\in\Zd$ be a nonzero vector in the same direction as $\eta$,
\ie, $\eta = \abs \eta \hat x$ ($\eta=0$ is permitted here).
If $d \le 2$, we further assume $\abs x \abs \eta \ge \so$ for some $\so > 0$.
Then, as $\abs x \to \infty$,
\begin{equation} \label{eq:Q_int-intro}
\GQ (x) :=
\int_0^\infty \D t    \int_\Td \frac{\D k }{(2\pi)^d}
	e^{ik\cdot x} \hat g(k)  e^{-t [ \hat \Jsupmu(0) - \hat \Jsupmu (k) ] }
=  \bg(x; \eta, \Lam) \Bigl[\hat{g}(0) + O\Bigl(\frac 1 {\abs x^\eps}\Bigr) \Bigr]
\end{equation}
with some $\eps > 0$
and with the constant depending only on $d, \zeta, M, \KIR, \so$.
\end{theorem}

\begin{remark} \label{rmk:Fubini}
In dimensions $d > 2$,
the double integral in \eqref{eq:Q_int-intro} is absolutely integrable
by the infrared bound \eqref{eq:Q_infrared} and
the identity $z\inv = \int_0^\infty e^{-t z} \D t$ for $\Re(z) > 0$.
Fubini's theorem then implies that
\begin{equation}
\label{eq:GFub}
\GQ (x)
= \int_\Td \frac{\D k }{(2\pi)^d}
	\frac { \hat g(k)  } { \hat \Jsupmu(0) - \hat \Jsupmu (k) } e^{ik\cdot x} .
\end{equation}
\end{remark}

It is easy to guess that \eqref{eq:Q_int-intro} should hold, as follows.
We expect the behaviour of the integral for large $x$ to be determined by small $k$.  Accordingly, we make the second-order Taylor approximation
\begin{equation}
    \hat \Jsupmu(0) - \hat \Jsupmu (k)
    \approx
    ik\cdot \eta + \frac 12 k\cdot\Lam k,
\end{equation}
we approximate $\hat g(k)$ by $\hat g(0)$, and then we extend the integral over $k$ from $\T^d$ to $\R^d$.
Recall from \eqref{eq:def_rho} the definition of $\rho_t(k;\eta,\Lam)$, the heat kernel of the Brownian motion on $\Rd$ with drift $\eta$ and covariance $\Lambda$.
With the above approximations, the left-hand side of \eqref{eq:Q_int-intro} becomes
\begin{align}
    \hat g(0)
    \int_0^\infty \D t    \int_\Rd \frac{\D k }{(2\pi)^d}
	e^{ik\cdot x}   e^{-t [ ik\cdot \eta + \frac 12 k\cdot\Lam k ] }
    &=
    \hat g(0)
    \int_0^\infty \D t    \int_\Rd \frac{\D k }{(2\pi)^d}
	e^{ik\cdot x}
    \hat \rho_t(k;\eta,\Lam)
    \nnb & =
    \hat g(0)
    \int_0^\infty \D t\,
    \rho_t(x;\eta,\Lam)
    =
    \hat g(0) \bg(x;\eta,\Lam).
\end{align}
The proof of Theorem~\ref{thm:deconv-intro}, in Section~\ref{sec:tilted_gaussian}, involves a careful justification of the above uncontrolled approximations.

When $g(x)=\delta_{0,x}$ and when $Q$ is a probability distribution on $\Zd$,
$\GQ (x)$ is the Green function of the
random walk with transition probability $Q$.
This random walk has drift $\eta$
and covariance $\Lam$.
Theorem~\ref{thm:deconv-intro} shows that
when we evaluate the Green function at a large multiple of the drift,
asymptotically we see the Green function of Brownian motion with the same
drift and covariance.

Theorem~\ref{thm:deconv-intro} is related to a classical theorem
of Ney and Spitzer \cite[Theorem~2.2]{NS66}.
Important improvements
present in Theorem~\ref{thm:deconv-intro}, compared to \cite[Theorem~2.2]{NS66},
are the facts that we do \emph{not}
assume that $Q(x) \ge 0$, and that the convergence in \eqref{eq:Q_int-intro}
is uniform as $\eta \to 0$.  The former is crucial for our applications to
self-avoiding walk and percolation, and the latter permits control of the crossover to a critical point.

\subsection{Proof of the Crossover Theorem~\ref{thm:crossover}}
\label{sec:pf_crossover}

The proof of Theorem~\ref{thm:crossover} is easier in dimensions $d> 2$ because we can apply Fubini's theorem as in Remark~\ref{rmk:Fubini}.
We prove this case here and defer the proof for $d\le 2$ to Appendix~\ref{app:1-2D}.

\begin{proof}[Proof of Theorem~\ref{thm:crossover} for $d> 2$]
Suppose Assumptions~\ref{ass:Omega} and \ref{ass:J} hold.
We want to compute the asymptotic behaviour of $S(x)$ as $\abs x \to \infty$, by applying Theorem~\ref{thm:deconv-intro} with $\Jsupmu = J \supmux$ and $g = \g \supmux$, with $\mux$ given by Proposition~\ref{prop:geom-intro}.
The importance of tilting by $\mux$ is that then $\eta_{\hat x}$ is aligned with $x$, by Proposition~\ref{prop:geom-intro}, as is required for application of Theorem~\ref{thm:deconv-intro}.

To begin,
we let $\theta \in [0, 1)$ and multiply the generalised OZ
equation $\GL =\g +J*\GL$ from Assumption~\ref{ass:J} by $e^{\theta \mux \cdot x} = e^{\theta \mux \cdot (x-y)} e^{\theta \mux \cdot y}$.
This gives the tilted equation
\begin{equation}
S(x) e^{\theta \mux \cdot x}
= \g(x) e^{ \theta \mux \cdot x }
	+ \sum_{y\in \Zd} J(x-y) e^{\theta \mux \cdot (x-y)}
			S(y) e^{\theta \mux \cdot (y)} ,
\end{equation}
which can be written as $S \supk{ \theta \mux } = \g \supk{ \theta \mux } + J \supk{ \theta \mux } * S \supk{ \theta \mux }$.
Since $\theta < 1$, since $\mux \in \del \Omega$, and since $\Omega$ is strictly convex by Proposition~\ref{prop:geom-intro},
we know ${ \theta \mux } \in \Omega$.  This tells us that
$S \supk{ \theta \mux } \in L^1(\Zd)$, by the definition of $\Omega$.
We also know $J \supk{ \theta \mux }, \g \supk{ \theta \mux } \in L^1(\Zd)$ by Assumption~\ref{ass:J}.
Therefore, we can solve the convolution equation using the (inverse) Fourier transform, as
\begin{equation} \label{eq:S_int_pf}
S(x) e^{\theta \mu_{\hat x}\cdot x}
= \int_\Td \frac{\D k }{(2\pi)^d}
	\frac { \hat \g\supk{\theta \mux} (k)  } { 1 - \hat J\supk{\theta \mux} (k) }   e^{ik\cdot x} .
\end{equation}
We want to take the limit $\theta \to 1$
using dominated convergence.
The numerator obeys $\abs{\hat \g\supk{\theta \mux} (k) }
\le \norm{ \g\supk{\theta \mux} (y) }_1 \le M$.
Also, since $\hat J\supk{\theta \mux} (0)$ is real and less than $1$
by Proposition~\ref{prop:geom-intro}, the denominator obeys
\begin{equation}
\abs{ 1 - \hat J\supk{\theta \mux} (k) }
\ge
    \Re[ 1 - \hat J\supk{\theta \mux}(k) ]
\ge
    \Re[ \hat J\supk{\theta \mux} (0) - \hat J\supk{\theta \mux}(k) ]
\ge \KIR \abs k ^2  .
\end{equation}
Since $d > 2$, this allows us to take the limit $\theta \to 1$
 in \eqref{eq:S_int_pf} using dominated convergence,
to get \eqref{eq:S_int_pf} with $\theta = 1$.
Then, since $\hat J\supmux(0) = 1$
and $\eta_\hatx = \abs{ \eta_\hatx} \hat x$ by Proposition~\ref{prop:geom-intro},
it follows from Theorem~\ref{thm:deconv-intro} and Remark~\ref{rmk:Fubini} that
\begin{equation}
S(x) e^{\mu_{\hat x}\cdot x}
= \int_\Td \frac{\D k }{(2\pi)^d}
	\frac { \hat \g\supk{\mux} (k)  } { 1 - \hat J\supk{\mux} (k) }
= \bg(x; \eta_\hatx, \Lam_\hatx) \Bigl[\hat \g\supk{\mux} (0)  + O\Bigl(\frac 1 {\abs x^\eps}\Bigr) \Bigr]
\end{equation}
as $\abs x\to \infty$.
This gives the desired result after multiplication by $e^{-\mu_{\hat x}\cdot x} = e^{- \mS  \abs x_S}$ (which holds by Corollary~\ref{cor:norm}).
\end{proof}

\subsection{Guide to the paper}

In Section~\ref{sec:introduction}, we have been concerned with a single OZ equation $\GL=\g+J*\GL$.  For applications, we instead consider a family of OZ equations $\GL_z=\g_z+J_z*\GL_z$ indexed by a parameter $z \in [0,z_c]$.  The point $z_c$ is a critical point in the sense that
the mass $m_{z}$ of $S_z$ goes to $0$ as $z \to z_c$.
In Section~\ref{sec:critical},
under this vanishing mass assumption, and assuming bounds on the
tilted functions $(J_z\supmu,\g_z\supmu)$ that hold uniformly in $\mu\in \overline\Omega_z$ and in $z$ near $z_c$, we analyse the asymptotic behaviour of the norm $|\cdot|_z$,
the susceptibility $\chi(z)$, and the correlation length $\xi_\phi(z)$ of any order $\phi >0$, in the limit $z \to z_c$.  In particular, we identify a universal limit for $m_{z}\xi_\phi(z)$ as $z \to z_c$.

Section~\ref{sec:RW} is concerned with applications of our results to Green functions for random walks on $\Zd$.
All our results apply to irreducible random walks whose kernel decays super-exponentially.
We also study the massive ($z\to 0$) limit of the norm $|\cdot|_z$ when the kernel is finitely supported.
The norm is computed explicitly for two examples.
We also give an example of a kernel without super-exponential decay, whose Green function does not have OZ decay when $z$ is small.

In Section~\ref{sec:rate},
we prove the important Proposition~\ref{prop:geom-intro}, concerning the geometry and variational problem associated to the set $\overline \Omega$. We also justify the terminology ``Wulff shape'' for $\overline\Omega$.

In Section~\ref{sec:tilted_gaussian}, we prove the noncentred Gaussian Lemma (Theorem~\ref{thm:deconv-intro}).
The proof is based on Fourier analysis and is independent of the rest of the paper.

Section~\ref{sec:SAW} contains the application of our results to the nearest-neighbour self-avoiding walk on $\Zd$ in dimensions $d \ge 5$.
Section~\ref{sec:perc} contains the application to  nearest-neighbour Bernoulli bond percolation on $\Zd$ in dimensions $d \ge 15$.
For both models, we prove that all results of Sections~\ref{sec:introduction}--\ref{sec:critical} apply
in the vicinity of the critical point.
In particular, we obtain detailed asymptotic behaviour of the two-point function, including the crossover from OZ to critical decay.

Finally, there are two appendices.
Appendix~\ref{app:Brown} provides elementary properties of the Brownian Green function.
Appendix~\ref{app:1-2D} gives the extra argument needed to prove the $d\le 2$ case of Theorem~\ref{thm:crossover};
it uses estimates obtained in Section~\ref{sec:tilted_gaussian}.

\section{Critical limits}
\label{sec:critical}

In our applications of the crossover theorem, we consider models which are parametrised by $z \in [0,z_c]$ with a critical point $z_c$.
The OZ equation then becomes
\begin{equation}
\label{eq:OZz}
    S_z = \g_z + J_z*\GL_z.
\end{equation}
We use a subscript $z$ to denote quantities associated with $S_z$. For example, $\Omega_z$ denotes the set $\Omega$ defined from $S_z$, and
$\mu_{\hat x, z } \in \del \Omega_z$ denotes the vector produced by Proposition~\ref{prop:geom-intro}.
We also write $\mz = m_{S_z}$ and $|\cdot|_z = |\cdot|_{S_z}$.
We assume that $m_z>0$ for $z<z_c$ and $\lim_{z\to z_c}m_z=0$; the
latter is what makes $z_c$ \emph{critical}.
In this section, we consider general results which apply when Assumptions~\ref{ass:Omega} and \ref{ass:J} hold uniformly near the critical point.

\subsection{Critical limit of the norm}

The first theorem says that the limit of the norm $\abs x_z$ is the Euclidean norm, as $z$ approaches the critical value $z_c$.
Its proof uses only Proposition~\ref{prop:geom-intro} and not the Crossover Theorem.
This greatly generalises a result of \cite{MS22}, which proved
Theorem~\ref{thm:critical_limit} for the specific example of the Green function for the nearest-neighbour random walk.

For $z \in [0,z_c)$, the \emph{susceptibility} is defined by
$    \chi(z) = \sum_{x\in\Zd}S_z(x).$
We also define $\sigma_z$ by
\begin{equation}
    \sigma_z^2 = \sum_{x\in\Zd}\norm x_2^2 J_z(x);
\end{equation}
note that the right-hand side is positive when the infrared bound is satisfied.

\begin{theorem} \label{thm:critical_limit}
Let $d\ge 1$ and $\delta>0$.
Suppose that Assumptions~\ref{ass:Omega} and~\ref{ass:J} hold for $S_z$ for each $z\in [z_c-\delta,z_c)$, suppose that $m_z \to 0$ as $z \to z_c$,
and suppose that there are constants $M,\KIR,\zeta$ such that
$(J_z\supmu,\g_z\supmu) \in \Qcal_{M,\KIR,\zeta}$ uniformly in $\mu\in \overline\Omega_z$ and in $z \in [z_c-\delta,z_c)$.
Then, in the limit $z \to z_c$,
\begin{equation} \label{eq:crit_limit}
\abs x_{z}
= \norm x_2 \bigl[ 1 + O(\mz ^{ \zeta \wedge 2 } ) \bigr]
\end{equation}
uniformly in $x\in \Rd$.
\end{theorem}

Theorem~\ref{thm:critical_limit} is complemented by
Theorem~\ref{thm:massive_limit} for the $z\to 0$ limit, for the case of
random walk with finitely supported step distribution.

\begin{proof}
Since the hypothesis is weaker when $\zeta$ is smaller, we assume $\zeta \le 2$.
By \eqref{eq:Q_moments}--\eqref{eq:Q_infrared} with $Q = J_z$, we have
$2d\KIR \le \sigma_z^2 \le M$.
We may assume without loss of generality that $x\ne 0$, since both norms in \eqref{eq:crit_limit} equal zero when $x=0$.
For nonzero $x$, by Corollary~\ref{cor:norm} the norm is given by
\begin{equation}
\abs x_{z}
= \frac{ \mu_{\hatx, z} \cdot x } { m_z } .
\end{equation}

Let $x \ne 0$.
Since $\muxz \in \del \Omega_z$, we know that $\norm{ \muxz }_2 \asymp \mz$.
We first compute its asymptotic behaviour.
By summing the OZ equation \eqref{eq:OZz}
and by using $\hat J_z \supmuxz(0)  = 1$, we have
\begin{equation}
\frac{ \hat \g_z(0) } { \chi(z) }
= 1 - \hat J_z(0)
= \hat J_z \supmuxz(0) - \hat J_z(0)
= \sum_{y\in \Zd} J_z(y) ( \cosh(\muxz \cdot y) - 1 ) .
\end{equation}
Since $\cosh t = 1 + \frac 12 t^2 +O(|t|^{2+\zeta}\cosh t)$ when $\zeta\le2$, we have
\begin{equation}
\label{eq:muxz_asymp0}
\frac{ \hat \g_z(0) } { \chi(z) }
= \half \sum_{y\in \Zd} J_z(y)  (\muxz \cdot y)^2
	+ O( \mz^{2+\zeta} ) \sum_{y\in \Zd} \abs y^{2+\zeta} \abs{ J_z \supmuxz(y) }
= \frac{ \sigma_z^2 }{ 2d } \norm{ \muxz }_2^2
	+ O( \mz^{2+\zeta} ) ,
\end{equation}
where in the last equality we used the bound on the $(2+\zeta)^{\rm th}$ moment
of $ J_z \supmuxz(y)$ given by \eqref{eq:Q_moments}.
The same formula holds when $\hat x = e_1$, for which $\mu_{e_1, z} = m_z e_1$ (due to \eqref{eq:Omega_incl}).
Comparing the two formulas then shows
\begin{equation} \label{eq:muxz_asymp}
\norm{ \muxz }_2^2
= \frac{ 2d \hat \g_z(0) } { \sigma_z^2 \chi(z) } + O(m_z^{2+\zeta})
= \norm{ \mu_{e_1, z} }_2^2  + O(m_z^{2+\zeta})
= m_z^2 [ 1 + O(m_z^\zeta) ] ,
\end{equation}
uniformly in $x\ne 0$.

Next, using the fact that $\sinh t = t + O(\abs t^{1+\zeta} \cosh t)$ when $\zeta \le 2$, we similarly get
\begin{equation} \label{eq:eta_asymp}
(\eta_{\hat x,z})_j
= \sum_{y\in \Zd} y_j J_z(y) \sinh(\muxz \cdot y)
= \sum_{y\in \Zd} y_j J_z(y) (\muxz \cdot y)
	+ O(m_z^{1+\zeta})
= \frac{ \sigma_z^2 } d (\muxz)_j
	+ O(m_z^{1+\zeta})
\end{equation}
for all components $j$.
Since $\eta_{\hatx,z}$ is in the direction of $\hatx$ by Proposition~\ref{prop:geom-intro}, the above gives
\begin{equation}
\mu_{\hatx, z} \cdot x
= { \norm x_2 } \frac { \mu_{\hatx, z} \cdot \eta_{\hatx, z}  }
	{ \norm{ \eta_{\hatx, z} }_2  }
= { \norm x_2 } \frac
	{ \frac{ \sigma_z^2 } d \norm{ \muxz }_2^2 + O(\mz^{2+\zeta}) }
	{ \frac{ \sigma_z^2 } d \norm{ \muxz }_2 + O(\mz^{1+\zeta}) }
= \norm x_2 \norm { \muxz }_2 [ 1 + O(\mz^\zeta) ] .
\end{equation}
The desired result then follows from \eqref{eq:muxz_asymp}.
\end{proof}

From the proof of Theorem~\ref{thm:critical_limit}, we extract the following corollary.

\begin{corollary}
\label{cor:chim}
Let $d\ge 1$ and $\delta>0$.
Suppose that Assumptions~\ref{ass:Omega} and~\ref{ass:J} hold for $S_z$ for each $z\in [z_c-\delta,z_c)$, suppose that $m_z \to 0$ as $z \to z_c$,
and suppose that there are constants $M,\KIR,\zeta$ such that
$(J_z\supmu,\g_z\supmu) \in \Qcal_{M,\KIR,\zeta}$ uniformly in $\mu\in \overline\Omega_z$ and in $z \in [z_c-\delta,z_c)$.
Then $\chi(z) \to \infty$ as $z \to z_c$.  More precisely,
\begin{equation}
\label{eq:chim2}
    \chi(z)
    =
    \frac{2d\hat h_{z}(0)}{\sigma_{z}^2}
      \frac{1}{m_z^{2}}
    \big[ 1+O(m_z^{\zeta\wedge2}) \big] .
\end{equation}
\end{corollary}

\begin{proof}
The combination of \eqref{eq:muxz_asymp0} and \eqref{eq:muxz_asymp}, together with their assumption that $\zeta \le 2$,  gives
\begin{equation}
    \frac{\hat h_z(0)}{\chi(z)} = \frac{\sigma_z^2}{2d} m_z^2
    \bigl[ 1+O(m_z^{\zeta\wedge2}) \bigr] .
\end{equation}
The desired result then follows by solving for $\chi(z)$.
\end{proof}

\subsection{Critical decay and correlation length}

We now consider the (massive) critical limit where $\mz \abs x_z \to s \in [0,\infty)$.
The limit involves the (massive) continuum Green function, which is defined
for $a \ge 0$ if $d>2$, and for $a>0$ for $d \le 2$, by
\begin{equation} \label{eq:def_mG}
    \mG_a(x) = \int_0^\infty \D t\, \rho_t(x;0,\mathrm{Id}) e^{-\half ta^2}
    =
    \int_0^\infty \frac{ \D t}{(2\pi t)^{d/2} } e^{-\frac{1}{2t}\|x\|_2^2} e^{-\half ta^2} .
\end{equation}

\begin{theorem}  [Critical decay]
\label{thm:critical_decay}
Let $d\ge 1$ and $\delta>0$.
Suppose that Assumptions~\ref{ass:Omega} and~\ref{ass:J} hold for $S_z$ for each $z\in [z_c-\delta,z_c)$, suppose that $m_z \to 0$ as $z \to z_c$,
and suppose that there are constants $M,\KIR,\zeta$ such that
$(J_z\supmu,\g_z\supmu) \in \Qcal_{M,\KIR,\zeta}$ uniformly in $\mu\in \overline\Omega_z$ and in $z \in [z_c-\delta,z_c)$.
If $z \to z_c$ and $\abs x \to \infty$ in such a way that $m_z \abs x_z \to s \in [0,\infty)$, with $s > 0$ if $d \le 2$, then we have
\begin{equation}
S_z(x)
= \frac { d } {  \sigma_z^2 \norm x_2^{d-2} } \mG_{ s }(\hat x)
	[ \hat \g_z (0) + o(1)] .
\end{equation}
\end{theorem}

\begin{remark}
One goal of the renormalisation group method is to represent the two-point function
of an interacting model as the two-point function of a non-interacting model
at renormalised parameters.  An instance of this for $4$-dimensional models
can be seen in \cite[(2.6) and (2.8)]{BSTW-clp}.  Theorem~\ref{thm:critical_decay} is another
example:
after application of the scaling relation \eqref{eq:Gscaling} for $\mG$,
it expresses $\GL_z(x)$ in terms of a multiple of the free Green function
$\mG_{m_z}(x)$ at a renormalised mass $m_z$.
\end{remark}

\begin{proof}[Proof  of Theorem~\ref{thm:critical_decay}]
For nonzero $x \in \Zd$, let $\mu_{\hat x,z},\eta_{\hat x,z},\Lam_{\hat x,z}$ be given by Proposition~\ref{prop:geom-intro}.
By Theorem~\ref{thm:crossover},
\begin{equation}  \label{eq:SG1}
\GL_z(x)
= \bg(x; \eta_{\hat x,z} , \Lambda_{\hat x,z}) e^{-\mz  \abs x_z} [ \hat \g \supmuxz_z (0) + o_x(1)]
	\qquad (\abs x \to \infty) .
\end{equation}
We define $a>0$ by $a^2 = \eta_{\hatx, z} \cdot \Lam_{\hatx, z}\inv \eta _{\hatx, z }$.
By Lemma~\ref{lem:C=G},
\begin{equation}
\label{eq:SG2}
    \bg(x; \eta_{\hat x,z} , \Lambda_{\hat x,z})
    =
    \frac {1} { \norm x_2^{d-2} \sqrt{\det \Lam_\hatxz}}
	\mG_{a\norm x_2}( \Lam_\hatxz^{-1/2} \hat x)
	\exp\bigl\{ x\cdot \Lam_\hatxz\inv \eta_\hatxz \bigr\}.
\end{equation}
Since $\Lam_\hatxz \sim  (\sigma_z^2 / d ) \mathrm{Id}$ as $z\to z_c$ (uniformly in $x$),
we have $\Lam_\hatxz^{-1/2} x \sim ( \sqrt d /\sigma_z ) x$.
From \eqref{eq:eta_asymp} and \eqref{eq:muxz_asymp}, we have
\begin{equation} \label{eq:SG3}
a
\sim \frac{ \sqrt d}{ \sigma_z } \norm{ \eta_{\hatx, z} }_2
\sim \frac{ \sigma_z }{ \sqrt d} \norm{ \mu_{\hatx, z} }_2
\sim \frac{ \sigma_z \mz }{ \sqrt d} .
\end{equation}
Since $m_z \abs x_z \to s$ by assumption and $\abs x_z \to \norm x_2$ by Theorem~\ref{thm:critical_limit}, we
thus have $a \norm x_2 \sim \sigma_z s / \sqrt d$ and
\begin{equation}
x \cdot \Lam_\hatxz\inv \eta_\hatxz
= (\hat x \cdot \Lam_{\hatx, z}\inv \hat x)
	\norm x_2 \norm{ \eta_\hatxz} _2
\sim \mz \norm x_2
\to s .
\end{equation}
We combine these with $\hat h_z \supmuxz(0) \sim \hat h_z (0)$ and
\eqref{eq:SG1}--\eqref{eq:SG2}, to obtain
\begin{equation}
S_z(x)
= \frac {1} { \norm x_2^{d-2} \sqrt{\det \Lam_\hatxz}}
	\mG_{a\norm x_2}( \Lam_\hatxz^{-1/2} \hat x)
	\exp\{ s - s\}
	[ \hat \g_z (0) + o(1)]  .
\end{equation}
By the joint continuity of $\mG_a(x)$ in $a$ and $x$ (apparent from
the right-hand side of \eqref{eq:def_mG}),
and by the scaling identity \eqref{eq:Gscaling} for $\mG$, we therefore have
\begin{align}
S_z(x)
&=  \frac {1} { \norm x_2^{d-2}  (\sigma_z / \sqrt d)^d }
	\mG_{ \sigma_z s / \sqrt d}\biggl( \frac { \sqrt d }{ \sigma_z} \hat x\biggr)
	 [ \hat \g_z (0) + o(1)] \nl
&=  \frac { d } {  \sigma_z^2 \norm x_2^{d-2} } \mG_{ s }(\hat x)
	[ \hat \g_z (0) + o(1)] ,
\end{align}
as desired.
\end{proof}

The asymptotic formula \eqref{eq:chim2} for the susceptibility
is greatly extended in the next theorem.  For its statement, we define the \emph{correlation length $\xi_\phi(z)$ of order} $\phi > 0$ by
\begin{equation}
\label{eq:xidef}
    \xi_\phi(z)  =
    \biggl( \frac{1}{\chi(z)} \sum_{x\in\Zd}\|x\|_2^\phi S_z(x)
    \biggr)^{1/\phi}.
\end{equation}
For $\phi\ge 0$, we define a positive constant $A_\phi$ by
\begin{equation} \label{eq:Aphidef}
A_\phi     =
\int_{\Rd} \|y\|_2^\phi  \mG_{1}(y) \D y
=  	2^{\phi + 1}
	\frac{\Gamma( \frac{\phi+2}{2})\Gamma( \frac{\phi+d}{2} )}{\Gamma(\frac d2)}
\end{equation}
(the last equality is justified in Lemma~\ref{lem:G1}).
In particular, $A_0 = 2$.

\begin{theorem} \label{thm:xiphi}
Let $d\ge 1$ and $\delta >0$.
Suppose that Assumptions~\ref{ass:Omega} and~\ref{ass:J} hold for $S_z$ for each $z\in [z_c-\delta,z_c)$, suppose that $m_z \to 0$ as $z \to z_c$,
and suppose that there are constants $M,\KIR,\zeta$ such that
$(J_z\supmu,\g_z\supmu) \in \Qcal_{M,\KIR,\zeta}$
and $\hat \g_z\supmu (0)  \ge M\inv > 0$
uniformly in $\mu\in \overline\Omega_z$ and in $z \in [z_c-\delta,z_c)$.
Suppose further that $\lim_{z\to z_c}\sigma_z = \sigma_{z_c}$
and $\lim_{z\to z_c}\hat h_z(0) = \hat h_{z_c}(0)$.
Then for every $\phi \ge 0$,
as $z \to z_c$ we have
\begin{equation} \label{eq:xi-unnormalised}
\sum_{x\in \Zd} \| x\|_2^\phi S_z(x)
\sim A_\phi \frac { d \hat h_{z_c}(0) } {  \sigma_{z_c}^2   }
\frac {1} {    \mz^{2+\phi} }
\end{equation}
and
\begin{equation}\label{eq:xiphiasy}
\xi_\phi(z)
\sim \Bigl(\frac{A_\phi}{A_0}\Bigr)^{1/\phi} \frac 1 \mz
= 2 \biggl(\frac{\Gamma( \frac{\phi+2}{2})\Gamma( \frac{\phi+d}{2} )}{\Gamma(\frac d2)}\biggr)^{1/\phi}
	\frac 1 \mz .
\end{equation}
\end{theorem}

In terms of \emph{the correlation length} $\xi(z)=m_z\inv$,
\eqref{eq:xiphiasy} shows that for every $\phi>0$
the ratio $\xi_\phi/\xi$ converges to a $d$-dependent constant:
\begin{equation}
\label{eq:xiphi-xi-ratio}
    \lim_{z\to z_c}
    \frac{\xi_\phi(z)}{\xi(z)} =
    \Bigl(\frac{A_\phi}{A_0}\Bigr)^{1/\phi}
    =
    2 \biggl(\frac{\Gamma( \frac{\phi+2}{2})\Gamma( \frac{\phi+d}{2} )}{\Gamma(\frac d2)}\biggr)^{1/\phi}.
\end{equation}
The limiting constant is \emph{universal} in the sense that it
does not depend on $h_z$ or $J_z$.
The particular case
$\xi_2(z)/\xi(z) \to \sqrt{2d}$ was proved for self-avoiding walk in dimensions $d \ge 5$,  in \cite[Theorems~1.2,~1.5]{HS92a}.
Similarly, for any $\phi,\psi>0$, the limit
\begin{equation}
\label{eq:xiratio}
    \lim_{z\to z_c}
    \frac{\xi_\phi(z)}{\xi_\psi(z)}
    =
    \frac{[\Gamma( \frac{\phi+2}{2})\Gamma( \frac{\phi+d}{2})]^{1/\phi} }
    {[\Gamma( \frac{\psi+2}{2})\Gamma( \frac{\psi+d}{2} )]^{1/\psi}}
    \Gamma\Big(\frac d2 \Big)^{\frac{1}{\psi}-\frac{1}{\phi}}
\end{equation}
is a universal ratio.
The universal ratio \eqref{eq:xiratio} was proved for $4$-dimensional $n$-component $|\varphi|^4$ models in \cite{BSTW-clp}.
We see from \eqref{eq:xiphi-xi-ratio}--\eqref{eq:xiratio} that all correlation lengths are
equivalent, up to universal constants.

\begin{proof} [Proof of Theorem~\ref{thm:xiphi}]
The $\phi = 0$ case of \eqref{eq:xi-unnormalised} follows from \eqref{eq:chim2} and $A_0 = 2$, and \eqref{eq:xiphiasy} follows
from \eqref{eq:xi-unnormalised}, so it suffices to prove \eqref{eq:xi-unnormalised} for $\phi >0$.
Throughout the proof, we write $|x|$ for the Euclidean norm $\|x\|_2$.

Let $\phi > 0$ and fix $\so > 0$.
For $x$ that satisfies $\abs x \le \so / \mz$, by \eqref{eq:chim2} we have
\begin{equation} \label{eq:corr_pf}
0 \le \sum_{x : \abs x \le \so / \mz} \abs x^\phi S_z(x)
\le \Bigl( \frac \so \mz \Bigr)^\phi \chi(z)
\lesssim \Bigl( \frac \so \mz \Bigr)^\phi
	\frac{ \hat h_z(0) }{ \sigma_z^2 \mz^2 } .
\end{equation}
For the remaining large $x$, we change variables $y = \mz x$ and write
\begin{equation}
\sum_{ \substack{ x \in \Zd \\ \abs x > \so / \mz } } \abs x^\phi S_z(x)
= \sum_{ \substack{ y\in \mz \Zd \\ \abs y > \so} }
	\Bigabs{ \frac y {\mz} }^\phi S_z\Bigl(\frac y{\mz}\Bigr)
= \frac {d \hat \g_z (0)} { \sigma_z^2 \mz^{2+\phi} } \mz^d
	\sum_{ \substack{ y\in \mz \Zd \\ \abs y > \so}}
		\frac{ \sigma_z^2 } { d \mz^{d-2} \hat \g_z (0)}
		\abs y^\phi S_z\Bigl(\frac y{\mz}\Bigr) .
\end{equation}
For nonzero $y\in \Rd$, we define
\begin{equation}
f_z(y)
= \frac{ \sigma_z^2 } { d \mz^{d-2} \hat \g_z (0)}
		\mz^\phi \Biginteger{ \frac y {\mz} }^\phi
		S_z\Bigl( \Biginteger{ \frac y{\mz} } \Bigr) ,	
\end{equation}
so that
\begin{equation}
\sum_{ \substack{ x \in \Zd \\ \abs x > \so / \mz } } \abs x^\phi S_z(x)
= \frac {d \hat \g_z (0)} { \sigma_z^2 \mz^{2+\phi} }
	\int_ {\abs{\integer{y/\mz}}> \so / \mz} f_z(y) \D y  .
\end{equation}
By Proposition~\ref{thm:critical_decay},
followed by the scaling identity \eqref{eq:Gscaling} for $\mG$,
\begin{equation}
\lim_{z\to z_c} f_z(y)
=  \frac {  \abs y^\phi } {  \abs {y}^{d-2} } \mG_{ \abs y }(\hat y)
=   \abs y^\phi  \mG_{ 1 }(y) 	.
\end{equation}
Also, by Corollary~\ref{cor:crossover_weak}, by the fact that
$\sigma_z^2 \le M$, and by our assumption that $\hat \g_z (0)  \ge M\inv$,
we have the uniform estimate
\begin{equation}
f_z(y)
\lesssim \frac{ 1 } { \mz^{d-2} }
	\abs y^\phi
	\frac{ \mz^{d-2} } { \abs y^{(d-1)/2} } e^{-c \abs y}
= 	\frac{ \abs y^\phi } { \abs y^{(d-1)/2} } e^{-c \abs y}
	\qquad (\abs y \ge \half s_0)
\end{equation}
for some $c > 0$.
This gives a dominating function, so by dominated convergence
and by inserting the limits of $\sigma_z$ and $\hat h_z(0)$,
\begin{equation}
\sum_{\substack{ x\in \Zd \\ \abs x > \so / \mz}} \abs x^\phi S_z(x)
\sim \frac {d \hat h_{z_c}(0) }{ \sigma_{z_c}^2 \mz^{2+\phi} }
	\int_{\abs y > \so}  \abs y^\phi \mG_1(y) \D y .
\end{equation}
Combined with \eqref{eq:corr_pf}, we obtain
\begin{equation}
\begin{aligned}
\frac{  \sigma_{z_c}^2   } { d \hat h_{z_c}(0) }
\limsup_{z\to z_c}
	\mz^{2+\phi}
	\sum_{x\in \Zd} \abs x^\phi S_z(x)
&\le O(\so^\phi) + \int_{\abs y > \so}  \abs y^\phi \mG_1(y) \D y ,
\\
\frac{  \sigma_{z_c}^2   } { d \hat h_{z_c}(0) }
\liminf_{z\to z_c}
	\mz^{2+\phi}
	\sum_{x\in \Zd} \abs x^\phi S_z(x)
&\ge 0 + \int_{\abs y > \so}  \abs y^\phi \mG_1(y) \D y .
\end{aligned}
\end{equation}
Since the left-hand sides are independent of $\so$, we can take $\so \to 0$ to conclude
\begin{equation}
\frac{  \sigma_{z_c}^2   } { d \hat h_{z_c}(0) }
\lim_{z\to z_c}
	\mz^{2+\phi}
	\sum_{x\in \Zd} \abs x^\phi S_z(x)
= \int_{\Rd}  \abs y^\phi \mG_1(y) \D y
= A_\phi,
\end{equation}
as desired.
\end{proof}

The combination of Corollary~\ref{cor:chim} and Theorem~\ref{thm:xiphi}
leads to an expression for the limiting ratios $\xi_\phi^2/\chi$
and $\xi_\phi/\xi$ (with $\xi=m_z^{-1}$).
For $\phi=2$, the next lemma enhances this with an error estimate, which
we apply to percolation in Section~\ref{sec:perc-mr}.

\begin{lemma} \label{lem:xi2-chi-ratio}
Let $d\ge 1$ and $\delta >0$.
Suppose that Assumptions~\ref{ass:Omega} and~\ref{ass:J} hold for $S_z$ for each $z\in [z_c-\delta,z_c)$, suppose that $m_z \to 0$ as $z \to z_c$,
and suppose that there are constants $M,\KIR,\zeta$ such that
$(J_z\supmu,\g_z\supmu) \in \Qcal_{M,\KIR,\zeta}$
and $\hat \g_z\supmu (0)  \ge M\inv > 0$
uniformly in $\mu\in \overline\Omega_z$ and in $z \in [z_c-\delta,z_c)$.
Then
\begin{equation} \label{eq:xi2-chi-ratio}
\frac{\xi_2(z)^2}{\chi(z)}
	= \frac{\sigma_{z}^2}{\hat \g_z(0)}
		\Bigl[1 + O\Bigl( \frac 1 { \chi(z) } \Bigr) \Bigr] ,
\qquad
\frac{\xi_2(z)}{\xi(z)}
	= \sqrt{2d} \Bigl[ 1 +O\Bigl( \frac 1 {\xi(z)^{\zeta\wedge 2} }\Bigr)  \Bigr] .
\end{equation}
\end{lemma}

\begin{proof}
Our hypotheses imply $\hat h_z(0) / \sigma_z^2 \asymp 1$,
so by inverting the asymptotic relation given by Corollary~\ref{cor:chim}, we have
\begin{equation}
\xi(z)^2 = \frac 1 { \mz^2 }
= \frac{ \sigma_z^2}{ 2d \hat h_z(0) }
	\chi(z) \Bigl[ 1 + O\Bigl(\frac1{\chi(z)^{ (\zeta \wedge 2)/2 }} \Bigr) \Bigr] .
\end{equation}
By dividing by the above and taking a square root, the second ratio of \eqref{eq:xi2-chi-ratio} follows from the first ratio. We therefore only need to prove the first ratio of \eqref{eq:xi2-chi-ratio}.

Given $f:\Zd\to\R$, we write $\sigma_f^2 = \sum_{x\in\Zd}|x|^2f(x)$
(so $\sigma_{J_z}$ is the same as $\sigma_z$).
In what follows,
we drop labels $z$, and all Fourier transforms are
evaluated at $k=0$.
By taking the second moment of the OZ equation \eqref{eq:OZz} and using $\Zd$-symmetries, we have
\begin{equation}
    \sigma_S^2 = \sigma_{\g}^2 + \sigma_J^2\chi + \hat J \sigma_S^2 .
\end{equation}
We solve for $\sigma_S^2$, and then use the relation $\chi = \hat \g / (1-\hat J)$ which comes from summing the OZ equation, to get
\begin{equation}
    \sigma_S^2  = \frac{\sigma_{\g}^2}{1-\hat J} + \frac{ \sigma_J^2\chi}{1-\hat J}
	= \frac{ \sigma_h^2 \chi }{ \hat h } + \frac{ \sigma_J^2 \chi^2 }{ \hat h } .
\end{equation}
Dividing by $\chi^2$ and using the definition for $\xi_2$ in \eqref{eq:xidef}, we obtain
\begin{equation}
\frac{ \xi_2^2 }\chi
= \frac{ \sigma_S^2 }{ \chi^2 }
= \frac{ \sigma_J^2 }{ \hat h } + \frac{ \sigma_h^2 }{ \hat h \chi }
= \frac{ \sigma_J^2 }{ \hat h }
	\Bigl[1 + O\Bigl( \frac 1 { \chi } \Bigr) \Bigr] ,
\end{equation}
which is the desired result.
\end{proof}

\section{Random walk}
\label{sec:RW}

In this section, we consider random walks on $\Zd$.
Let $D: \Zd \to [0,1]$ be a $\Zd$-symmetric probability distribution
that is not supported on a proper sublattice of $\Zd$, with
$\sigma_D^2 = \sum_{x\in \Zd}|x|^2 D(x) < \infty$.
The random walk $(X_n)_{n\ge 0}$ started at the origin with step distribution $D$, i.e., with $\P(X_n=x)=D^{*n}(x)$, is therefore irreducible.
The support hypothesis also ensures that $\hat D(k) <1$ for all nonzero $k \in \T^d$.
We combine this fact
with a Taylor expansion near $k=0$ to get the infrared bound
\begin{equation} \label{eq:D_infrared}
\hat D(0) - \hat D(k)
\ge K_D \abs k^2
	\qquad (k\in \Td)
\end{equation}
for some $K_D > 0$.

For $z\in (0,1)$, the \emph{subcritical Green function} $S_z:\Zd\to [0,\infty)$ is defined by
\begin{equation}
    S_z(x) = \sum_{n=0}^\infty     \P(X_n=x) z^n
    = \sum_{n=0}^\infty z^n D^{*n}(x).
\end{equation}
By the Markov property (or from the last equality above), $S_z$ obeys
\begin{equation} \label{eq:OZRW}
S_z(x) = \delta_{0,x} + (zD * S_z)(x) .
\end{equation}
This is the OZ equation \eqref{eq:OZz} with $h_z = \delta_0$ and $J_z = z D$.  For $z \in [0,1)$, summation of \eqref{eq:OZRW} over $\Z^d$ gives
\begin{equation}
\label{eq:chiRW}
    \chi(z) = \frac{1}{1-z}.
\end{equation}
By definition, we have
\begin{equation} \label{eq:Omega_RW}
\Omega_z
= \bigl\{ \mu\in \Rd : \chi\supmu(z) < \infty \bigr \}
= \bigl\{ \mu\in \Rd : z \hat D\supmu(0) < 1 \bigr \} .
\end{equation}

\subsection{Vanishing mass at criticality}

The next lemma proves that the mass $m_z$ goes to zero as $z$ approaches the critical point $z_c=1$.

\begin{lemma} \label{lem:m_to_0}
Let $d \ge 1$ and $z \in (0,1)$.
Suppose $D: \Zd \to [0,1]$ is a $\Zd$-symmetric
probability distribution, not supported on a proper sublattice of $\Zd$, with $\sigma_D^2 = \sum_{x\in \Zd}|x|^2 D(x) < \infty$.
Then the limit
\begin{equation} \label{eq:tilde_m}
\tilde m_z
= \lim_{n\to \infty}  \frac {- \log S_z(ne_1)  } n
\in [0,\infty)
\end{equation}
is well-defined,
coincides with the quantity $m_z$ defined in \eqref{eq:def_mS},
and satisfies $\tilde m_z \to 0$ as $z\to 1$.
\end{lemma}

\begin{proof}
Exactly as in \cite[Theorem~A.2]{MS93}, it follows from
the FKG-type inequality
\begin{equation}
    S_z(x) \ge S_z(0)^{-1}S_z(y)S_z(x-y)
   \qquad (x,y\in \Zd)
\end{equation}
and Fekete's subadditive lemma that the limit \eqref{eq:tilde_m}
exists in $[0,\infty)$ and that
\begin{equation} \label{eq:RW_mass_pf}
S_z(x) \le S_z(0) e^{- \tilde m_z \norm x_\infty}
   \qquad (x\in \Zd) .
\end{equation}
This implies $\chi\supk{te_1}(z) = \sum_{x\in  \Zd} S_z(x) e^{t x_1}< \infty$ for all $\abs t < \tilde m_z$,
so we obtain $m_z \ge \tilde m_z$.
On the other hand, it follows from \eqref{eq:tilde_m} that $\chi^{(te_1)}(z) =\infty$ if $t>\tilde m_z$, so we also have
the opposite inequality $m_z \le \tilde m_z$.

To prove $\tilde m_z \to 0$, we assume for contradiction that $\tilde m_z \ge \eps > 0$ along some subsequence $z_n \to 1$.
Summing \eqref{eq:RW_mass_pf} over $x\in \Zd$ then gives
\begin{equation} \label{eq:RW_mass_pf2}
\frac 1 {1-z_n} = \chi(z_n)
\le S_{z_n}(0) \sum_{x\in \Zd} e^{-  \eps \norm x_\infty}
= C_\eps S_{z_n}(0) .
\end{equation}
On the other hand, by the infrared bound \eqref{eq:D_infrared}
and the inverse Fourier transform,
\begin{equation}
S_{z_n}(0) = \int_\Td \frac{\D k }{(2\pi)^d}
	\frac { 1 } { 1 - z_n\hat D (k) }
\le \int_\Td \frac{\D k }{(2\pi)^d}
	\frac { 1  } { 1- z_n + z_n K_D \abs k^2 }
\lesssim \begin{cases}
	(1-z_n)^{-1/2}		&( d=1 )  \\
	\abs{ \log (1-z_n)   } 	& (d=2) 	\\
	1 				& (d>2) .
\end{cases}
\end{equation}
This is not compatible with \eqref{eq:RW_mass_pf2} as $z_n\to1$.
\end{proof}

\subsection{Crossover for random walk}

We say that $D(x)$ has \emph{super-exponential decay} if, for every $a>0$,  $D(x) = o(e^{-a\abs x})$ as $\abs x \to \infty$.

\begin{theorem} \label{thm:RW}
Let $d \ge 1$. Suppose $D(x)$ has super-exponential decay.
Then $S_z$ obeys Assumptions~\ref{ass:Omega} and~\ref{ass:J} for every $z\in (0,1)$.
Moreover, for any $\eps \in (0,1)$ and any $\zeta > 0$,
there are constants $M, \KIR$ such that
$(zD\supmu, \delta_0\supmu) \in \Qcal_{M,\KIR,\zeta}$ \emph{uniformly} in $\mu\in \overline\Omega_z$ and in $z \in [\eps, 1)$.
\end{theorem}

\begin{proof}
\emph{Verification of Assumption~\ref{ass:Omega}.}
By our assumption on $D$, the random walk is irreducible and hence $p_k = D^{*k}(e_1) > 0$ for some $k > 0$.
Since $D^{*nk}(ne_1) \ge [D^{*k}(e_1)]^n$,
we thus have
\begin{equation}
\chi\supk{\mu_1 e_1 }(z)
\ge \sum_{n=1}^\infty S_z(ne_1) e^{n \mu_1}
\ge \sum_{n=1}^\infty z^{nk} p_k^n  e^{n \mu_1}
= \infty
\end{equation}
for all $\mu_1 \ge - \log(z^k p_k)$.
Such $\mu_1 e_1$ do not belong to $\Omega_z$.

To prove that $\Omega_z$ is an open set, we take $\mu\in\Omega_z$ and apply an argument in the spirit of Simon--Lieb.
For any $R \ge 1$ and any $x\not \in B_R =\{y\in\Zd : \|y\|_\infty \le R\}$, we have
\begin{equation}
    S\supmu_z (x) \le \sum_{y\in B_R}\sum_{w\not\in B_R} S\supmu_z(y) zD\supmu (w-y)S\supmu_z (x-w).
\end{equation}
This can be seen via the Markov property, by taking $w$ to be the first visited vertex outside of $B_R$, taking $y$ to be the vertex right before $w$, and then tilting.
Since $D(x)$ has super-exponential decay, $\sum_{x\in \Zd} zD\supmu(x)e^{\abs x}<\infty$.
We also have $\chi\supmu(z) <\infty$ since $\mu\in \Omega_z$.
It is then an immediate consequence of \cite[Proposition~A.3]{Hara90} that
\begin{equation}
    \GL\supmu_z(x) \le C e^{-c\abs x}
    	\qquad(x\in \Zd)
\end{equation}
for some $c,C>0$ (depending on $z,\mu$).
By the Cauchy--Schwarz inequality, for $\nu\in \Rd$ this gives
\begin{equation}
    \GL_z^{(\mu+\nu)}(x)
    = \GL_z^{(\mu)}(x) e^{\nu \cdot x}
    \le C e^{-(c-\abs\nu )\abs x }.
\end{equation}
Summation then shows that $\chi^{(\mu+\nu)}(z)<\infty$ when $\abs\nu <c$, so $\Omega_z$ is an open set.

\smallskip\noindent
\emph{Verification of uniform Assumption~\ref{ass:J}.}
We fix $\eps \in (0,1)$, $\zeta >0$ and want to show $(zD\supmu, \delta_0\supmu) \in \Qcal_{M, \KIR, \zeta}$ uniformly in $\mu\in \overline\Omega_z$ and in $z \in [\eps, 1)$.
The conditions on $g = \delta_0\supmu = \delta_0$ hold trivially for any $M\ge 1$.
For the moments of $Q = zD\supmu$,
using $z\le 1$ and $\Omega_z \subset \Omega_\eps$ we have
\begin{equation} \label{eq:RW_moment_pf}
\bignorm{ \abs y^{2+\zeta} z D \supmu(y) }_1
\le \Bignorm{ \abs y^{2+\zeta} D(y)
	\exp\bigl\{ \max_{\mu \in \overline \Omega_\eps}\, \abs \mu \abs y \bigr\} }_1 .
\end{equation}
Since $\Omega_\eps$ is a bounded set by Lemma~\ref{lem:Omega_convex}, the maximum over $\overline \Omega_\eps$ is finite, and we can bound \eqref{eq:RW_moment_pf} by a constant using the super-exponential decay of $D(x)$.
Other moments of $zD\supmu$ can be bounded analogously.
For the infrared bound, we recall from \eqref{eq:D_infrared} that
\begin{equation}
\hat D(0) - \hat D(k)
\ge K_D \abs k^2
	\qquad (k\in \Td) .
\end{equation}
Using $\cosh(\cdot) \ge 1$, for any $z \in [\eps, 1)$ and $\mu\in \Rd$
we then have
\begin{align} \label{eq:RW_infrared_pf}
\Re [z \hat D\supmu(0) - z \hat D\supmu(k) ]
&= z \sum_{y\in \Zd} D(y) \cosh(\mu \cdot y) [ 1 - \cos (k \cdot y) ] \nl
&\ge \eps \sum_{y\in \Zd} D(y) [ 1 - \cos (k \cdot y) ]
\ge \eps K_D \abs k^2 ,
\end{align}
so we can take $\KIR = \eps K_D$.
This concludes the proof.
\end{proof}

\begin{remark}
The assumption of super-exponential decay in Theorem~\ref{thm:RW} has been made for convenience.
If we instead have the weaker hypothesis of exponential decay (i.e., $D(x) \lesssim e^{-a \abs x}$ for some $a>0$),
then the above proof can be adapted to prove that Assumptions~\ref{ass:Omega} and \ref{ass:J} hold uniformly in
$z \in [1-\delta_a,1)$ for some $\delta_a >0$.
Both the Simon--Lieb argument and the tilted moments of $J_z = z D$ impose restrictions on the size of $\delta_a$.
Related results without uniformity in $z$ are stated (but not proved) in \cite[Appendix~C]{TH15}.
It is not always true that the Assumptions will hold
for all $z \in (0,1)$.
Indeed, OZ decay can fail for small $z$; see Section~\ref{sec:counterx}.
\end{remark}

By Theorem~\ref{thm:RW}, if $D$ has super-exponential decay, then
Theorem~\ref{thm:crossover} and
Corollaries~\ref{cor:OZ} and~\ref{cor:crossover_weak} hold uniformly in $z \in [\eps,1)$ for any choice of $\eps > 0$,
and all results of Section~\ref{sec:critical} apply.
We highlight a few of the results:
\begin{itemize}
\item
Let $d \ge 1$ and $z\in (0,1)$.
Precise asymptotics for the Green function are given in Theorem~\ref{thm:crossover}:
\begin{equation} \label{eq:RW-asymp-intro}
\GL_z(x)
= \C(x; \eta_{\hat x,z} , \Lambda_{\hat x,z}) e^{-\mz  |x|_z} [ 1 + O(|x|^{-\eps})]
\end{equation}
as $\abs x \to \infty$,
with a $z$-dependent norm $|\cdot|_z$, mass $m_z>0$,
$|\eta_{\hatxz}| \asymp m_z$ and $\hat x \cdot \Lam_{\hatxz}\inv \hat x\asymp 1$.

\item
Let $d >2$.  In these transient dimensions, the error term in \eqref{eq:RW-asymp-intro} remains uniform in the limit $z\to 1$.  In this limit, $m_z\to 0$, $\eta_{\hat x,z}\to 0$,
and $\Lambda_{\hat x,z} \to (\sigma_D^2/d){\rm Id}$.  It then follows from
\eqref{eq:Ccrit} that
\begin{equation}
\label{eq:RWcrit}
\GL_{1}(x) =
    \frac{d\Gamma(\frac{d-2}{2})}{2\pi^{d/2}}
    \frac{1}{\sigma_D^2|x|^{d-2}}
    [ 1 + O(|x|^{-\eps})].
\end{equation}
This recovers the well-studied asymptotic behaviour of the critical Green function
\cite{LL10,Uchi98}.

\item
Let $d \ge 1$ and $z\in (0,1)$.
By Corollary~\ref{cor:OZ}, $S_z$ exhibits the OZ decay
\begin{equation}
\label{eq:RW-OZ}
\GL_z(x)
=
\frac 1 { (2\pi)^{(d-1)/2} \sqrt{\det\Lambda_{\hatxz}} }
	\frac{1}{ (\hat x \cdot \Lam_{\hatxz}\inv \hat x)^{1/2} }
    \frac{ \abs {\eta_{\hatxz}}^{(d-3)/2}}{ \abs  x^{(d-1)/2}}
    e^{-\mz  \abs x_z}[1+o(1)].
\end{equation}

\item
Let $d \ge 1$ and $z\in (0,1)$.
By Theorem~\ref{thm:critical_limit}, we have
\begin{equation} \label{eq:crit_limit-RW}
\abs x_z = \norm x_2 [ 1 + O(\mz ^{ 2 } ) ]
	\qquad (z\to 1).
\end{equation}
\item
Let $d \ge 1$ and $z\in (0,1)$.
By Corollary~\ref{cor:chim}  and
the exact expression $\chi(z) = (1-z)\inv$ from \eqref{eq:chiRW}, we have
\begin{equation}
\label{eq:chimRW}
\mz^2 = \frac{2d}{\sigma_D^2} (1-z)[ 1+O(1-z) ]
	\qquad (z\to 1).
\end{equation}

\item
Let $d>2$.
By Corollary~\ref{cor:crossover_weak}, we have
\begin{equation}
    \GL_z(x) \asymp
    \frac{    \max\{ 1 , \mz  |x|_z \} ^{(d-3)/2} } { \abs x_z^{d-2} }
	e^{- m_z| x|_z}
	\qquad (z \in [\eps,1),\ |x|_z \ge R)
\end{equation}
with constants depending only on $\eps$.
A similar but weaker result holds for $d\le 2$.

\item
Let $d\ge 1$ and $z\in (0,1)$.
Recall the definition of the correlation length of order $\phi$ from \eqref{eq:xidef}.
By Theorem~\ref{thm:xiphi} and \eqref{eq:chimRW}, for any $\phi>0$,
\begin{equation}
\label{eq:RWxiphi}
    \xi_\phi(z)
    \sim
    \Bigl(\frac{A_\phi}{A_0}\Bigr)^{1/\phi} \frac 1 \mz
    \sim
    \Bigl(\frac{A_\phi}{2}\Bigr)^{1/\phi}
    \frac{\sigma_{D} }{\sqrt{2d} } \frac{1}{(1-z)^{1/2}}
    \qquad (z\to 1).
\end{equation}
In the special case of simple random walk in dimensions $d>2$, for which $D(x) = P(x) = \frac{1}{2d}\1\{\|x\|_1=1\}$, a very different proof of \eqref{eq:RWxiphi} is
given in the Appendix of \cite{BSTW-clp}.
\end{itemize}

All the above conclusions of Theorem~\ref{thm:RW} also apply
to $D$ decaying  exponentially, uniformly in $z \in [1-\delta, 1)$ for some $\delta >0$.

For the case of the nearest-neighbour random walk,
the asymptotic behaviour of $S_z$ (the lattice Green function) given by Theorem~\ref{thm:crossover} recovers the results of \cite{MS22}.
The method of \cite{MS22} used an explicit representation of
the lattice Green function as an integral of modified Bessel functions of the
first kind, and it does not apply in our general
setting.  Related results for the nearest-neighbour walk can be found in \cite{Mess06,MY12}.

The leading behaviour of $S_z$ in the massive critical regime
of Theorem~\ref{thm:critical_decay}
can be inferred from \cite[Proposition~3.1]{DGGZ22},
whose proof uses the local central limit theorem.
In the massive critical regime,
higher-order terms beyond the leading term in Theorem~\ref{thm:critical_decay} are presented in \cite[(21)]{PS99} without rigorous control of the error.

\subsection{Massive limit of norm}

We now prove a complementary result to \eqref{eq:crit_limit-RW}
for the $z\to 0$ limit of the norm $|\cdot |_z$,
under the additional assumption that $D(x)$ has finite support $\Ucal$.
We write $\conv(\Ucal)$ for the closed convex hull of $\Ucal$,
and define
\begin{equation}
\rad(\Ucal) = \max_{x\in \Ucal}\, x_1  > 0 .
\end{equation}
Using the finite support,
and using $\mz e_1 \in \del \Omega_z$,
we have
\begin{equation} \label{eq:exp_rad}
1 = \hat J_z \supk{ \mz e_1 }(0)
= z \sum_{y \in \Ucal} D(y) e^{\mz y_1}
\asymp  z e^{m_z \rad(\Ucal)} .
\end{equation}
It follows that $\mz \sim \rad(\Ucal)\inv \log(1/z) \to \infty$ as $z\to 0$, so the mass diverges logarithmically to infinity in this limit.

\begin{theorem} \label{thm:massive_limit}
Let $d \ge 1$. Suppose $D(x)$ has finite support $\Ucal$.
Let $x \in \Rd$ be on the boundary of $\conv(\Ucal)$.
Then as $z \to 0$, we have
\begin{equation}
\abs x_{z} = \rad(\Ucal) + O(\mz\inv) .
\end{equation}
In particular, $\lim_{z\to 0} | \cdot |_z$ defines a norm on $\Rd$, with its unit ball being $\conv(\Ucal) / \rad(\Ucal)$.
\end{theorem}

\begin{proof}
Let $x$ be on the boundary of $\conv(\Ucal)$.
Let $z\in (0,1)$ and $\nu \in \Rd$.
Since the support $\Ucal$ of $D$
is finite, we have
\begin{equation}
\hat J_z \supk{ \mz \nu } (0)
= z \sum_{y\in \Ucal} D(y) e^{\mz \nu \cdot y}
\asymp z  \exp\{ \mz \max_{y\in \Ucal}\, \nu \cdot y \} .
\end{equation}
By dividing the above by \eqref{eq:exp_rad}, we find that there are constants (independent of $z,\nu$) such that
\begin{equation} \label{eq:Jmnu}
\hat J_z \supk{ \mz \nu } (0)
\asymp
\exp\Bigl\{  m_z ( \max_{y\in \Ucal}\, \nu \cdot y   - \rad(\Ucal)) \Bigr\} .
\end{equation}

\smallskip
\noindent \emph{Upper bound.}
By taking a convex combination, it suffices to consider only $x\in \Ucal$.
Since $\muxz \in \del \Omega_z$,
we have $\hat J_z \supk{ \muxz }(0) = 1$.
Applying \eqref{eq:Jmnu} with $\nu = \muxz / \mz$ then gives
\begin{equation}
\Bigabs{ \max_{y\in \Ucal}\, \frac{ \muxz \cdot y } { \mz }   - \rad(\Ucal) }
\asymp \frac 1 {\mz} .
\end{equation}
Using the definition of the norm, we then have
\begin{equation}
\abs x_z
= \frac{ \muxz \cdot x } { \mz}
\le \max_{y\in \Ucal}\, \frac{ \muxz \cdot y } { \mz }
= \rad(\Ucal)  + O(\mz\inv) ,
\end{equation}
as desired.

\smallskip
\noindent \emph{Lower bound.}
Since $x$ is on the boundary of $\conv(\Ucal)$, there exists $\nu_x \in \Rd$ for which $\nu_x \cdot x  = \max_{y\in \Ucal} \nu_x \cdot y > 0$.
Let $\theta \in (0,1)$.
We apply \eqref{eq:Jmnu}  to $\nu = \frac{ \theta \rad(\Ucal) }{ \nu_x \cdot x } \nu_x $, and obtain
\begin{equation}
\hat J_z \supk{ \mz \nu } (0)
\asymp
\exp\Bigl\{  m_z \Bigl( \frac{ \theta \rad(\Ucal) }{ \nu_x \cdot x } \nu_x \cdot x   - \rad(\Ucal) \Bigr) \Bigr\}
= \exp\Bigl\{  -(1-\theta) m_z  \rad(\Ucal) \Bigr\} .
\end{equation}
Using $\mz \to \infty$ as $z\to 0$,
we pick $\theta = 1 - C \mz\inv$ with a sufficiently large $C$ so that $\hat J_z \supk{ \mz \nu }(0) \le \frac 1 2$ for all $z$ sufficiently small.
Then \eqref{eq:J1-intro} shows $\mz \nu \in \Omega_z$ for these $z$,
and the optimality of $\muxz$ from \eqref{eq:optimal-mu-intro} gives
\begin{equation}
\abs x_z
= \frac{ \muxz \cdot x }{ \mz }
\ge \nu \cdot x
= \theta \rad(\Ucal)
= ( 1 - C \mz\inv) \rad( \Ucal) .
\end{equation}
This concludes the proof.
\end{proof}

\subsection{Examples}
\label{sec:norm_comp}

We now explicitly compute the norm $|\cdot|_z$ for two examples.
We use Proposition~\ref{prop:geom-intro} with $J_z = zD$.
As discussed after Proposition~\ref{prop:geom-intro}, the mass $\mz$ can be computed from
$z \hat D\supk{\mz e_1 }(0) = 1$.
When $D(x)$ decays super-exponentially, the function $\mu \mapsto \hat D\supmu(0) = \sum_{x\in \Zd} D(x) \cosh(\mu\cdot x)$ is smooth on $\Rd$, so by \eqref{eq:Omega_RW} we have
\begin{equation}
\overline \Omega_z
= \bigl\{ \mu\in \Rd : z \hat D\supmu(0) \le 1 \bigr \} .
\end{equation}
Therefore, for any nonzero $x \in \Rd$, the optimisation problem \eqref{eq:optimal-mu-intro} that determines $\mu_{\hatx, z}$ is exactly
\begin{equation}
\mu_{\hat x, z}\cdot x = \max_\mu \bigl\{ \mu\cdot x : z \hat D\supmu(0) \le 1  \bigr\} .
\end{equation}
This can be solved by the method of Lagrange multipliers: the optimiser $\mu_{\hatx, z}$ and the multiplier $\lam_{x, z}$ are the unique solution to the system
\begin{equation} \label{eq:Lagrange}
z \hat D \supmu(0) = 1,
\qquad
z \grad \hat D \supmu(0) = \lam x .
\end{equation}

\subsubsection{Simple random walk}
\label{sec:nnRW}

The (nearest-neighbour) simple random walk has transition probability
\begin{equation}
D(x) = P(x) = \frac{1}{2d} \1\{ \|x\|_1=1 \} ,
\end{equation}
so $\hat D\supmu(0) = \frac 1 d\sum_{j=1}^d \cosh \mu_j$.
By setting $\mu=\mz e_1$ in the first equation of \eqref{eq:Lagrange}, we find that $\mz$ is given by the positive solution of
\begin{equation} \label{eq:mass_NN}
    \cosh \mz  = 1 + \frac dz (1-z).
\end{equation}

For $x\ne 0$, the system \eqref{eq:Lagrange} becomes
\begin{equation}
\sum_{j=1}^d \cosh \mu_j = \frac{d}{z},
\qquad
\sinh \mu_j = \frac {d  \lam x_j} z .
\end{equation}
We solve for $\mu_j$ in the second equation, and then use $\cosh(\arcsinh  t) = \sqrt{1+t^2}$ in the first equation to get
\begin{equation} \label{eq:nn_implicit_eqn}
\sum_{j=1}^d  \sqrt{1+\Bigl( \frac{d \lam x_j}{z} \Bigr)^2 } = \frac{d}{z}.
\end{equation}
Since the left-hand side is strictly increasing in $\lam$,
there is a unique solution $\lam = \lam_{z,x}$.
By Corollary~\ref{cor:norm}, we obtain
\begin{equation}
\label{eq:normSRW}
|x|_z = \frac{\mu_{\hat x, z}\cdot x}{\mz}
=
\frac{1}{\mz }\sum_{j=1}^d x_j \arcsinh \frac{d \lam_{z,x} x_j}{z}.
\end{equation}
(In the notation of \cite{MS22}, $\lam_{z,x} = \frac z d u_a(x)$ where $1 + a^2 = 1/z$.)
This norm was proven in \cite{MS22} to interpolate monotonically from the $\ell^1$-norm at $z=0$ to the $\ell^2$-norm at $z=1$; see Figure~\ref{fig:nn}.

\begin{figure}[ht]
\centering{
\includegraphics[width=4cm, height=4cm]{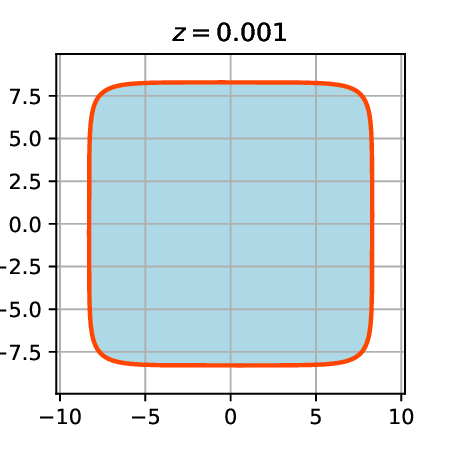}
\hspace{5mm}
\includegraphics[width=4cm, height=4cm]{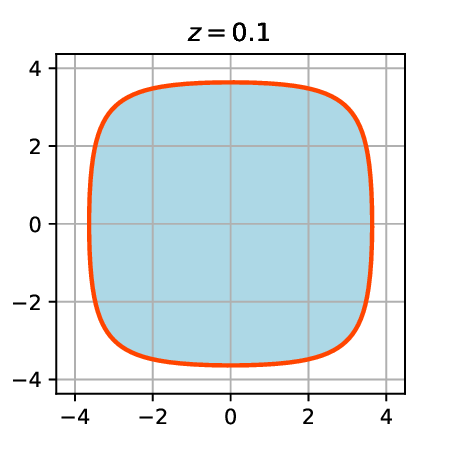}
\hspace{5mm}
\includegraphics[width=4cm, height=4cm]{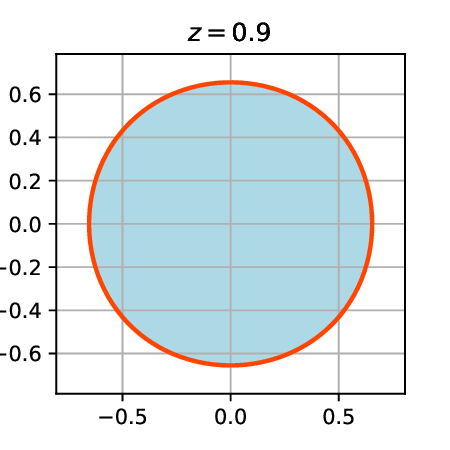}

\vspace{5mm}

\includegraphics[width=4cm, height=4cm]{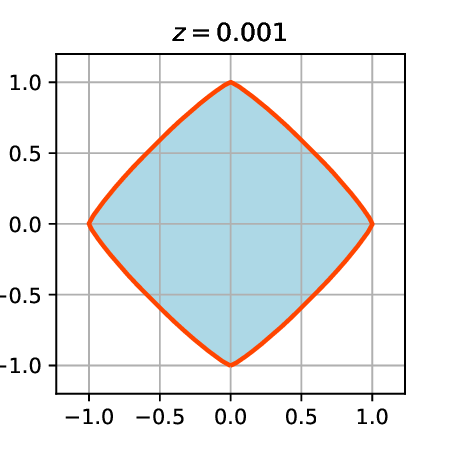}
\hspace{5mm}
\includegraphics[width=4cm, height=4cm]{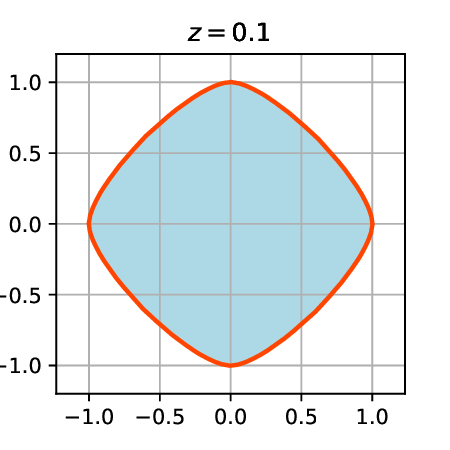}
\hspace{5mm}
\includegraphics[width=4cm, height=4cm]{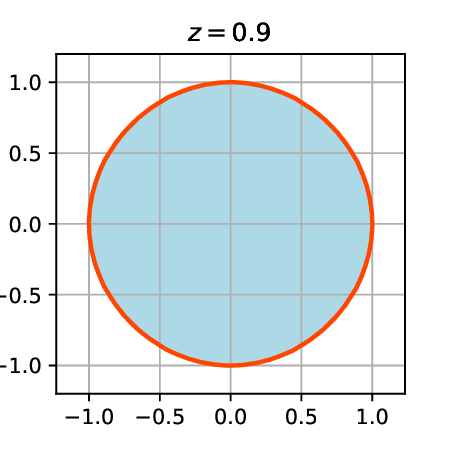}

\caption{
Wulff shape $\overline \Omega_z$ (upper) and
unit ball of $|\cdot|_z$ (lower) for simple random walk in dimension $d=2$.
}
\label{fig:nn}}
\end{figure}

\begin{remark}
For $d=1$, the simple random walk Green function can be computed exactly
(by contour integration, or by using \cite[(3.616.7)]{GR07}
with $b = e^{-\mz}$) as
\begin{equation}
    S_z(x) = \int_{-\pi}^\pi \frac{e^{ikx}}{1- z\cos k} \frac{dk}{2\pi}
    =
    \frac{e^{-m_z |x|}}{z\sinh m_z}
    = \frac{e^{-m_z |x|}}{ \sqrt{1-z^2}} ,
\end{equation}
while the Crossover Theorem gives
\begin{equation}
    S_z(x) \sim \bg(|x|;\eta_1,\Lam_1) e^{-m_z|x|}
    = \frac{1}{|\eta_1|} e^{-m_z|x|}.
\end{equation}
Using \eqref{eq:nn_implicit_eqn} and \eqref{eq:Lagrange}, one can easily compute $\eta_1 = \sqrt{1-z^2}$,
so the asymptotic form is identical to the exact formula for $S_z(x)$ when $d=1$.
For $d=3$, $S_1(0)$ is a Watson integral, and values for $z<1$ and some nonzero $x$ have also been studied \cite{Zuck11,Gutt10}.
\end{remark}

\begin{figure}[ht]
\centering{
\includegraphics[width=4cm, height=4cm]{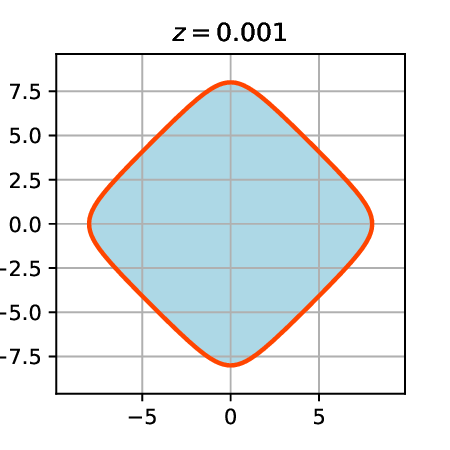}
\hspace{5mm}
\includegraphics[width=4cm, height=4cm]{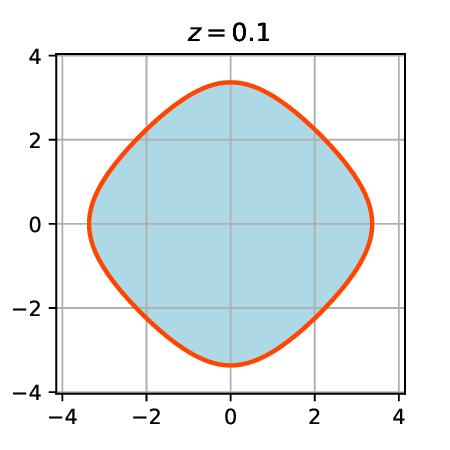}
\hspace{5mm}
\includegraphics[width=4cm, height=4cm]{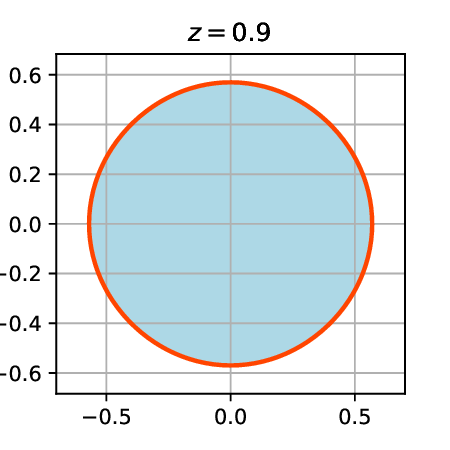}

\vspace{5mm}

\includegraphics[width=4cm, height=4cm]{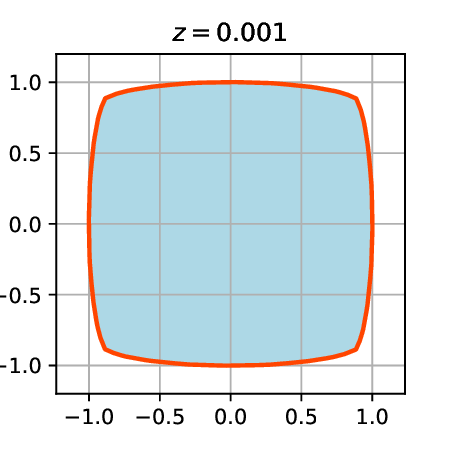}
\hspace{5mm}
\includegraphics[width=4cm, height=4cm]{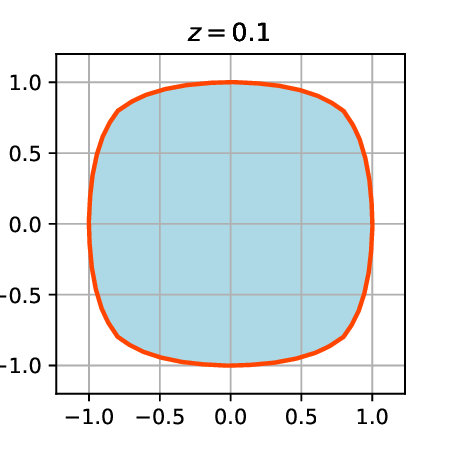}
\hspace{5mm}
\includegraphics[width=4cm, height=4cm]{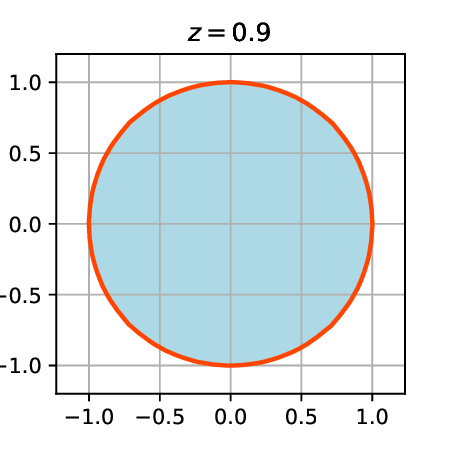}

\caption{
Wulff shape $\overline \Omega_z$ (upper) and
unit ball of $|\cdot|_z$ (lower) for the $\ell^\infty$ random walk in dimension $d=2$.
}
\label{fig:infty}}
\end{figure}

\subsubsection{$\ell^\infty$ random walk}

We now consider the kernel
\begin{equation}
D(x) = \frac 1 {3^d}\1\{\|x\|_\infty \le 1\} .
\end{equation}
By definition, $\hat D\supmu(0) = \frac{1}{3^d} \prod_{j=1}^d (1+2\cosh \mu_j)$.
By setting $\mu=\mz e_1$ in the first equation of \eqref{eq:Lagrange}, we find that $\mz$ is given by the positive solution of
\begin{equation} \label{eq:mass_infty}
    \cosh \mz  = \frac{ 3-z}{2z} .
\end{equation}

For $x\ne 0$,
to solve the system \eqref{eq:Lagrange},
we insert the first equation into the second to obtain
\begin{equation} \label{eq:Lagrange_infty}
\prod_{j=1}^d ( 1 + 2 \cosh \mu_j ) = \frac{3^d}{z},
\qquad
\lam x_j = \frac{ \grad_j \hat D \supmu(0) } { \hat D \supmu(0) }
	=  \frac{2\sinh \mu_j}{1+2\cosh \mu_j} .
\end{equation}
Since $\sinh \mu_j = \sqrt{ (\cosh \mu_j)^2 - 1 }$, the second equation
in \eqref{eq:Lagrange_infty} can be solved using the quadratic formula. The solution is
\begin{equation}
2\cosh \mu_j = \frac{ (\lam x_j)^2 + \sqrt{4 - 3 (\lam x_j)^2} }{1 - (\lam x_j)^2 } ,
\end{equation}
which can be inserted into the first equation of \eqref{eq:Lagrange_infty} to give
\begin{equation} \label{eq:lam_infty}
\prod_{j=1}^d  \frac{ 1+ \sqrt{4 - 3 (\lam x_j)^2} }{1 - (\lam x_j)^2 } = \frac{3^d}{z} .
\end{equation}
Since the left-hand side is strictly increasing in $\lam$, there is a unique solution $\lam = \lam_{z,x}$.
By Corollary~\ref{cor:norm}, we obtain
\begin{equation}
|x|_z = \frac{\mu_{\hat x, z}\cdot x}{\mz}
= \frac 1 { \mz  } 	\sum_{j=1}^d x_j
	\arccosh \biggl(
	\frac{ (\lam_{z,x} x_j)^2 + \sqrt{4 - 3 (\lam_{z,x} x_j)^2} }{2[1 - (\lam_{z,x} x_j)^2] }
	\biggr) .
\end{equation}
This norm converges to the $\ell^2$ norm as $z\to 1$ by Theorem~\ref{thm:critical_limit},
and it converges to the $\ell^\infty$ norm as $z \to 0$ by Theorem~\ref{thm:massive_limit}.
We believe the norm interpolates between the two limits monotonically (see Figure~\ref{fig:infty}), but we do not have a proof.

\subsubsection{On monotonicity of the norm}

Motivated in part by the above two examples,
we conjecture a sufficient condition for monotonicity of the norm.

\begin{conjecture}
Let $d\ge 1$.
Suppose $D(x)$ has finite support $\Ucal$.
If $D(x)$ is uniformly distributed on $\Ucal$,
and if every $x\in \Ucal$ is on the boundary of $\conv(\Ucal)$, then the map $z\mapsto |x|_z$ is monotone for each $x$.
\end{conjecture}

The monotonicity can fail when $D(x)$ is not a uniform distribution, as a small perturbation to $D$ can drastically change $\conv(\Ucal)$ and the $z\to 0$ limit of $|\cdot |_z$. The next example illustrates this with a perturbation of the simple random walk.

\begin{example} \label{ex:non_monotone}
Let $d=2$ and let $\alpha \in (0,1]$ be a small parameter.
We consider
\begin{equation}
D(x) = \frac{ 1 - \alpha } 4 \1\{ \norm x_1 = 1 \}
	+ \frac \alpha 4 \1\{ x = (\pm 1, \pm 1) \} .
\end{equation}
Since $\alpha > 0$,
we know its norm converges to the $\ell^\infty$ norm (not the $\ell^1$ norm) as $z \to 0$ by Theorem~\ref{thm:massive_limit},
and to the $\ell^2$ norm as $z\to 1$ by Theorem~\ref{thm:critical_limit}.
Using $\Zd$ symmetries, we solved the system \eqref{eq:Lagrange} and computed explicitly $|(1,1)|_z$ as a function of $\alpha,z$.
Numerical results show that, when $\alpha \in (0,0.17]$, there exists $z \in (0,1)$ for which
\begin{equation}
|(1,1)|_z > \sqrt 2 = \max \{ \norm{ (1,1)  }_\infty, \norm{ (1,1) }_2 \} .
\end{equation}
The map $z\mapsto |(1,1)|_z$ is therefore not monotone for these values of $\alpha$,
as can be discerned from Figure~\ref{fig:non_monotone}.
\end{example}

\begin{figure}[ht]
\centering{
\includegraphics[width=4cm, height=4cm]{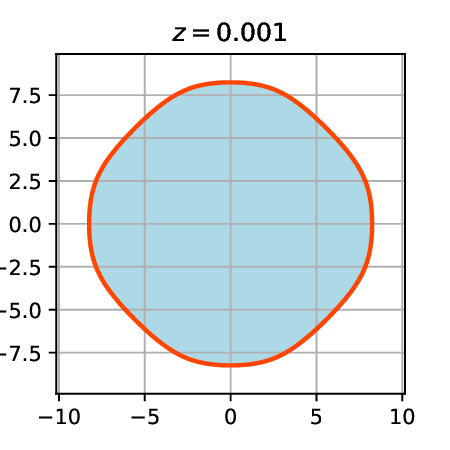}
\hspace{5mm}
\includegraphics[width=4cm, height=4cm]{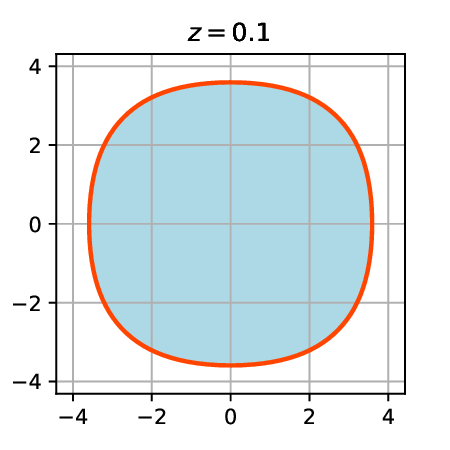}
\hspace{5mm}
\includegraphics[width=4cm, height=4cm]{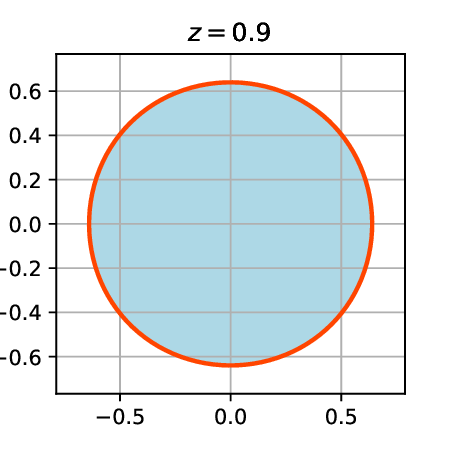}

\vspace{5mm}

\includegraphics[width=4cm, height=4cm]{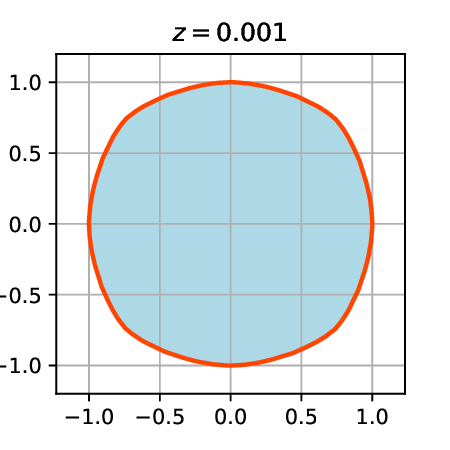}
\hspace{5mm}
\includegraphics[width=4cm, height=4cm]{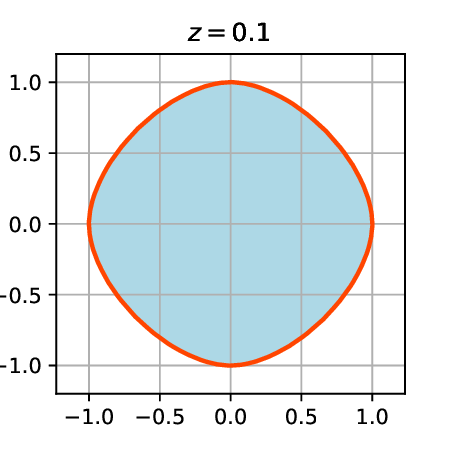}
\hspace{5mm}
\includegraphics[width=4cm, height=4cm]{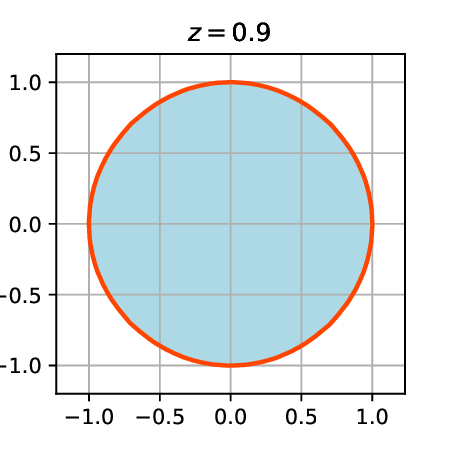}

\caption{
Wulff shape $\overline \Omega_z$ (upper) and
unit ball of $|\cdot|_z$ (lower) for Example~\ref{ex:non_monotone} with $\alpha = 0.05$. The map $z \mapsto |(1,1)|_z$ is not monotone.
}
\label{fig:non_monotone}}
\end{figure}

\subsection{Failure of OZ decay: saturation phenomenon}
\label{sec:counterx}

It is natural to expect that the Green function of any
random walk with rapidly decaying
transition probabilities should exhibit OZ decay.
However, this is false.  It is a discovery of
\cite{AIOV21,AIOV21PRE,AOV23,AOV24} that rapidly decaying
infinite-range models can fail to satisfy OZ decay
when $z$ is small. This failure is due to a
\emph{saturation} phenomenon, in which the mass $m_z$ becomes constant for small $z$.
A simple example of a $1$-dimensional random walk, adapted from \cite{AIOV21PRE}, is presented below.
A much broader discussion of saturation phenomena can be found in \cite{AIOV21,AIOV21PRE,AOV23,AOV24}.
In particular, for the Potts model, the precise decay of $S_z$ in the saturation regime (not OZ decay) is the subject of \cite{AOV23}.

\begin{proposition} \label{prop:xex}
Let $d=1$ and $p>  1$.
Consider the kernel
\begin{equation} \label{eq:def_Dsat}
D(x) = \frac{c_p}{|x|^p}e^{-|x|} \1\{x\neq 0\} ,
\end{equation}
where $c_p$ is the normalising constant.
Then:
\begin{enumerate}
\item[(i)] \cite{AIOV21PRE}
There exists $z_0 > 0$ such that for all $z \in (0, z_0)$,
\begin{equation} \label{eq:xexG}
zD(x) \le S_z(x) \le 2 zD(x)
\qquad(x\ne 0).
\end{equation}

\item[(ii)]
Let $z_{\rm sat} = 1/ \sum_{x\in \Z} D(x) e^{x}  \in (0,1)$.
For every $z \in (z_{\rm sat}, 1)$,
Assumptions~\ref{ass:Omega} and~\ref{ass:J} hold.
If $p > 3$,
then Assumptions~\ref{ass:Omega} and~\ref{ass:J} also hold for $z=z_{\rm sat}$.
\end{enumerate}
\end{proposition}

In particular, for $z \in (0,z_0)$, \eqref{eq:xexG} implies $S_z(x) \asymp z \abs x^{-p} e^{-\abs x}$;
this is different from the OZ decay which would be $|x|^{-(d-1)/2}e^{-m_z|x|}=e^{-m_z|x|}$ for $d=1$.
In contrast,
for $z \in (z_{\rm sat}, 1)$  (including $z_\sat$ if $p>3$), Corollary~\ref{cor:OZ} implies that $S_z(x)$ does satisfy the OZ decay.

\begin{proof}[Proof of Proposition~\ref{prop:xex}(i)]
Let $x \neq 0$, so $S_z(x) = 0 + \sum_{n=1}^\infty (zD)^{*n}(x)$.
The $n=1$ term gives the desired lower bound.
For the upper bound, we use the triangle inequality to estimate
\begin{equation}
    (zD)^{*n}(x)
    = z^n \sum_{y_1+\cdots +y_n=x} D(y_1)\cdots D(y_n)
    \le (zc_p)^n e^{-|x|}\sum_{\substack{y_1+\cdots +y_n=x \\ y_i \neq 0}} \prod_{i=1}^n \frac{1}{|y_i|^p}.
\end{equation}
At the cost of a factor $n$ in the sum, we may assume that $|y_1|$ is
the maximum of the $|y_i|$. Summation over $n$ then gives
\begin{equation}  \label{eq:xex-1}
S_z(x) \le zD(x) + e^{-|x|}\sum_{n=2}^\infty (zc_p)^n
    n\sum_{\substack{y_1+\cdots +y_n=x  \\y_i\neq 0 \\ |y_1|=\mathrm{max}|y_i|}}
    \biggl( \prod_{i=1}^n \frac{1}{|y_i|} \biggr)^p.
\end{equation}
We use the fact that if $s \ge t > 0$ then
\begin{equation}
    \frac{1}{s}  \le  \frac{2}{s+t} .
\end{equation}
Applying the inequality $n-1$ times,
with $s=|y_1|+\cdots+|y_{i-1}|$ and $t=|y_i|$ for $i=2,\ldots,n$, we get
\begin{equation}
    \prod_{i=1}^n \frac{1}{|y_i|}
    \le
    \frac{2^{n-1}}{|y_1|+\cdots+|y_n|}\prod_{i=2}^n \frac{1}{|y_i|}
    \le
    \frac{1}{|x|}\prod_{i=2}^n \frac{2}{|y_i|}.
\end{equation}
The large term on the right-hand side of \eqref{eq:xex-1} is then bounded above by
\begin{equation}
zD(x)
    \sum_{n=2}^\infty (zc_p)^{n-1} n
    \sum_{y_2,\ldots, y_n \neq 0}
    \prod_{i=2}^n \frac{2^p}{|y_i|^p}
    =
    zD(x) \sum_{n=2}^\infty  n
    \biggl(zc_p \sum_{y\neq 0}\frac{2^p}{|y|^p}\biggr)^{n-1}.
\end{equation}
Since $p>1=d$, the series in $n$ is convergent and is $O(z)$ when $z$ is small. We pick $z_0$ sufficiently small so that the series is $\le 1$ for all $z\le z_0$. Then from \eqref{eq:xex-1} we get $S_z(x) \le 2zD(x)$, as desired.
\end{proof}

Consistent with the failure of OZ decay when $z < z_0$, we can show that $S_z$ fails to satisfy Assumption~\ref{ass:Omega} when  $z < z_\sat$.
Let $z\in (0,1)$, $p>1$, and $\mu \in \R$.
By tilting and summing $S_z = \sum_{n=0}^\infty (zD)^{*n}$, Fubini's theorem gives
\begin{equation} \label{eq:chi_Dsat}
\chi\supmu(z)
= \sum_{n=0}^\infty  \biggl( z \sum_{x\in \Z} D(x) e^{\mu x} \biggr)^n .
\end{equation}
Since
\begin{equation}
\sum_{x\in \Z} D(x) e^{\mu x}
= \sum_{x\neq 0} \frac{c_p}{|x|^p}e^{-|x|+\mu x} < \infty
	\quad \text{if and only if} \quad
	\abs \mu \le 1 ,
\end{equation}
and since $z_{\rm sat} \sum_{x\in \Z} D(x) e^{x} = 1$ by definition,
when $z < z_\sat$ we have
\begin{equation}
z \sum_{x\in \Z} D(x) e^{\mu x}  \;
\begin{cases}
< 1		&(\abs \mu \le 1) \\
= \infty	&(\abs \mu > 1) .
\end{cases}
\end{equation}
This implies that $\Omega_z = [1,1]$, which is not an open set, violating Assumption~\ref{ass:Omega}.

\begin{proof}[Proof of Proposition~\ref{prop:xex}(ii)]
Let $z \in [ z_\sat,1)$.
Using the exact formula \eqref{eq:chi_Dsat} for $\chi\supmu(z)$,
we directly have $\Omega_z = (-\mz, \mz)$ where $\mz \in (0,1]$ is given implicitly by the equation
\begin{equation}
z \sum_{x\in \Z} D(x) e^{\mz x} = 1 .
\end{equation}
In particular, $\Omega_z$ is open and Assumption~\ref{ass:Omega} holds.
Observe also that $m_{z_\sat} = 1$ if and only if $z = z_\sat$.

For Assumption~\ref{ass:J},
we prove the following stronger statement:
For any $\eps \in (0,1)$ and any $\zeta > 0$,
there are constants $M, \KIR$ such that
$(zD\supmu, \delta_0\supmu) \in \Qcal_{M,\KIR,\zeta}$ uniformly in $\mu\in \overline\Omega_z$ and in $z \in [z_{\rm sat}+\eps, 1)$.
The infrared bound in fact is uniform over $z \in [z_\sat, 1)$,
as can be seen
by setting $\eps = z_\sat$ in \eqref{eq:RW_infrared_pf}.
For the moments, using $z \le 1$
and $\abs \mu \le \mz \le m_{z_\sat + \eps}$ we have
\begin{equation}
    \bignorm{ |y|^{2+\zeta} zD\supmu(y) }_1
    \le
    c_p\sum_{y \neq 0} \frac{|y|^{2+\zeta}}{|y|^p}e^{-(1-m_{z_\sat + \eps})|y|} ,
\end{equation}
which is finite since $m_{z_\sat + \eps} < 1$.

Lastly, if $p > 3$ we can verify the moment condition with $\eps = 0$ and any $\zeta \in (0,  p -3)$, since the sum above converges due to $ p - (2 + \zeta ) > 1 = d$.
This concludes the proof.
\end{proof}

\section{Geometry of the Wulff shape}
\label{sec:rate}

In this section, we prove Proposition~\ref{prop:geom-intro},
which provides key properties of the vector $\eta_{\hat x}$, the matrix $\Lam_{\hat x}$, and the vector $\mux$ that ultimately give the rate of exponential decay for $\GL(x)$.
We begin in Section~\ref{sec:Omega} by deriving geometric properties of
the Wulff shape $\overline\Omega$, under Assumptions~\ref{ass:Omega} and \ref{ass:J}.
Then in Section~\ref{sec:geom_pf} we prove Proposition~\ref{prop:geom-intro}.
Finally, in Section~\ref{sec:Wulff}, we justify the terminology ``Wulff shape'' for $\overline\Omega$.

\subsection{Properties of the set $\Omega$}
\label{sec:Omega}

Recall the notation $\GL\supmu(x) =\GL(x)e^{\mu\cdot  x}$ for exponentially tilting by a vector $\mu \in \Rd$.
Also recall the definition
\begin{equation}
\Omega = \Bigl\{\mu\in \R^d : \chi\supmu = \sum_{y\in \Zd} S\supmu(y) < \infty \Bigr\} .
\end{equation}

\begin{lemma} \label{lem:strictly_convex}
Let $d\ge 1$.
Suppose that Assumptions~\ref{ass:Omega} and~\ref{ass:J}
hold, but without requiring that $\g\supmu,J\supmu$ satisfy
\eqref{eq:Q_d-1}.
Then
\begin{align} \label{eq:J<1}
\Omega &\subset \Bigl\{ \mu \in \Rd :
	J\supmu   \in L^1(\Zd),\;
	\sum_{y\in \Zd} J \supmu(y) < 1 \Bigr\} ,
\\
\label{eq:OmegaJ}
    \partial \Omega &\subset
     \Bigl\{ \mu \in \Rd :
	J\supmu   \in L^1(\Zd),\;
	\sum_{y\in \Zd} J \supmu(y) = 1 \Bigr\} ,
\end{align}
the set $\overline \Omega$ is strictly convex,
and its boundary $\del \Omega$ is $C^2$.
\end{lemma}

\begin{proof}
We first prove \eqref{eq:J<1}.
Suppose $\mu\in\Omega$, so that $\chi \supmu < \infty$.
The function $S$ cannot be identically zero since
$\chi \supk {\mu_1 e_1} = \infty$ in Assumption~\ref{ass:Omega}, so
$\chi\supmu >0$.  By Assumption~\ref{ass:J},
$J\supmu   \in L^1(\Zd)$.
We multiply the OZ equation \eqref{eq:OZeq} by $e^{\mu\cdot x}$ and sum over $x \in \Zd$, to obtain
\begin{equation} \label{eq:chisum}
    \chi\supmu = \hat \g \supmu(0) + \hat J \supmu(0) \chi\supmu.
\end{equation}
Since $\hat \g \supmu(0) > 0$ by Assumption~\ref{ass:J}, the above
equation can only hold if $\hat J \supmu(0) <1$.
This proves \eqref{eq:J<1}.
Also, \eqref{eq:chisum} then implies the identity
\begin{equation} \label{eq:chi^mu}
    \chi^{(\mu)} = \frac{\hat \g^{(\mu)}(0)}{1-\hat J^{(\mu)}(0)}
    \qquad (\mu \in \Omega) .
\end{equation}

Next we prove \eqref{eq:OmegaJ}.
Let $\mu_\infty\in \partial\Omega$.  Since $\mu_\infty \not\in \Omega$, we have $\chi\supk {\mu_\infty}= \infty$.
There must be a sequence $\{ \mu_n \}_n \subset \Omega$ that converges to $\mu_\infty$.
By passing to a subsequence, we can assume each of its component sequences $\{ \mu_{ n,j} \}_n \subset \R$, $1\le j \le d$, takes only one sign ($\ge0$ or $\le 0$) and is monotone.
We define another sequence  $\{ \tilde \mu_n \}_n \subset \R^d$ by
\begin{equation}
\tilde \mu_{n,j} = \begin{cases}
\mu_{\infty,j} 	& \text{if $\abs{  \mu_{n,j}  } \ge \abs{ \mu_{\infty,j} }$}  \\
\mu_{n,j}	 	& \text{otherwise.}
\end{cases}
\end{equation}
Then $\tilde \mu_n \to \mu_\infty$ too, and
$\{ \abs{ \tilde \mu_{ n,j} } \}_n$ is increasing in $n$ for each $j$.
Also, using \eqref{eq:chisym}, we have
\begin{equation}
\chi\supk { \tilde \mu_n }
= \sum_{y\in \Z^d} \GL(y)  \prod_{j=1}^d \cosh( \tilde\mu_{n,j} y_j)
\le \sum_{y\in \Z^d} \GL(y)  \prod_{j=1}^d \cosh( \mu_{n,j} y_j)
< \infty ,
\end{equation}
so $\tilde \mu_n \in \Omega$ for all $n$.
We now apply \eqref{eq:chi^mu} to $\tilde \mu_n$ and take the limit as $n \to \infty$. By monotone convergence,
\begin{equation}
\lim_{n\to \infty}  \frac{\hat \g^{(\tilde \mu_n)}(0)}{1-\hat J^{(\tilde \mu_n)}(0)}
= \lim_{n\to \infty} \chi\supk {\tilde \mu_n}
= \chi\supk {\mu_\infty}
= \infty.
\end{equation}
Since $\abs{ \hat \g^{(\tilde \mu_n)}(0) }
\le \sup_{\mu \in \Omega} \norm{ \g\supmu(y)}_1 < \infty $,
the continuity of the map $\mu \mapsto \hat J\supmu(0)$
gives $\hat J \supk {\mu_\infty} (0)
=  \lim_{n\to \infty} \hat J^{(\tilde \mu_n)} (0) = 1$.
This proves \eqref{eq:OmegaJ}.

The set inclusion \eqref{eq:OmegaJ} implies that
the boundary is $C^2$, because the map $\mu \mapsto \sum_{y\in \Zd} J \supmu(y)$ is $C^2$ in $\overline \Omega$, by the second moment estimate for $J\supmu(y)$ in \eqref{eq:Q_moments}.

Finally,
for the strict convexity of $\overline \Omega$,
we need to show that $\del \Omega$ contains no line segments.
The map $\mu \mapsto \sum_{y\in \Zd} J \supmu(y)$ is uniformly convex, because the tilted infrared bound
in \eqref{eq:Q_infrared}
implies a uniform lower bound on all second derivatives.
Therefore, the level sets of $\mu \mapsto \sum_{y\in \Zd} J \supmu(y)$ cannot contain any line segments (along which the second derivative would be zero).
It then follows from \eqref{eq:OmegaJ} that $\del \Omega$ does not contain any line segments, and the proof is complete.
\end{proof}

\subsection{Proof of Proposition~\ref{prop:geom-intro}}
\label{sec:geom_pf}

We restate Proposition~\ref{prop:geom-intro}, for convenient reference,
as the next proposition.

\begin{proposition} \label{prop:geom}
Let $d\ge 1$.
Suppose that Assumptions~\ref{ass:Omega} and~\ref{ass:J}
hold, but without requiring that $\g\supmu,J\supmu$ satisfy
\eqref{eq:Q_d-1}.
Then the set $\overline \Omega$ is strictly convex, and every $\mu \in \overline \Omega$ obeys $\abs \mu^2 \le M/\KIR$.
For each $\mu \in \overline\Omega$, we have
\begin{equation} \label{eq:J1}
    \sum_{y\in \Zd} J \supmu(y) \le 1
    \quad \text{with equality if and only if $\mu\in\partial\Omega$.}
\end{equation}
For each nonzero $ x \in \Rd$,
there is a unique vector $\mux \in \del \Omega$, depending only on the direction $\hat x = x / \abs x$ of $x$, such that
\begin{equation}
\label{eq:optimal-mu}
\mu_{\hat x}\cdot x = \max_{\mu \in \overline\Omega }\, \mu\cdot x.
\end{equation}
The vector $\eta_\hatx \in \Rd$ with components
\begin{equation} \label{eq:eta_cond}
\eta_{\hatx,j}
= \sum_{y\in \Zd} y_j J(y) e^{\mux \cdot y}
\end{equation}
satisfies $\eta_{\hat x} = |\eta_{\hat x}|\hat x$.
Also,
with the matrix $\Lam_\hatx$ defined by
$\Lambda_{\hat x, j l } = \sum_{y\in \Zd} y_j y_l J(y) e^{\mux \cdot y}$,
we have
\begin{equation} \label{eq:crossover-bds}
 \mu_\hatx \cdot x  \asymp \mS  \abs x ,
	\qquad
|\eta_{\hat x}| \asymp \mS ,
	\qquad
k \cdot \Lam_\hatx k \asymp \abs k^2
	\quad (k\in \Rd) ,
\end{equation}
with constants depending only on $d,M,\KIR$.
\end{proposition}

\begin{remark}
The two equations (consequences of \eqref{eq:J1} and \eqref{eq:eta_cond})
\begin{equation} \label{eq:J_cond-intro}
\hat J \supmux(0)
= 1,
\qquad
\nabla_\mu  \hat J \supmux(0) = \lambda x
\end{equation}
are exactly the conditions for a critical point of the Lagrangian
(cf.\ \eqref{eq:Lagrange} for random walks),
but we are working in a more general setting where $\hat J\supmu$ need not be defined for all $\mu \in \Rd$.
\end{remark}

\begin{proof}[Proof of Proposition~\ref{prop:geom}]
We have already proved the strict convexity
and \eqref{eq:J1} in Lemma~\ref{lem:strictly_convex}.
For the bound $\abs\mu^2 \le M/\KIR$, we observe that for any $\mu \in \overline \Omega$, \eqref{eq:Q_moments} and \eqref{eq:Q_infrared} give
\begin{equation}
M \ge  \sum_{y\in \Zd} \abs{ J(y) } \cosh(\mu \cdot y)
\ge  \sum_{y\in \Zd}  J(y)  \frac{ (\mu \cdot y)^2 } 2
\ge \KIR \abs\mu^2 .
\end{equation}

For the existence of the optimal $\mux$,
let $0 \ne x \in \Rd$.
We want to maximise the linear functional $\mu \cdot x$ over $\mu \in \overline \Omega$.
We know the set $\overline \Omega$ is convex and bounded by Lemma~\ref{lem:Omega_convex}, so a maximiser exists on $\del \Omega$.
The maximiser is unique because $\overline \Omega$ is strictly convex by Lemma~\ref{lem:strictly_convex}.
Also, it depends on the direction $\hat x = x / \norm x_2$ only, because maximising $\mu \cdot x$ is equivalent to maximising $\mu \cdot \hat x$.

To prove  that $\eta_{\mux}$ defined by \eqref{eq:eta_cond} points in the same direction
as $\hat x$, we argue as follows.
Since $\mux$ is a maximiser,
the tangent vectors to $\del \Omega$ at $\mux$ (which exist because $\del \Omega$ is $C^2$ by Lemma~\ref{lem:strictly_convex}) must be contained in a level set of $\mu \cdot x$.
This implies the outward-pointing normal vector of $\del \Omega$ at $\mux$ to be in the same direction as $x$.
Using the inclusions \eqref{eq:J<1}--\eqref{eq:OmegaJ}, we can compute this normal vector by computing the gradient of the map $\mu \mapsto \sum_{y\in \Zd} J(y) e^{\mu \cdot y}$ at $\mu = \mux$, which gives exactly $\eta_\hatx$.
Therefore, we can write $\eta_\hatx = \norm{ \eta_\hatx}_2 \hat x$, as desired.

Finally, we prove the bounds \eqref{eq:crossover-bds}.
The bound $k \cdot \Lam_\hatx k \asymp \abs k^2$
has been shown to follow from Assumption~\ref{ass:J} below the statement of Definition~\ref{def:A}.
For $\mux \cdot x$, using Lemma~\ref{lem:Omega_convex} we have $\mux \cdot x \le \abs{ \mux } \abs x \lesssim \mS  \abs x$.
By choosing $a > 0$ such that $a\hat x \in \del\Omega$, the optimality of $\mux$ gives the opposite inequality
\begin{equation}
\mux \cdot x
\ge (a\hatx) \cdot x
= \abs{ a \hat x } \abs{ x }
\gtrsim \mS  \abs x .
\end{equation}
For the upper bound on $\abs{ \eta_\hatx}$, we use $\abs{\sinh t} \le \abs t\cosh t$ to get
\begin{equation}
\abs{ \eta_{\hatx} }
\le \sum_{y\in \Zd} \abs{y} \abs{ J(y) }
	\abs{ \mux \cdot y } \cosh(\mux \cdot y)
\le \abs {\mux}  \sum_{y\in \Zd} \abs{y}^2 \abs{ J \supmux(y) }
\le \abs {\mux} M
\lesssim \mS  .
\end{equation}
For the lower bound on $\abs{ \eta_\hatx}$,
we observe that, for any $t\in [0,1]$,
\begin{equation}
\frac{ \D }{ \D t }
	\sum_{y\in \Zd} (\mux \cdot y) J(y) \sinh( t \mux \cdot y )
= \sum_{y\in \Zd} (\mux \cdot y)^2 J(y) \cosh( t \mux \cdot y )
\ge \KIR \abs{ \mux }^2
\end{equation}
by the infrared bound \eqref{eq:Q_infrared} of $J\supk{ t\mux }$ and \eqref{eq:quadratic_form}.
Integrating the above from $t=0$ to $1$, we get
\begin{equation}
\abs{ \mux } \abs{ \eta_\hatx}
\ge \mux \cdot \eta_\hatx
\ge \mux \cdot 0 + \KIR \abs{ \mux }^2 ,
\end{equation}
so $\abs{ \eta_\hatx} \ge \KIR \abs \mux  \gtrsim \mS $.
This completes the proof.
\end{proof}

\subsection{Wulff shape and norm duality}
\label{sec:Wulff}

We now justify the terminology ``Wulff shape'' for $\overline \Omega$.

\begin{proposition}
Let $d\ge 1$.
Suppose that Assumptions~\ref{ass:Omega} and~\ref{ass:J}
hold, but without requiring that $\g\supmu,J\supmu$ satisfy
\eqref{eq:Q_d-1}.
For nonzero $x \in \Rd$,
define its \emph{surface tension} by $\sigma(x) = \mS  \abs x_S
= \mu_\hatx \cdot x $.
Then the \emph{Wulff shape} for $\sigma(x)$, defined by
\begin{equation}
\label{eq:Wulff-def}
\Wcal
= \bigl\{ \mu \in \Rd    :     \mu \cdot x \le  \sigma(x)
	\text{ for all nonzero } x\in \Rd   \bigr\} ,
\end{equation}
coincides with $\overline \Omega$.
\end{proposition}

\begin{proof}
Since $\mu_\hatx$ is defined to be the maximiser of the optimisation problem \eqref{eq:optimal-mu}, we have $\overline \Omega \subset \Wcal$ by definition.
On the other hand, if $\mu' \not \in \overline \Omega$,
then by the hyperplane separation theorem (using that $\overline \Omega$ is closed and convex)
there is a vector $x\in \Rd$ for which
\begin{equation}
\mu' \cdot x >  \max_{\mu\in \overline \Omega}\, \mu \cdot x
= \mu_\hatx \cdot x
= \sigma(x) ,
\end{equation}
so $\mu' \not \in \Wcal$.
This proves $\Wcal \subset \overline \Omega$ and therefore we have
$\overline \Omega = \Wcal $.
\end{proof}

The definition \eqref{eq:Wulff-def} of the Wulff shape can be rewritten as
\begin{align}
\Wcal
&= \bigl\{ \mu \in \Rd    :     \mu \cdot \frac x {\abs x_S} \le  m_S
	\text{ for all nonzero } x\in \Rd   \bigr\}  \nl
&= \bigl\{ \mu \in \Rd    :     \mu \cdot y \le  m_S
	\text{ for all $y\in \Rd$ with } \abs y_S \le 1 \bigr\}  .
\end{align}
Therefore, $\overline \Omega = \Wcal$ is the ball of radius $\mS$ for the norm dual to $|\cdot|_S$; see Figures~\ref{fig:nn}--\ref{fig:non_monotone}.

\section{Proof of the noncentred Gaussian Lemma}
\label{sec:tilted_gaussian}

In this section, we prove Theorem~\ref{thm:deconv-intro}, which we restate here as
the following theorem.  Recall from \eqref{eq:def_Lambda-intro} the notation
\begin{equation} \label{eq:def_Lambda}
\eta_j = \sum_{y\in \Zd} y_j \Jsupmu(y),
\qquad
\Lambda_{j l } = \sum_{y\in \Zd} y_j y_l \Jsupmu(x) .
\end{equation}
As observed below Definition~\ref{def:A},
these obey $\abs {\eta_j}, \abs{ \Lam_{jl} } \le M$
and $k\cdot \Lambda k \asymp \abs k^2$
when $(\Jsupmu,g)\in \Qcal_{M,\KIR,\zeta}$.

\begin{theorem} \label{thm:deconv}
Let $d \ge 1$ and
let $M,\KIR,\zeta>0$ be given.
Suppose $(\Jsupmu,g)\in \Qcal_{M,\KIR,\zeta}$.
Let $\eta$ and $\Lambda$ be given by \eqref{eq:def_Lambda},
and let $x\in\Zd$ be a nonzero vector in the same direction as $\eta$,
\ie, $\eta = \abs \eta \hat x$ ($\eta=0$ is permitted here).
If $d \le 2$, we further assume $\abs x \abs \eta \ge \so$ for some $\so > 0$.
Then, as $\abs x \to \infty$,
\begin{equation} \label{eq:Q_int}
\GQ (x) :=
\int_0^\infty \D t    \int_\Td \frac{\D k }{(2\pi)^d}
	e^{ik\cdot x} \hat g(k)  e^{-t [ \hat \Jsupmu(0) - \hat \Jsupmu (k) ] }
=  \bg(x; \eta, \Lam) \Bigl[\hat{g}(0) + O\Bigl(\frac 1 {\abs x^\delta}\Bigr) \Bigr]
\end{equation}
with some $\delta > 0$
and with the constant depending only on $d, \zeta, M, \KIR, \so$.
\end{theorem}

We define
\begin{equation} \label{eq:G^mu}
I_t(x) =	
	\int_\Td \frac{\D k }{(2\pi)^d}
	e^{ik\cdot x} \hat g(k)  e^{-t [\hat \Jsupmu(0) - \hat \Jsupmu (k) ] } ,
\end{equation}
so that $\GQ(x)
= \int_0^\infty \D t\,  I_t(x)$.

\subsection{Proof of Theorem~\ref{thm:deconv}}

Our goal is to compute the asymptotic behaviour of the Fourier integral \eqref{eq:Q_int} as $\abs x \to \infty$.
We will prove that the leading behaviour of $\GQ(x)$
is dominated by the integral over small $k$, where $\hat \Jsupmu(0)-\hat \Jsupmu(k)$ is dominated
by its quadratic Taylor approximation and $\hat g(k)$ is close to $\hat g(0)$.

The hypothesis \eqref{eq:Q_moments} on the moments of $\Jsupmu$ permits Taylor expansion
of its Fourier transform $\hat \Jsupmu(k) = \sum_{y\in \Zd} \Jsupmu(y) e^{-ik\cdot y}$ up to the second order.
With $\eta$ and $\Lam$ defined by \eqref{eq:def_Lambda},
we have
\begin{equation} \label{eq:J_small_k}
\hat \Jsupmu(0)  - \hat \Jsupmu(k)
= i k \cdot \eta + \half k \cdot \Lambda k + \hat R_2(k) ,
\end{equation}
where $\hat R_2(k) = \sum_{x\in \Zd} \Jsupmu(x) [1 - ik\cdot x - \half (k\cdot x)^2 - e^{-i k \cdot x}] $ is twice differentiable and satisfies
\begin{equation} \label{eq:R2bds}
\frac{ \abs{ \hat R_2(k)} }{ \abs k^2 } , \;\;
\frac{ \abs{ \grad \hat R_2(k)} }{ \abs k } ,\;\;
\abs{ \grad^2 \hat R_2(k) }
	\lesssim \abs k^{\zeta} M
	\qquad(k \in \Td).
\end{equation}
Ignoring the remainder $\hat R_2$ for a moment,
inserting the Taylor polynomial of $\hat \Jsupmu(k)$ into \eqref{eq:G^mu} gives
\begin{equation}
I_t(x) \approx
    \hat g(0)
	\int_\Rd \frac{\D k }{(2\pi)^d}
	e^{ik\cdot x} e^{-t [i k \cdot \eta + \half k \cdot \Lambda k ] }
= \hat g(0) \rho_t(x; \eta, \Lam) ,
\end{equation}
where $\rho_t(x; \eta, \Lam)$ is
the heat kernel of the Brownian motion on $\Rd$ with drift $\eta$ and covariance $\Lambda$, defined in \eqref{eq:def_rho}.
We therefore expect and aim to exploit $\GQ(x) \approx
\hat g(0) \int_0^\infty \D t\, \rho_t(x;\eta, \Lam) =\hat g(0) \bg(x;\eta,\Lam)$.

A more careful analysis is done by writing
\begin{equation} \label{eq:GLSR}
\GQ(x)
= \hat g(0) \bg(x;\eta,\Lam) +  \int_0^\infty \D t\, [I_t(x) - \hat g(0) \rho_t(x;\eta,\Lam)] .
\end{equation}
We will show the integral on the right-hand side is small
relative to $\C(x;\eta,\Lam)$.
We introduce a small parameter $\eps > 0$ and split the above integral at time $T$ defined by
\begin{equation} \label{eq:def_T}
T = \frac{ \eps \abs x }{  (\so\inv\abs \eta) \vee \abs x \inv } .
\end{equation}
The constant $\so$ is only relevant for dimensions $d\le 2$; we simply take $\so = 1$ if $d > 2$.
We do not need any cancellation between $I_t$ and $\hat g(0)\rho_t$ when $t \le T$,
since the next lemma shows that their contributions are individually small.
It is shown in Lemma~\ref{lem:C_bdds} that,
in the specific case where $\eta$ is a nonnegative multiple of $\hat x$, $\bg(x;\eta,\Lam)$ is bounded above and below by multiples of
\begin{equation}
\label{eq:Bdef}
    \Casy(x;\eta)=
\frac{ ( \abs \eta  \vee \abs x\inv )^{(d-3)/2} }{ \abs x^{(d-1)/2} } ,
\end{equation}
with constants depending only on $M,\KIR,s_0$.

\begin{lemma} \label{lem:small_time_intro}
Let $d \ge 1$, $x\in\Zd$, and $\eps \in (0, \half \so\inv]$.
For all $\abs x \ge 2 \sqrt d$, we have
\begin{equation}
\int_0^{T} \D t\, \abs{ I_{t}(x) } ,\;\;
\int_0^{T} \D t\, \rho_t(x;\eta,\Lam)
\lesssim \sqrt \eps  \Casy(x; \eta) ,
\end{equation}
with constants that depend only on $d, M, \KIR, \so$.
\end{lemma}

For $t \ge T$, we do harvest the similarities between $I_t$ and $\hat g(0)\rho_t$ when $x$ is in the same direction as $\eta$.
In view of the definition of $T$ in \eqref{eq:def_T}, the hypothesis on $T$ requires $\abs x $ to be large.

\begin{lemma} \label{lem:R3_integral_intro}
Let $d \ge 1$, $\eps \in (0,1) $, and denote
$\zeta_0 = \zeta \wedge \frac {d \wedge 2}2$.
Let $x\in\Zd$ be a nonzero vector in the same direction as $\eta$,
\ie, $\eta = \abs \eta \hat x$.
If $d \le 2$, we further assume $\abs x \abs \eta \ge \so$.
Then when $T \ge \eps\inv \vee \eps^{-(2+\zeta_0)/\zeta_0}$ and
$\abs x \ge 4dM / \so$, we have
\begin{equation}
\frac{ 1 }{ \Casy(x; \eta )}
	\int_T^\infty \D t\, \abs{ I_{t}(x)
	- \hat g(0) \rho_t(x; \eta,\Lam) }
\lesssim  e^{-c_1/\eps}  + \frac 1 { \eps^{d+1+\zeta_0} \abs x^{\zeta_0} }
	+ \frac 1 { \eps^{ (d\vee4) + \zeta_0} \abs x^{\zeta_0 / 2} } ,
\end{equation}
with $c_1 = \frac 1 8 \KIR$ and the implicit constant depending only on $d, \zeta, M, \KIR,\so$.
\end{lemma}

\begin{proof}[Proof of Theorem~\ref{thm:deconv} assuming Lemmas~\ref{lem:small_time_intro} and~\ref{lem:R3_integral_intro}]
Since $x\in\Zd$ is a nonzero vector in the same direction as $\eta$,
we have $\C(x;\eta,\Lam) \asymp \Casy(x;\eta)$ by Lemma~\ref{lem:C_bdds}.
Since $\zeta_0 >0$, it is possible to choose
$\beta \in (0,\half)$ sufficiently small so that we simultaneously have
\begin{equation} \label{eq:beta_T_cond}
(1-\beta)\zeta_0 > (2+\zeta_0) \beta
\end{equation}
and
\begin{equation} \label{eq:beta_gap_cond}
\beta(d+1+\zeta_0) < \zeta_0 ,\qquad
\beta[(d\vee4)+\zeta_0] < \half \zeta_0.
\end{equation}
We set $\eps = \abs x^{- \beta}$.
To verify the hypothesis of Lemma~\ref{lem:R3_integral_intro}, we use \eqref{eq:beta_T_cond} and $\abs \eta \le M$ to get
\begin{equation}
T =  \frac{ \abs x^{1-\beta} }{  (\so\inv\abs \eta) \vee \abs x \inv }
\gtrsim \abs x^{1-\beta}
\gg \abs x^{(2+\zetao)\beta / \zetao}
= \eps^{-(2+\zetao) / \zetao} .
\end{equation}
Similarly, using $\beta < \half$ we get
$T \gtrsim \abs x^{1-\beta} \gg \abs x^\beta = \eps\inv $ for large $\abs x$.
It then follows from \eqref{eq:GLSR} and Lemmas~\ref{lem:small_time_intro}--\ref{lem:R3_integral_intro} that
\begin{align}
\biggabs{  \frac{ \GQ(x) }{\bg(x; \eta,\Lam) }  - \hat g(0)  }
&\lesssim (1 + \abs{\hat g(0)}) \sqrt \eps   + e^{-c_1 / \eps}
	+ \frac 1 { \eps^{d+1+\zeta_0} \abs x^{\zeta_0} }
	+ \frac 1 { \eps^{ (d\vee4) + \zeta_0} \abs x^{\zeta_0 / 2} } \nl
&\le \frac{1 + \norm{g(y)}_1}{\abs x^{\beta/2}}  + e^{-c_1 \abs x^\beta }
	+ \frac 2 { \abs x^{\delta} }
\end{align}
for some $\delta > 0$ whose existence is due to \eqref{eq:beta_gap_cond}.
Since $\norm{ g(y) }_1 \le M$, picking the slowest decay gives the desired result.
\end{proof}

It remains to prove Lemmas~\ref{lem:small_time_intro}--\ref{lem:R3_integral_intro}.
We do so in Sections~\ref{sec:small-t-int}--\ref{sec:large-t-int},
assuming estimates for $I_{t}(x)$ whose proofs are deferred to Section~\ref{sec:Ilemmas}.

\subsection{Small time integral: proof of Lemma~\ref{lem:small_time_intro}}
\label{sec:small-t-int}

We now prove Lemma~\ref{lem:small_time_intro}, which estimates the integrals of $I_t(x)$ and $\rho_t(x;\eta,\Lam)$ from time $0$ to $T$.
We first note that the Brownian heat kernel defined in \eqref{eq:def_rho} satisfies the inequality
\begin{equation} \label{eq:rhobd}
\rho_t(x; \eta, \Lambda)
\lesssim \frac { 1 } {(\sqrt t)^d}
	\biggl( \frac{ \sqrt t } { \abs{ x -{t \eta} } } \biggr)^{n}
\qquad (n\ge 0),
\end{equation}
(here $n=0$ means the factor is absent)
by bounding the exponential by a polynomial
and by using \eqref{eq:quadratic_form}.
The following lemma establishes a similar upper bound for $I_t(x)$.
For a vector $y \in \Rd$, we write $\integer{y} \in \Zd$ for the lattice point closest to $y$ (breaking ties arbitrarily), and we write $\fractional{y}$ for the fractional part of $y$, defined by
\begin{equation}
y = \integer y + \fractional y.
\end{equation}
The fact that $1+\sqrt{t}$ appears in \eqref{eq:It_decay}, rather than
$\sqrt{t}$ as in \eqref{eq:rhobd}, is a reflection of the fact that $I_t(x)$
is finite for $t=0$ due to the finite integration domain in \eqref{eq:G^mu}.

\begin{lemma} \label{lem:It_decay}
Let $d\ge 1$ and $t > 0$.
For any integer $n \in [0, 2 \vee (d-1)]$, we have
\begin{equation} \label{eq:It_decay}
\abs{ I_{t} (x) }
	\lesssim \frac 1 {(1+\sqrt t  )^d}
	\biggl( \frac{ 1 + \sqrt t } { \abs{ x - \integer{t \eta} } } \biggr)^n
\end{equation}
(here $n=0$ means the factor is absent)
with a constant depending only on
$n, d, M, \KIR$.
\end{lemma}

The proof of Lemma~\ref{lem:It_decay} is deferred to Section~\ref{sec:Ilemmas}.
The centred ($\eta =0$) case of Lemma~\ref{lem:It_decay} was proved by Hara in \cite[Lemma~2.3]{Hara08},
assuming
the decay estimate $\abs{\Jsupmu(y)} \lesssim 1 / \abs y^{d+2}$ instead of the $(d-1)^{\rm th}$ moment of $\Jsupmu(y)$ in \eqref{eq:Q_d-1}. It is not clear how to extend Hara's strategy to noncentred $\Jsupmu(y)$.
We use a different strategy inspired by the use of weak derivatives in
\cite{LS24a}.

\begin{proof}[Proof of Lemma~\ref{lem:small_time_intro} assuming Lemma~\ref{lem:It_decay}]
Let $\eps \in (0, \half \so\inv]$ and $\abs x \ge 2 \sqrt d$.
We want to estimate the integrals of $I_t(x)$ and $\rho_t(x;\eta,\Lam)$ from time $0$ to $T$, with $T$ defined in \eqref{eq:def_T}.
In this regime,
we have  $\abs{ t\eta }
\le T \abs \eta
\le \eps \so \abs x
\le \half \abs x$, so
\begin{gather}
\abs{ x - t\eta } \ge  \abs x - \half \abs x = \half \abs x ,
\\
\abs{ x - \integer{t\eta} }
\ge \abs{ x - t\eta } - \abs{ t \eta - \integer{ t\eta } }
\ge \half \abs x - \frac{ \sqrt d } 2
\ge \frac 1 4 \abs x .
\end{gather}
Thus, the bounds \eqref{eq:rhobd} and \eqref{eq:It_decay} simplify to
\begin{equation} \label{eq:It_decay_simpler}
\rho_t(x; \eta, \Lambda)
	\lesssim \frac 1 { \abs x^n } \frac 1 { (\sqrt t)^{d-n} } ,
	\qquad
\abs{ I_{t} (x) }
	\lesssim \frac 1 { \abs x^n } \frac 1 { (1 + \sqrt t)^{d-n} } ,
\end{equation}
with $n \in [0, 2 \vee (d-1)]$.
The rest of the proof is simply integration of these upper bounds.
We consider $\rho_t(x; \eta, \Lambda)$, for which the upper bound is weaker.

In dimensions $d \ge 2$,
we integrate the $n = d-1$ case of \eqref{eq:It_decay_simpler}, to get
\begin{equation} \label{eq:small_time_pf}
\int_0^{T} \D t\,  \rho_t(x; \eta, \Lambda)
\lesssim \frac 1 { \abs x^{d-1} }
	\int_0^T \frac{ \D t } { \sqrt t }
= \frac { 2 \sqrt T} { \abs x^{d-1} } .
\end{equation}
In the critical regime $\abs \eta \abs x \le \so$,
we have $T = \eps \abs x^2$, so the above gives
\begin{equation}
\int_0^{T} \D t\,  \rho_t(x; \eta, \Lambda)
\lesssim \frac { \sqrt \eps } { \abs x^{d-2} }
\asymp_{\so} \sqrt \eps \Casy(x;\eta) .
\end{equation}
In the OZ regime $\abs \eta \abs x \ge \so$,
we have $T=\eps \so \abs x/\abs \eta$ and $\abs x \ge \so \abs \eta \inv$,
so \eqref{eq:small_time_pf} gives, as required,
\begin{equation}
\int_0^{T} \D t\,  \rho_t(x; \eta, \Lambda)
\lesssim  \sqrt {\eps \so}  \frac { \abs x^{1/2} } { \abs x^{d - 1} \abs \eta^{1/2} }
\le  \sqrt \eps \frac {  ( \so\inv \abs  \eta )^{(d-3)/2} } { \abs x^{(d-1)/2} }
\asymp_{\so} \sqrt \eps \Casy(x;\eta) .
\end{equation}

Finally, for dimension $d=1$, we integrate the $n=1$ case of \eqref{eq:It_decay_simpler}.
This gives
\begin{equation}
\int_0^{T} \D t\,  \rho_t(x; \eta, \Lambda)
\lesssim \frac 1 { \abs x }
	\int_0^T \D t
= \frac { T} { \abs x }
= \frac \eps { (\so\inv \abs \eta) \vee \abs x\inv }
\asymp_{\so} \eps  \Casy(x;\eta) ,
\end{equation}
and the desired result follows from $\eps \le \sqrt{ \eps / (2\so) }$.
This concludes the proof.
\end{proof}

\subsection{Large time integral: proof of Lemma~\ref{lem:R3_integral_intro}}
\label{sec:large-t-int}

We now prove Lemma~\ref{lem:R3_integral_intro}, which we restate here as Lemma~\ref{lem:R3_integral-2}.

\begin{lemma} \label{lem:R3_integral-2}
Let $d \ge 1$, $\eps \in (0,1) $, and denote
$\zeta_0 = \zeta \wedge \frac {d \wedge 2}2$.
Let $x\in\Zd$ be a nonzero vector in the same direction as $\eta$,
\ie, $\eta = \abs \eta \hat x$.
If $d \le 2$, we further assume $\abs x \abs \eta \ge \so$.
Then when $T \ge \eps\inv \vee \eps^{-(2+\zeta_0)/\zeta_0}$ and
$\abs x \ge 4dM / \so$, we have
\begin{equation} \label{eq:large_time_int}
\frac{ 1 }{ \Casy(x; \eta )}
	\int_T^\infty \D t\, \abs{ I_{t}(x)
	- \hat g(0) \rho_t(x; \eta,\Lam) }
\lesssim  e^{-c_1/\eps}  + \frac 1 { \eps^{d+1+\zeta_0} \abs x^{\zeta_0} }
	+ \frac 1 { \eps^{ (d\vee4) + \zeta_0} \abs x^{\zeta_0 / 2} } ,
\end{equation}
with $c_1 = \frac 1 8 \KIR$ and the implicit constant depending only on $d, \zeta, M, \KIR,\so$.
\end{lemma}

The next lemma provides a bound on the difference $ I_{t}(x)
	-  \hat g(0)  \rho_t(x; \eta,\Lam)$ when $t$ is large.
A similar result with $n=0$ and $\eta=0$ was obtained in \cite[Lemma~2.2]{Hara08}.
We defer the proof of Lemma~\ref{lem:It_asymp-2} to Section~\ref{sec:Ilemmas}.

\begin{lemma} \label{lem:It_asymp-2}
Let $d \ge 1$, $\eps \in (0,1] $, and denote
$\zeta_0 = \zeta \wedge \frac {d \wedge 2}2$.
For any $t \ge \eps\inv \vee \eps^{-(2+\zeta_0)/\zeta_0}$ and any $x\in \Zd$, we have
\begin{equation}  \label{eq:large_t_error-2a}
\bigabs{ I_{t}(x) - \hat g(0) \rho_t(x; \eta, \Lam) }
\le C_1 \biggl[ e^{-c/\eps}
	+ \frac{1}{  (\sqrt \eps )^{d+4} (\sqrt{ \eps t})^{\zeta_0}} \biggr]
	\frac 1 {(\sqrt t)^d}
	\biggl( \frac{ \sqrt t } { \abs{ x - \integer{t \eta} } } \biggr)^n
	\qquad (n = 0,1,2)
\end{equation}
(here $n=0$ means the factor is absent),
with $c = \frac 1 4 \KIR $ and $C_1$ depending only on $n, d, \zeta, M, \KIR$.
\end{lemma}

\begin{proof}[Proof of Lemma~\ref{lem:R3_integral-2}]
Since $T$ is assumed to obey $T \ge \eps\inv \vee  \eps^{-(2+\zeta_0)/\zeta_0}$, we can apply Lemma~\ref{lem:It_asymp-2} to all $t \ge T$.
By taking a larger constant, we can also assume the constant $C_1$ in \eqref{eq:large_t_error-2a} is independent of $n$.
For simplicity, we write $\varphi_\eps (t) = C_1 / [(\sqrt \eps )^{d+4} (\sqrt{ \eps t})^{\zeta_0} ]$,
write $\bar \varphi_\eps(T) = \sup_{t\ge T} \abs{ \varphi_\eps(t) } = \varphi_\eps(T)$,
and write
\begin{equation}
    f(x) =\frac{ 1 }{ \Casy(x; \eta )}
	\int_T^\infty \D t\, \abs{ I_{t}(x)
	- \hat g(0) \rho_t(x; \eta,\Lam) } .
\end{equation}
We seek an upper bound on $f(x)$.
We consider the two cases $\abs \eta \abs x \le \so$ and $\abs \eta \abs x \ge \so$ separately.

The critical case $\abs \eta \abs x \le \so$ is relevant only
when $d > 2$.  In this case, we
have $\Casy(x; \eta ) \asymp |x|^{-(d-2)}$
and $T=\eps \abs x^2$, and
the $n=0$ case of \eqref{eq:large_t_error-2a} gives
\begin{equation}
f(x) \lesssim \frac{ C_1 e^{-c/\eps} + \bar \varphi_\eps(T) }{ \abs x^{-(d-2)} }
	\int_{T}^\infty \frac{\D t}{ (\sqrt{t} )^d}
\lesssim
\frac{ C_1 e^{-c/\eps} + \bar \varphi_\eps(\eps \abs x^2) }{ \abs x^{-(d-2)} }
(\eps \abs x^2)^{ - (d-2) / 2} .
\end{equation}
The powers of $\abs x$ on the right-hand side cancel.
Inserting the definition of $\bar \varphi_\eps$, we obtain
\begin{equation}
f(x)
\lesssim \frac{ C_1 e^{-c/\eps}
	+ \bar \varphi_\eps(\eps \abs x^2) }{ (\sqrt\eps)^{ d-2 } }
\lesssim e^{-c/(2\eps)}
	+ \frac 1 { (\sqrt\eps)^{ 2d+2 } (\eps \abs x )^{\zeta_0} } ,
\end{equation}
which gives the desired \eqref{eq:large_time_int} (without its last term).

For the OZ case $\abs \eta\abs x \ge \so$, we have
$\Casy(x; \eta ) \asymp \abs\eta^{(d-3)/2}/|x|^{(d-1)/2}$
and $T=\eps \so \abs x/\abs\eta$.
Let $t_0=\abs x  /\abs \eta \in  (T,\infty)$.
We argue in a manner similar to Laplace's method, with
$\abs \eta \abs x$ playing the role of a large parameter,
and with the integral concentrated in a small window around $t_0$ as $\abs \eta \abs x \to \infty$.
Accordingly, we split the integral
over $[T,\infty)$ into two parts, depending on whether or not
$|t-t_0| \le \sqrt{t_0 \so}/(2|\eta|)$.
We apply the $n=0$ case of \eqref{eq:large_t_error-2a} when $|t-t_0| \le \sqrt{t_0 \so}/(2|\eta|)$, and apply the $n=2$ case otherwise.
This gives
\begin{multline} \label{eq:fI}
    f(x) \lesssim [C_1 e^{-c/\eps} + \bar \varphi_\eps(T) ]\frac{|x|^{(d-1)/2}}{\abs\eta^{(d-3)/2}}
    \\ \times
    \int_T^\infty \frac{\D t}{ (\sqrt{t})^d}
\biggl[
    	\1 \Bigl\{ \abs{t - t_0 }
	\le \frac{ \sqrt{t_0 \so} }{ 2\abs \eta }  \Bigr\}
  +  \Bigl( \frac{ \sqrt t } { \abs{ x - \integer{t \eta} } } \Bigr)^2
  	\1 \Bigl\{ \abs{t - t_0 }
	\ge \frac{ \sqrt{t_0 \so} }{ 2\abs \eta }  \Bigr\}
\biggr]  .
\end{multline}
We compute the two integrals separately.

For $t$ that satisfies $|t-t_0| \le \sqrt{t_0 \so}/(2|\eta|)$,
since $\abs \eta \ge \so \abs x\inv$, we have
\begin{equation}
    \frac{\sqrt{t_0 \so}}{2|\eta|}
    \le \frac{\sqrt{t_0}}{2}  \frac{ |x|^{1/2}}{|\eta|^{1/2}}
    = \frac 12 t_0,
\end{equation}
so $\frac 12 t_0 \le t \le \frac 32 t_0$
for all $t$ in this interval.
Therefore,
\begin{equation}
\int_T^\infty \frac {\D t} {(\sqrt t)^d} \1 \Bigl\{ \abs{t - t_0 }
	\le \frac{ \sqrt{t_0 \so} }{ 2\abs \eta }  \Bigr\}
\lesssim  \frac{1}{ (\sqrt{ t_0 } )^{d}}\frac{ \sqrt{t_0 \so}}{|\eta|}
= \sqrt{\so} \frac{ \abs \eta^{(d-3)/2} }{ \abs x^{(d-1)/2} } .
\end{equation}

For $t$ that satisfies $|t-t_0| \ge \sqrt{t_0 \so}/(2|\eta|)$,
we use the fact that $x=t_0\eta$ (because $x,\eta$ are in the same direction) to get
\begin{equation}
\label{eq:xetat0}
    \abs{ x - \integer{t \eta} }
    \ge \abs{ x - {t \eta} } -  \abs{  t \eta - \integer{t \eta} }
     \ge
     |\eta|  |t-t_0| - \half \sqrt{d} .
\end{equation}
Since $\abs x \ge 4dM / \so$ and $\abs \eta \le M$,
we have $|\eta|  |t-t_0|  \ge \half \sqrt{t_0 \so}
= \half \sqrt{ \so \abs x / \abs \eta }
\ge \sqrt d$,
so
the right-hand side of \eqref{eq:xetat0}
is bounded below by $\half |\eta|  |t-t_0|$.
With this replacement, the second part of the integral in \eqref{eq:fI} can be bounded by a multiple of (via change of variable $t=t_0z$)
\begin{align}
&\int_T^\infty \frac {\D t} {(\sqrt t)^d} \frac{t}{(|\eta||t-t_0|)^2} \1 \Bigl\{ \abs{t - t_0} \ge \frac{ \sqrt{t_0 \so}}{2|\eta|}  \Bigr\}
\nnb & \qquad
=
\frac{1}{|\eta|^2}\frac{1}{(\sqrt{t_0})^{d}}
\int_{\eps \so}^\infty \frac {\D z} {(\sqrt z)^d} \frac{z}{|z-1|^2} \1 \Bigl\{ \abs{z-1} \ge \frac{ \sqrt{ \so} }{2|\eta|\sqrt{t_0}}  \Bigr\}
\nnb & \qquad
=
\frac{\abs \eta^{(d-3)/2}}{\abs x^{(d-1)/2}}\frac{1}{\sqrt{\abs\eta\abs x}}
\int_{\eps \so}^\infty \frac {\D z} {z^{(d-2)/2}} \frac{1}{|z-1|^2} \1 \Bigl\{ \abs{z-1} \ge \frac{ \sqrt{ \so}}{2 \sqrt{\abs\eta\abs x}}  \Bigr\}
.
\end{align}
Since $d > 0$,
the integral on the right-hand side converges at $z=\infty$, is of order $\sqrt{\abs\eta\abs x}$ near $z=1$,
and is of order $1 \vee \eps^{-(d-4)/2}$ at $z=0$.

After insertion of the two integrals into \eqref{eq:fI}, we obtain
\begin{equation}
f(x)
\lesssim [C_1 e^{-c/\eps} + \bar \varphi_\eps(\eps \abs x / \abs \eta) ]
	\biggl( 1 + \frac{1 + \eps^{-(d-4)/2}}{\sqrt{\abs\eta\abs x}}  \biggr) .
\end{equation}
With the definition of $\bar \varphi_\eps$,
and using $\abs \eta\abs x \ge \so$, we get
\begin{equation}
f(x)
\lesssim \biggl[ e^{-c/\eps}
	+ \frac 1  { (\sqrt \eps )^{d+4} ( \eps \sqrt{ \abs x / \abs \eta})^{\zeta_0} } \biggr]
	(1 + \eps^{-(d-4)/2})
\lesssim e^{-c/(2\eps)} +
	\frac { \abs \eta^{\zeta_0/2} }  { (\sqrt \eps )^{d+(d\vee4)} \eps^{\zeta_0}  \abs x^{\zeta_0/2} } .
\end{equation}
This gives the desired \eqref{eq:large_time_int} (without its middle term), since $\abs \eta \le M$ and $d \le d\vee4$.
\end{proof}

\subsection{Proof of estimates on $I_{t} (x)$}
\label{sec:Ilemmas}

We now prove Lemmas~\ref{lem:It_decay} and~\ref{lem:It_asymp-2}.
Part of the analysis uses the theory of weak derivatives, as
found in \cite[Appendix~A]{LS24a}.

\subsubsection{Fourier transform estimate}

We begin with an elementary estimate for derivatives in the Fourier space.

\begin{lemma} \label{lem:Fourier}
Let $d\ge 1$, $a > b \ge 0$, and $1 \le p \le \infty$.
Suppose $f : \Zd \to \R$ satisfies
\begin{equation}
\bignorm{ \abs y^b f(y) }_{1} ,
\bignorm{ \abs y ^a f(y) }_{p \wedge 2} \le M
\end{equation}
for some constant $M < \infty$.
Then its Fourier transform $\hat f$ is $\floor a$ times weakly differentiable on $\Td$, and we have
\begin{equation}
\norm{ \grad^\gamma \hat f }_q \le M
\quad \text{with}  \quad
\frac 1 q = \frac{ \abs \gamma - b }{ a -b }
	\biggl( 1  - \frac 1 p \biggr) ,
\end{equation}
for all multi-indices $\gamma$ with $b \le \abs \gamma \le a$.
\end{lemma}

\begin{proof}
The hypothesis implies $\abs y^{\floor a} f(y) \in L^2(\Zd)$,
so $\hat f(k)$ is $\floor a$ times weakly differentiable on $\Td$ by \cite[Lemma~A.4]{LS24a}.
To estimate the derivatives, we first bound moments of $f(y)$ using H\"older's inequality.
For any multi-index $\gamma$ with $b \le \abs \gamma \le a$,
we decompose
\begin{equation}
\abs y^{\abs \gamma} f(y)
= \bigl( \abs y^b f(y) \bigr)	^{ \frac{a - \abs \gamma}{a-b} }
	\bigl( \abs y^{a} f(y)  \bigr)^{ \frac{\abs \gamma - b}{a-b} } .
\end{equation}
Then, with $r\in [1,\infty]$ defined by
\begin{align} \label{eq:interp_pf}
\frac 1 r
&= \biggl( \frac{a - \abs \gamma}{a-b} \biggr) 1
	+ \biggl(  \frac{\abs \gamma - b}{a-b} \biggr) \frac 1 p ,
\\
\half
&= \biggl( \frac{a - \abs \gamma}{a-b} \biggr) \half
	+ \biggl( \frac{\abs \gamma - b}{a-b} \biggr) \half ,
\end{align}
it follows from H\"older's inequality
and the monotonicity of $L^p(\Zd)$ norms in $p$ that
\begin{equation}
\bignorm{ \abs y^{\abs \gamma} f(y) }_{r \wedge 2}
\le \bignorm{ \abs y^b f(y)  }_{1 \wedge 2}^{ \frac{a - \abs \gamma}{a-b} }
	\bignorm{ \abs y^{a} f(y)  }_{p\wedge 2}^{ \frac{\abs \gamma - b}{a-b} }
\le M .
\end{equation}
Using the Hausdorff–Young inequality and the monotonicity of $L^q(\Td)$ norms in $q$, we obtain
\begin{equation}
\bignorm{ \grad^\gamma \hat f }_q
\le \bignorm{ \grad^\gamma \hat f }_{q \vee 2}
\le \bignorm{ \abs y^{\abs \gamma} f(y) }_{r \wedge 2}
\le M ,
\end{equation}
where $q \inv  = 1 - r \inv$.
Inserting the value of $r\inv$ from \eqref{eq:interp_pf} then gives the desired result.
\end{proof}

\subsubsection{Small $t$ estimate: proof of Lemma~\ref{lem:It_decay}}

Lemma~\ref{lem:It_decay} asserts the upper bound
\begin{equation} \label{eq:It_decay-2}
\abs{ I_{t} (x) }
	\lesssim \frac 1 {(1 + \sqrt t)^d}
	\biggl( \frac{ 1 + \sqrt t } { \abs{ x - \integer{t \eta} } } \biggr)^n
\qquad ( 0 \le n \le 2 \vee (d-1)) ,
\end{equation}
with constants depending only on
$n, d, M, \KIR$.
Recall that
\begin{equation} \label{eq:Q_moments2}
\bignorm{ \Jsupmu(y) }_{1} ,
\bignorm{ \abs y ^{2} \Jsupmu(y) }_{1} \le M  ,
\end{equation}
by the hypothesis \eqref{eq:Q_moments} on the moments of $\Jsupmu$.

\begin{proof}[Proof of Lemma~\ref{lem:It_decay}]
We begin by rewriting the definition of $I_{t}(x)$ in \eqref{eq:G^mu} as
\begin{equation}
I_{t} (x)  = \int_\Td \frac{\D k }{(2\pi)^d}
	e^{ik\cdot x} \hat g(k) e^{-t [\hat \Jsupmu(0) - \hat \Jsupmu (k) ] }
= \int_\Td \frac{\D k }{(2\pi)^d}
	e^{ik\cdot (x - \integer{t\eta}) }
 \hat g(k)
 e^{ik\cdot \integer{t\eta} -t [\hat \Jsupmu(0) - \hat \Jsupmu (k) ] } .
\end{equation}
This can be viewed as the inverse Fourier transform of the function
$ \hat g(k) e^{-\hat R(k)}$
evaluated at the lattice point $x - \integer{t\eta}$, where
\begin{equation}
\hat R(k) = -ik\cdot \integer{t\eta} +t [\hat \Jsupmu(0) - \hat \Jsupmu (k) ] .
\end{equation}
We exploit the fact that derivatives in the Fourier space give control on the decay in the physical space.
To prove the desired \eqref{eq:It_decay-2}, by \cite[Lemma~2.3]{LS24a} it suffices to prove that $\hat g(k) e^{-\hat R(k)}$ is $2\vee(d-1)$ times weakly differentiable and that its derivatives obey
\begin{equation} \label{eq:It_decay_claim}
\bignorm{ \grad^\gamma ( \hat g e^{-\hat R} ) }_{1}
\lesssim \frac 1 { (1 + \sqrt t )^d} (1 + \sqrt t )^{\abs \gamma}
\end{equation}
for every multi-index $\gamma$ with $\abs \gamma \le 2 \vee(d-1)$.
(We needed $\integer{t\eta}\in \Zd$ to get a periodic function.)
Since the function $\hat g$ plays only a minor role in the proof of \eqref{eq:It_decay_claim},
we estimate first the derivatives of $e^{-\hat R}$.
By considering the two cases $\abs \gamma \le 2$ and $3 \le \abs \gamma \le d-1$,
we will prove that
\begin{equation} \label{eq:It_decay_claim_p}
\bignorm{ \grad^\gamma e^{-\hat R}  }_{p}
\lesssim \frac 1 { (1 + \sqrt t )^{d/p}} (1 + \sqrt t )^{\abs \gamma}
	\qquad (  \frac{ \abs \gamma - 2 } d \vee 0 < p\inv \le 1 ) .
\end{equation}

\smallskip \noindent
\emph{Case: $\abs \gamma \le 2$.}
For this case,
since the function $\abs y^2 \Jsupmu(y)$ is absolutely summable by hypothesis \eqref{eq:Q_moments}, the function $\hat \Jsupmu$, and thus $\hat R$ and $e^{- \hat R}$, are classically twice differentiable on $\Td$.
By calculus rules, for any $j,l\in \{1, \dots, d\}$ we have
\begin{equation} \label{eq:eR_derivative}
\begin{aligned}
\grad_j e^{-\hat R}
	&= e^{-\hat R} ( - \grad_j \hat R) , \\
\grad_j \grad_l e^{-\hat R}
	&= e^{-\hat R} ( - \grad_j \hat R)( - \grad_l \hat R)
		- e^{-\hat R} ( \grad_j \grad_l \hat R ) .
\end{aligned}
\end{equation}
By the Taylor expansion \eqref{eq:J_small_k}, we have
\begin{equation} \label{eq:R_derivative}
\begin{aligned}
\grad_j \hat R
&= - i \integer{t\eta}_j - t \grad_j \hat \Jsupmu
= i ( \fractional{ t \eta} )_j + t \grad_j [ \half k \cdot \Lambda k + \hat R_2(k)] ,
\\
\grad_j \grad_l \hat R
&=  t \grad_j \grad_l [ \half k \cdot \Lambda k + \hat R_2(k)] .
\end{aligned}
\end{equation}
These derivatives can then be bounded using \eqref{eq:Q_moments2}
as
\begin{equation} \label{eq:R1}
\abs{ \grad_j \hat R }
\lesssim  \abs{ \fractional{ t \eta} } + t\abs k   M
\le \tfrac{ \sqrt d } 2 +  t\abs k   M ,
	\qquad
\abs{ \grad_j \grad_l \hat R }
	\lesssim  t M     .
\end{equation}
Also, we can bound $e^{-\hat R}$ using the infrared bound \eqref{eq:Q_infrared} since $\hat R(k)$ and $t(\hat \Jsupmu(0)-\hat \Jsupmu(k))$ have the same real part.
This gives, for all $\abs \gamma \le 2$,
\begin{equation} \label{eq:small_gamma_int_pf}
\abs{ \grad^\gamma e^{-\hat R} }
\le C_{\abs \gamma, M}  e^{-\half t \KIR \abs k^2}
	\Bigl[  \bigl( 1+ t\abs k \bigr)^{\abs \gamma}
	+ t \1_{\abs \gamma = 2} \Bigr] .
\end{equation}
To compute its $L^p(\Td)$ norm, we either bound the exponential by $1$ or extend the integral to be over all $k\in \Rd$.
For this, we use the fact that for any real numbers $a > n \ge 0$ and any $p > 0$, there are constants such that
\begin{align} \label{eq:gamma_int}
\bignorm{ \abs k^n e^{- a \abs k^2} }_{L^p(\Td)}
&\le \min \Biggl\{
	\max_{k\in \T^d}\, \abs k^n 	,
	\biggl(
	\int_\Rd \frac{ \D k}{(2\pi)^d}  \abs k^{pn} e^{- pa \abs k^2}
	\biggr)^{1/p}
	\Biggr\}  .
\end{align}
The first item on the right-hand side is bounded by a constant, and with a change
of variables $\ell=\sqrt{pa}k$ we see that the second is bounded by a multiple of
$(\sqrt {p a})^{-(d+pn)}$. Therefore,
\begin{align}	
\bignorm{ \abs k^n e^{- a \abs k^2} }_{L^p(\Td)}
&
\le \frac{ c_{d,n,p} } { (1 + \sqrt a)^{\frac d p + n} } .
\end{align}
Taking the $L^p$ norm of \eqref{eq:small_gamma_int_pf} this way, we get
\begin{equation}
\bignorm{ \grad^\gamma e^{-\hat R} }_{p}
\le \frac{ C_{\abs \gamma, p, M, \KIR} }{ (1 + \sqrt t)^{d/p} }
	\Bigl[  \bigl( 1 + \sqrt t \bigr)^{\abs \gamma}
	+ t \1_{\abs \gamma = 2} \Bigr] ,
\end{equation}
which gives the claimed \eqref{eq:It_decay_claim_p} for all $1 \le p < \infty$.

\smallskip \noindent
\emph{Case: $3 \le \abs \gamma \le d-1$.}
In this case we can assume $d\ge 4$.
We first estimate higher weak derivatives of $\hat R(k)$.
Using \eqref{eq:Q_moments} and \eqref{eq:Q_d-1},
we can apply Lemma~\ref{lem:Fourier} to $f = \Jsupmu$ (with $a = d-1$, $b = 2$, and $p = d/3$), to conclude that $\hat \Jsupmu(k)$, and thus $\hat R(k)$, are $d-1$ times weakly differentiable.
Also, for any multi-index $\alpha$ with $2 \le \abs \alpha \le d-1$,
since
\begin{equation}
\frac{ \abs \alpha - 2 }{ (d-1) -2 } \biggl( 1  - \frac 3 d \biggr)
= \frac{ \abs \alpha - 2 }{ d } ,
\end{equation}
the quantitative estimate of Lemma~\ref{lem:Fourier} gives
\begin{equation} \label{eq:R2+}
\bignorm{ \grad^\alpha \hat R }_{\frac d {\abs \alpha -2}}
= t \bignorm{ \grad^\alpha \hat \Jsupmu }_{\frac d {\abs \alpha -2}}
\le t M .
\end{equation}

To show that $e^{-\hat R(k)}$ is $d-1$ times weakly differentiable, we use the product rule for weak derivatives \cite[Lemma~A.2]{LS24a}, which states that it suffices to compute $\grad^\gamma e^{-\hat R}$ formally for all $\abs \gamma \le d-1$ and verify that all results are integrable. We therefore consider all terms of the form
\begin{equation} \label{eq:grad_eR_decomp}
e^{- \hat R} \prod_{j = 1}^N (- \grad^{\alpha_j} \hat R) ,
\end{equation}
where $0 \le N \le \abs \gamma$, $\abs{ \alpha_j } \ge 1$, and $\gamma = \sum_{j=1}^N \alpha_j$;
a linear combination of these terms gives $\grad^\gamma e^{-\hat R}$.
Let $q \in [1,\infty)$ obey
\begin{equation}
1 \ge q\inv > \frac{ \abs \gamma - 2 } d .
\end{equation}
We define $a = \#\{ j : \abs {\alpha_j} = 1 \} \in [0,N]$
and define a number $p \ge q$ by
\begin{equation} \label{eq:It_decay_pf_p}
\frac 1 p = \frac 1 q - \sum_{j: \abs{\alpha_j} \ge 2} \frac{ \abs{\alpha_j} - 2 } d
= \frac{ d - (\abs \gamma -a) + 2 (N-a) } d .
\end{equation}
Note that $p\inv$ is smallest when $N=1$ and $\alpha_1 = \gamma$,
so we always have $p\inv \ge q\inv - (\abs\gamma - 2) / d  > 0$,
which implies $p < \infty$.
By H\"older's inequality, \eqref{eq:R1}, \eqref{eq:R2+}, and the infrared bound \eqref{eq:Q_infrared}, we have
\begin{align}
\Bignorm{ e^{- \hat R} \prod_{j = 1}^N (- \grad^{\alpha_j} \hat R) }_q
&\le  \Bignorm{ e^{- \hat R} \prod_{j : \abs{\alpha_j} = 1}
		(\grad^{\alpha_j} \hat R) }_p
	\prod_{j : \abs{\alpha_j} \ge 2}
		\bignorm{ \grad^{\alpha_j} \hat R }_{\frac d {\abs{\alpha_j} - 2}}
\nl
&\le C_{ \abs \gamma, M }  \Bignorm{ e^{- \half t \KIR \abs k^2}
	\bigl( 1 + t \abs k \bigr)^a }_p
	\prod_{j : \abs{\alpha_j} \ge 2}  t
\nl
&\le \frac{ C_{ \abs \gamma,q, M, \KIR } } {(1 + \sqrt t)^{d/p} }
	( 1 +  \sqrt t )^a
	t^{N-a}
\nl
&\le \frac{ C_{ \abs \gamma,q, M, \KIR } } {(1 + \sqrt t)^{d/q} }
	( 1 +  \sqrt t )^a
	(1 + \sqrt t)^{\abs \gamma - a} ,
\end{align}
where we used \eqref{eq:gamma_int} in the third and \eqref{eq:It_decay_pf_p} in the fourth line.
In particular, we can take $q = 1$ to conclude that $e^{-\hat R(k)}$ is $d-1$ times weakly differentiable.
The claimed estimate \eqref{eq:It_decay_claim_p} on $\norm{ \grad^\gamma e^{-\hat R} }_q$ also follows from the above estimate, by taking a linear combination.

\smallskip \noindent
\emph{Proof of \eqref{eq:It_decay_claim}.}
We now prove that $\hat g e^{-\hat R}$ is $(d-1)\vee 2$ times weakly differentiable and estimate its derivatives. We use the product rule again.
For the derivatives of $\hat g$,
we use \eqref{eq:g_moments} and \eqref{eq:Q_d-1}
to invoke Lemma~\ref{lem:Fourier} with $f = g$, $a = 2 \vee (d-1)$, $b = 0$, and
\begin{equation}
p = \begin{cases}
d 		&(d >2) \\
\infty		&(d \le 2) .
\end{cases}
\end{equation}
It gives $\norm{ \grad^\alpha \hat g }_{ (2\vee d) /{\abs \alpha}}
\le M$ for all multi-indices $\alpha$ with $\abs \alpha \le 2\vee(d-1)$.
If $d >2$,
since $\abs \alpha \le \abs \gamma \le d-1$,
the product rule and \eqref{eq:It_decay_claim_p} imply
\begin{equation}
\bignorm{ \grad^\gamma ( \hat g e^{-\hat R} ) }_{1}
\le \sum_{\alpha \le \gamma}
	\bignorm{ \grad^\alpha \hat g }_{ \frac d {\abs \alpha}}
	\bignorm{ \grad^{\gamma-\alpha} e^{-\hat R} }_{ \frac d {d - \abs \alpha}}
\lesssim \sum_{\alpha \le \gamma}
	\frac { (1 + \sqrt t )^{\abs \gamma - \abs \alpha} }
		{ (1 + \sqrt t )^{d - \abs \alpha}}
\lesssim 	\frac { (1 + \sqrt t )^{\abs \gamma } }
		{ (1 + \sqrt t )^{d}}
\end{equation}
for every multi-index $\gamma$ with $\abs \gamma \le d-1$,
as desired.
If $d \le 2$,
we instead have
\begin{equation} \label{eq:small_t_d=1}
\bignorm{ \grad^\gamma ( \hat g e^{-\hat R} ) }_{1}
\le \sum_{\alpha \le \gamma}
	\bignorm{ \grad^\alpha \hat g }_{ \frac 2 {\abs \alpha}}
	\bignorm{ \grad^{\gamma-\alpha} e^{-\hat R} }_{ \frac 2 {2 - \abs \alpha}}
\lesssim \sum_{\alpha \le \gamma}
	\frac { (1 + \sqrt t )^{\abs \gamma - \abs \alpha} }
		{ (1 + \sqrt t )^{ d - \frac d 2 \abs \alpha }}
\qquad (\abs \gamma \le 2)
\end{equation}
(note \eqref{eq:It_decay_claim_p} also works for the case $\abs \alpha = \abs \gamma = 2$ since no derivative is taken),
which is $\lesssim (1 + \sqrt t)^{\abs \gamma - d}$ because $d\le 2$.
This concludes the proof of \eqref{eq:It_decay_claim}.
\end{proof}

\subsubsection{Large $t$ estimate: proof of Lemma~\ref{lem:It_asymp-2}}

We now prove Lemma~\ref{lem:It_asymp-2}, which we restate for convenience as Lemma~\ref{lem:It_asymp-3}.

\begin{lemma} \label{lem:It_asymp-3}
Let $d \ge 1$, $\eps \in (0,1] $, and denote
$\zeta_0 = \zeta \wedge \frac {d \wedge 2}2$.
For any $t \ge \eps\inv \vee \eps^{-(2+\zeta_0)/\zeta_0}$ and any $x\in \Zd$, we have
\begin{equation}  \label{eq:large_t_error}
\bigabs{ I_{t}(x) - \hat g(0) \rho_t(x; \eta, \Lam) }
\le C_1 \biggl[ e^{-c/\eps}
	+ \frac{1}{  (\sqrt \eps )^{d+4} (\sqrt{ \eps t})^{\zeta_0}} \biggr]
	\frac 1 {(\sqrt t)^d}
	\biggl( \frac{ \sqrt t } { \abs{ x - \integer{t \eta} } } \biggr)^n
	\qquad (n = 0,1,2)
\end{equation}
(here $n=0$ means the factor is absent),
with $c = \frac 1 4 \KIR $ and $C_1$ depending only on $n, d, \zeta, M, \KIR$.
\end{lemma}

\begin{proof}
Since Assumption~\ref{ass:J} is weaker when $\zeta$ is smaller, we assume $\zeta \le \half( d\wedge 2)$.
We focus on the case $n=2$.
Let $\Delta = \sum_{j=1}^d \grad_j^2 $ denote the Laplacian.
As in the proof of Lemma~\ref{lem:It_decay}, using integration by parts for each coordinate, we can write
\begin{equation} \label{eq:large_t_pf1}
- \bigabs{ x - \integer{t\eta} }^2 I_{t}(x)
= \int_\Td \frac{\D k }{(2\pi)^d}
	e^{ik\cdot (x - \integer{t\eta}) } \Delta( \hat g(k) e^{-\hat R(k)} )  ,
\end{equation}
where $\hat R(k) = -ik\cdot \integer{t\eta} +t [\hat \Jsupmu(0) - \hat \Jsupmu (k) ]$.
We will show that the main contribution of the integral can be computed by replacing $\Delta( \hat g(k) e^{-\hat R(k)} )$ with $ \hat g(0) \Delta (e^{-\hat P(k)})$, where
\begin{equation}
\hat P(k) =  i k\cdot \fractional{ t \eta} + \frac t 2 k \cdot \Lam k
\end{equation}
is the second order Taylor polynomial of $\hat R(k)$ around $k=0$.
Indeed, the $\Rd$ integral of the latter
can be computed exactly, as
\begin{align} \label{eq:large_t_pf2}
\int_\Rd \frac{\D k }{(2\pi)^d}
	e^{ik\cdot (x - \integer{t\eta}) } \hat g(0)  \Delta( e^{-\hat P(k)} )
&= - \bigabs{ x - \integer{t\eta} }^2 \hat g(0)
	\int_\Rd \frac{\D k }{(2\pi)^d}
	e^{ik\cdot x } e^{-t [ i k\cdot \eta + \half k \cdot \Lambda k ] } \nl
&= - \bigabs{ x - \integer{t\eta} }^2 \hat g(0)
	\rho_t(x;\eta,\Lam) .
\end{align}
We will prove
\begin{equation} \label{eq:large_t_claim_g}
\int_\Td \frac{\D k }{(2\pi)^d}
	\bigabs{ \Delta( \hat g(k) e^{-\hat R(k)} )
		- \hat g(0) \Delta( e^{-\hat R(k)} ) }
\lesssim \frac 1 { t^{\zeta/2} } \frac t {(\sqrt t)^d } ,
\end{equation}
and
\begin{multline} \label{eq:large_t_claim}
\abs{ \hat g(0) } \biggabs{
	\int_\Td \frac{\D k }{(2\pi)^d}
		e^{ik\cdot (x - \integer{t\eta}) }
		\Delta( e^{-\hat R(k)} )
	- \int_\Rd \frac{\D k }{(2\pi)^d}
		e^{ik\cdot (x - \integer{t\eta}) }
		\Delta( e^{-\hat P(k)} )
} \\
\lesssim \Bigl( e^{-c/\eps} + \frac{1}{ \eps^{(d+4)/2} ( \eps t)^{\zeta/2}} \Bigr) \frac t {(\sqrt t)^d }.
\end{multline}
By the triangle inequality and by $\eps \le 1$,
the combination of \eqref{eq:large_t_pf1} with \eqref{eq:large_t_pf2}--\eqref{eq:large_t_claim} imply the desired \eqref{eq:large_t_error} with $n=2$.

\smallskip \noindent
\emph{Proof of \eqref{eq:large_t_claim_g}.}
We expand $\Delta( \hat g(k) e^{-\hat R(k)} )$ using the product rule.
If no derivative is applied to $\hat g$,
we use the $\zeta^{\rm th}$ moment of $g(y)$ in \eqref{eq:g_moments} to bound $\abs{ \hat g(k) - \hat g(0) } \lesssim \abs k^\zeta M$
(here we use $\zeta \le 1$.)
Then, by \eqref{eq:small_gamma_int_pf} and \eqref{eq:gamma_int}, we have
\begin{multline}
\int_\Td \frac{\D k }{(2\pi)^d}
	\bigabs{ \hat g(k) \Delta( e^{-\hat R(k)} )
		- \hat g(0) \Delta( e^{-\hat R(k)} ) } 	\\
\lesssim  \int_\Td \frac{\D k }{(2\pi)^d}
	\abs k^\zeta  e^{-\half t \KIR \abs k^2}
	\Bigl[  \bigl( 1+ t\abs k \bigr)^{\abs \gamma}
	+ t \Bigr]
\lesssim \frac t { (\sqrt t )^{d+\zeta} } .
\end{multline}
If some derivative is applied to $\hat g$,
in dimensions $d\ge 2$
we use Lemma~\ref{lem:Fourier} with $f = g$, $a = 2$, $b = \zeta$, and $p =p_d = d / (d-2+\zeta)$.
Since $\zeta \le 1$, and since
\begin{equation}
\frac{ \abs \alpha - \zeta }{ 2 - \zeta } \biggl( 1  - \frac {d-2+\zeta} d \biggr)
= \frac{ \abs \alpha - \zeta }{ d } ,
\end{equation}
we get $\norm{ \grad^\alpha \hat g }_{ d /{ ( \abs \alpha - \zeta )}}
\le M$ for all multi-indices $\alpha$ with $1 \le \abs \alpha \le 2$.
Then for any multi-index $\gamma \ge \alpha$ with $\abs \gamma = 2$,
H\"older's inequality and \eqref{eq:It_decay_claim_p} give
\begin{equation}
\bignorm{ \grad^\alpha \hat g   \grad^{\gamma-\alpha} e^{- \hat R} }_{1}
\le \bignorm{ \grad^\alpha \hat g }_{ \frac d {\abs \alpha - \zeta}}
	\bignorm{ \grad^{\gamma-\alpha} e^{-\hat R} }_{ \frac d {d - \abs \alpha + \zeta}}
\lesssim
	\frac { (1 + \sqrt t )^{\abs \gamma - \abs \alpha} }
		{ (1 + \sqrt t )^{d - \abs \alpha + \zeta}}
\le 	\frac { t }	{ ( \sqrt t )^{d+\zeta}} .
\end{equation}
For $d=1$, it can be seen in \eqref{eq:small_t_d=1} that the same quantity is bounded by $(1+\sqrt t)^{2 - d - \half \abs \alpha}$. This is smaller than the desired $ t  /( \sqrt t )^{d+\zeta}$ because we have $\half \abs \alpha \ge \half \ge \zeta$ when $\abs \alpha \ge 1$ (here we use $\zeta \le \half$).
This concludes the proof of \eqref{eq:large_t_claim_g}.

\smallskip \noindent
\emph{Proof of \eqref{eq:large_t_claim}.}
Since $\abs{\hat g(0)} \le \norm{ g(y) } _1 \le M$ by \eqref{eq:g_moments}, we only need to consider the difference of the two integrals in \eqref{eq:large_t_claim}.
We
decompose the difference of the integrals into three parts, similar to the proof of \cite[Lemma~2.2]{Hara08}, as follows.
We define
\begin{equation}
k_t = (\eps t)^{-1/2} \le 1
\end{equation}
and
\begin{equation} \label{eq:I2-4}
\begin{aligned}
\Ical_{2,j} &=
	- \int_{k\in \Rd : \abs k > k_t} \frac{\D k }{(2\pi)^d}
	e^{ik\cdot (x - \integer{t\eta}) }
    \Delta(e^{-\hat P(k)})
    ,
\\
\Ical_{3,j} &=
	\int_{\abs k \le k_t} \frac{\D k }{(2\pi)^d}
	e^{ik\cdot (x - \integer{t\eta} ) } \Bigl\{
     \Delta(e^{-\hat R(k)}) -  \Delta(e^{-\hat P(k)})
	\Bigr\}
    ,
\\
\Ical_{4,j} &=
	\int_{k\in \Td : \abs k > k_t} \frac{\D k }{(2\pi)^d}
	e^{ik\cdot (x - \integer{t\eta}) }
    \Delta(e^{-\hat R(k)})
    ,
\end{aligned}
\end{equation}
so that our goal becomes to estimate $\abs{ \sum_{j = 1}^d ( \Ical_{2,j} + \Ical_{3,j} + \Ical_{4,j}) }$.
It suffices to prove
\begin{equation} \label{eq:I234_claim}
\abs{ \Ical_{2,j} }, \abs{\Ical_{4,j} }
	\lesssim e^{-c/\eps} \frac t {(\sqrt t)^d},
\qquad
\abs{\Ical_{3,j}} \lesssim  \frac{1}{ \eps^{(d+4)/2} ( \eps t)^{\zeta/2}} \frac t {(\sqrt t)^d}
\end{equation}
for all $j$.
The oscillating factor $e^{ik\cdot (x - \integer{t\eta}) }$ plays no role in the bounds, and for all three integrals
we will take absolute values inside and bound that factor by $1$.

\smallskip\noindent
\emph{Bound on $\Ical_{4,j}$.}  A similar estimate has already been obtained in the proof of  Lemma~\ref{lem:It_decay}. We use the bound \eqref{eq:small_gamma_int_pf} on
$\Delta(e^{-\hat R(k)})$  to get
\begin{equation} \label{eq:I4_pf_1}
\abs{ \Ical_{4,j} }
\le C_{M} \int_{k\in \Td : \abs k > k_t} \frac{\D k }{(2\pi)^d}
	e^{-\half t \KIR \abs k^2}
	\bigl[ ( 1 + t \abs k )^2 + t  \bigr] .
\end{equation}
Since for any real numbers $a,b > 0$ and $n > -d$ we have
\begin{equation} \label{eq:truncated_gamma}
\int_{k\in \Rd : \abs k \ge b} \frac{ \D k}{(2\pi)^d}  \abs k^n e^{- a \abs k^2}
\le e^{- \half ab^2}
	\int_\Rd \frac{ \D k}{(2\pi)^d}  \abs k^n e^{- \half a \abs k^2}
=
e^{- \half ab^2}
\frac{c_{d,n}}{ (\sqrt {a})^{d+n} }
 ,
\end{equation}
we get
\begin{equation} \label{eq:I4_pf_2}
\abs{ \Ical_{4,j} }
\le C_{M, \KIR}  e^{-\frac 1 4 t \KIR k_t^2}
	\frac{   1 +  \sqrt t  + t } { (\sqrt t)^d }
\le C_{M, \KIR}  e^{-\frac 1 4 \KIR / \eps}
	\frac{ 3t} { (\sqrt t)^d } ,
\end{equation}
using the definition $k_t = (\eps t)^{-1/2}$ and the fact that $t\ge 1$.

\smallskip\noindent
\emph{Bound on $\Ical_{2,j}$.}
This is similar to $\Ical_{4,j}$.
We still have an infrared bound by the paragraph before \eqref{eq:quadratic_form}.
Using the definition of $\Lambda$ in \eqref{eq:def_Lambda}, we can estimate derivatives of $\hat P(k) = i k\cdot \fractional{ t \eta} + \frac t 2 k \cdot \Lam k$ by
\begin{equation} \label{eq:P1}
\begin{aligned}
\abs{ \grad_j \hat P }
&= \bigabs{ i ( \fractional{ t \eta} )_j + t (\Lam k)_j  }
\lesssim 1 + t \abs k M
,
\\
\abs{\grad_j \grad_l \hat P}
&=  t \abs{ \Lam_{jl} }
\le   t  M
.
\end{aligned}
\end{equation}
These imply the same upper bound \eqref{eq:I4_pf_1} (with $k\in \Rd$), and thus \eqref{eq:I4_pf_2}, for $\abs{ \Ical_{2,j}}$.

\smallskip\noindent
\emph{Bound on $\Ical_{3,j}$.}
For this, we distribute $\Delta = \sum_{j=1}^d \grad_j^2$ and bound
\begin{equation}
\abs{\Ical_{3,j} }
\le \sum_{j=1}^d \int_{\abs k \le k_t} \frac{\D k }{(2\pi)^d}
	\abs{   \grad_j^2(e^{-\hat R(k)}) -  \grad_j^2(e^{-\hat P(k)}) } .
\end{equation}
Since $\hat R = \hat P + t \hat R_2$ by the Taylor expansion \eqref{eq:J_small_k},
we can write
\begin{equation}
e^{- \hat R} - e^{- \hat P}
= e^{- \hat P} \hat \psi,
\qquad \text{with} \quad
\hat \psi = e^{-t\hat R_2}-1 .
\end{equation}
Then by calculus rules,
\begin{equation}
\abs{ \grad_j^2 e^{- \hat R} - \grad_j^2  e^{- \hat P} }
\lesssim \abs{ e^{- \hat P} }  \Bigl(
    | \grad_j \hat P|^2 | \hat \psi |
    +
    | \grad_j^2 \hat P| | \hat \psi |
    +
    | \grad_j \hat P| | \grad_j \hat \psi |
    +
    | \grad_j^2  \hat \psi | \Bigr) .
\end{equation}
We bound $\abs{ e^{- \hat P} } \le 1$ using the infrared bound, and we recall the derivatives of $\hat P$ from \eqref{eq:P1}.
For derivatives of $\hat \psi$, we first use \eqref{eq:R2bds} to get
\begin{equation}
    |t\hat R_2(k)| \lesssim t|k|^{2+\zeta},
    \qquad
    |t \grad_j \hat R_{2}(k)| \lesssim t|k|^{1+\zeta},
    \qquad
    |t \grad_j^2 \hat R_{2}(k)| \lesssim t|k|^{\zeta}.
    \qquad
\end{equation}
Since $|k|\le k_t = (\eps t)^{-1/2}$, and since $t^\zeta \ge \eps^{-(2+\zeta)}$ by our hypothesis, we have
\begin{equation}
t \abs k^{2+\zeta}
\le \frac 1 { ( \sqrt \eps )^{2+\zeta} (\sqrt t)^{\zeta} } \le 1 .
\end{equation}
It follows that $|tR_2(k)|$ is bounded, and we have
\begin{equation}
    |\hat\psi(k)| \lesssim t|k|^{2+\zeta},
    \qquad
    |\grad_j \hat\psi(k)| \lesssim t|k|^{1+\zeta},
    \qquad
    |\grad_j^2 \hat\psi(k)| \lesssim t|k|^{\zeta} + (t|k|^{1+\zeta})^2.
\end{equation}
Altogether (using $|k|\le k_t\le 1$ to reduce powers of $\abs k$),
\begin{align}
\abs{ \grad_j^2 e^{- \hat R} - \grad_j^2  e^{- \hat P} }
    & \lesssim
    (1+t|k|)^2 (t|k|^{2+\zeta} )
    +
    (t)(t|k|^{2+\zeta} )
    +
    (1+t|k|) (t|k|^{1+\zeta})
    +
    t|k|^{\zeta}
+ (t|k|^{1+\zeta})^2
        \nnb
& \lesssim
    t|k|^{\zeta}
     +
    t^2|k|^{2+\zeta}
    +
    t^3 |k|^{4+\zeta}
        .
\end{align}
Integrating the above over $\abs k \le k_t = (\eps t)^{-1/2}$ then gives
\begin{align}
     \int_{|k|\le k_t} \frac{\D k }{(2\pi)^d}
	|\nabla_j^2 e^{-\hat R} - \nabla_j^2 e^{-\hat P}|
&\lesssim
	t k_t^{d+\zeta} + t^2 k_t^{d+2+\zeta} + t^3 k_t^{d+4+\zeta}  \nl
&= \Bigl( \frac 1 {\eps^{\frac {d} 2 } }  +  \frac 1 {\eps^{\frac {d+2} 2 } }   +  \frac 1 {\eps^{\frac {d+4} 2 } }   \Bigr)
	\frac{1}{(\eps t)^{\zeta/2}} \frac{t}{t^{d/2}} .
\end{align}
Since $\eps \le 1$, this proves \eqref{eq:I234_claim} and completes the proof of \eqref{eq:large_t_error} for the case $n=2$.

\smallskip\noindent
\emph{Bounds for $n=0,1$.}
The case $n=0$ can be proved along the same lines, by decomposing
\begin{equation}
I_{t}(x)  - \hat g(0)\rho_t(x; \eta , \Lambda )
\end{equation}
into terms like in \eqref{eq:large_t_claim_g}, \eqref{eq:large_t_claim}, \eqref{eq:I2-4} but without the Laplacian.
The estimates are easier, and the factor $t$ present in the bound for $n=2$ is
absent for $n=0$ because it occurred only due to the derivatives of $\hat g$,
$\hat R$, and $\hat P$.
We omit the details.
The case $n=1$ follows from the cases $n=0,2$ and the elementary inequality $a \le \half (1 + a^2)$, applied to $a = \sqrt t / \abs{ x - \integer{t\eta} }$.
\end{proof}

\section{Self-avoiding walk}
\label{sec:SAW}

We now verify the hypotheses and apply the results of Sections~\ref{sec:introduction}--\ref{sec:critical}
to the nearest-neighbour self-avoiding walk on $\Zd$ in dimensions $d \ge 5$.

\subsection{Main result}

Let $d \ge 2$.  An $n$-step nearest-neighbour \emph{self-avoiding walk} on $\Zd$, from $0$ to $x$, is a sequence $w=(w_0,w_1,\ldots, w_n)$
with $w_0=0$, $w_n=x$, $|w_{i}-w_{i-1}|=1$ for $1\le i \le n$, and
$w_i \neq w_j$ for all $i \neq j$.
Let $c_n(x)$ denote the number of $n$-step self-avoiding walks from $0$ to $x$, and let $c_n=\sum_{x\in \Zd}c_n(x)$ denote the number of $n$-step self-avoiding walks starting at $0$.
The \emph{two-point function} is
\begin{equation}
    G_z(x)=\sum_{x\in \Z^d}c_n(x) z^n,
\end{equation}
and the \emph{susceptibility} is
\begin{equation}
    \chi(z) = \sum_{x\in\Zd}G_z(x) = \sum_{x\in \Z^d}c_n z^n.
\end{equation}
Both of the above power series have a finite
$d$-dependent radius of convergence $z_c>0$, which is independent of $x$
\cite[Corollary~3.2.6]{MS93}.

Let $P(x) = \frac{1}{2d} \1\{ \|x\|_1=1 \}$.
For $d \ge 5$ and for $z \in [0,z_c]$,
the lace expansion \cite{BS85,MS93,Slad06} provides an OZ equation
\begin{equation}
    G_z(x) = \delta_{0,x} + (J_z*G_z) (x)
    \quad \text{with} \quad
    J_z = 2dzP + \Pi_z.
\end{equation}
The function $\Pi_z$ is an explicit but intricate infinite-range
$\Zd$-symmetric function that takes both positive and negative values.
For $d \ge 5$, control of the function $\Pi_z$ uniformly in
the closed interval $z \in [0,z_c]$ is obtained in \cite{HS92a,HS92b,Hara08};
we rely heavily on that control.

By definition,
\begin{equation} \label{eq:Omega_SAW}
\Omega_z
= \Bigl\{ \mu\in \Rd : \chi\supmu(z) = \sum_{x\in \Zd} G_z(x) e^{\mu\cdot y} < \infty \Bigr \} .
\end{equation}
Our main result for self-avoiding walk is the following theorem.

\begin{theorem}
\label{thm:SAW-intro}
(i)
The function $G_z$ obeys Assumption~\ref{ass:Omega}
for all $d \ge 2$ and $z \in (0,z_c)$.
\\
(ii)
For $d \ge 5$, and for any fixed $\zeta \in (0, d-4]$,
there exist $\delta,M,\KIR >0$ such that,
uniformly in $z\in [z_c-\delta,z_c)$ and $\mu\in\overline\Omega_z$,
$(J_z\supmu,\delta_0\supmu)\in \Qcal_{M,\KIR,\zeta}$.
\end{theorem}

By Theorem~\ref{thm:SAW-intro}, the conclusions of
Theorem~\ref{thm:crossover} and
Corollaries~\ref{cor:OZ} and~\ref{cor:crossover_weak} hold
for dimensions $d \ge 5$,
uniformly in $z \in [z_c-\delta,z_c)$.

The elementary result \cite[Theorem~4.1.18]{MS93} tells us that
there is a norm $|\cdot|_z$ on $\Rd$, satisfying $\|v\|_\infty \le |v|_z \le \|v\|_1$ for every $v\in \Rd$, such that the limit
\begin{equation}
\label{eq:tildemSAW}
    \tilde m_z = \lim_{|x|_z\to \infty}\frac{-\log G_z(x)}{|x|_z}
\end{equation}
exists in $(0,\infty)$.  We also see from \cite[Theorem~4.1.18]{MS93}
that there is a $\kappa_z>0$ such that, for every $x \in \Zd$,
\begin{equation}
    G_z(x) \le \kappa_z e^{- \tilde m_z|x|_z}.
\end{equation}
Consequently, $\chi^{(te_1)}(z) < \infty$ for $t<\tilde m_z$, and
$\chi^{(te_1)}(z) > \infty$ for $t>\tilde m_z$.  Therefore, $\tilde m_z$ coincides with the quantity $m_{G_z}$ defined in \eqref{eq:def_mS}.
We denote this quantity $\tilde m_z=m_{G_z}$ henceforth by $m_z$.
We know from \cite[Corollary~4.1.15]{MS93} that
\begin{equation}
\label{eq:mass_asymp}
    \lim_{z\to z_c}m_z = 0,
\end{equation}
so all results of Section~\ref{sec:critical} also apply.
By comparing the OZ decay of Corollary~\ref{cor:OZ} with \eqref{eq:tildemSAW}, we see that $| x |_z$ coincides with the norm $|x|_{G_z}$ of Corollary~\ref{cor:norm}.

We list the results with comparison to previous literature:
\begin{itemize}
\item
Precise asymptotics are given in Theorem~\ref{thm:crossover}:
Uniformly in  $z\in[z_c-\delta,z_c)$, for some $\eps >0$,
\begin{equation} \label{eq:SAW-asymp-intro}
G_z(x)
= \C(x; \eta_{\hat x,z} , \Lambda_{\hat x,z}) e^{-\mz  |x|_z} [ 1 + O(|x|^{-\eps})]
\end{equation}
as $\abs x \to \infty$,
with $|\eta_{\hatxz}| \asymp m_z$ and $\hat x \cdot \Lam_{\hatxz}\inv \hat x\asymp 1$.

\item
The error term in \eqref{eq:SAW-asymp-intro} remains uniform in the limit $z\to z_c$.  In this limit, $m_z\to 0$, $\eta_{\hat x,z}\to 0$,
and $\Lambda_{\hat x,z} \to (\sigma_{z_c}^2/d){\rm Id}$,
with $\sigma_{z_c}^2 = \sum_{x\in \Zd}|x|^2 J_{z_c}(x)$.  It then follows from
\eqref{eq:Ccrit} that
\begin{equation}
\label{eq:SAWcrit}
G_{z_c}(x) =
    \frac{d\Gamma(\frac{d-2}{2})}{2\pi^{d/2}}
    \frac{1}{\sigma_{z_c}^2|x|^{d-2}}
    [ 1 + O(|x|^{-\eps})].
\end{equation}
This critical decay was proved previously in \cite{Hara08,LS24a}.

\item
By Corollary~\ref{cor:OZ}, $G_z$ exhibits the OZ decay
\begin{equation}
\label{eq:SAW-OZ}
G_z(x)
=
\frac 1 { (2\pi)^{(d-1)/2} \sqrt{\det\Lambda_{\hatxz}} }
	\frac{1}{ (\hat x \cdot \Lam_{\hatxz}\inv \hat x)^{1/2} }
    \frac{ \abs {\eta_{\hatxz}}^{(d-3)/2}}{ \abs  x^{(d-1)/2}}
    e^{-\mz  \abs x_z}[1+o(1)].
\end{equation}
It was proved decades ago that, for all $d \ge 2$ and all
\emph{fixed} $z\in (0,z_c)$, there exists  $c_{d,z,\hat x}>0$ such that
\begin{equation}
\label{eq:SAW-CC}
    G_z(x) \sim c_{d,z,\hat x} \frac{1}{|x|^{(d-1)/2}}e^{-\mz |x|_z}
\end{equation}
as $|x|\to \infty$,
first for $x$ on axis \cite{CC86,CC86b,MS93} and later for all $x \in \Rd$ \cite{Ioff98}.
However, the results of \cite{CC86,CC86b,Ioff98,MS93} do not control the precise $z$-dependence of the prefactor
$c_{d,z,\hat x}$ present in \eqref{eq:SAW-OZ}, do not include the uniformity in $z$ present in \eqref{eq:SAW-asymp-intro}, and do not include the critical decay for $z=z_c$.

\item
By Theorem~\ref{thm:critical_limit}, Euclidean invariance of the norm emerges at the critical point:
\begin{equation} \label{eq:crit_limit-SAW}
\abs x_z = \norm x_2 [ 1 + O(\mz ^{ 2 } ) ]
	\qquad (z\to z_c).
\end{equation}

\item
It is proved in \cite{HS92a,HS92b} that, for any $\eps < \frac{d-4}{2}\wedge 1$,
\begin{equation}
\label{eq:chimSAW}
    \chi(z)
    =
    \frac{1}{\partial_z \hat J_{z_c}(0)} \frac{1}{1-z/z_c}[1+O(1-z/z_c)^\eps]
    	\qquad (z\to z_c).
\end{equation}
By Corollary~\ref{cor:chim}, we therefore have
\begin{equation}
\mz^2 = \frac{2d \, \partial_z \hat J_{z_c}(0)}{\sigma_{z_c}^2} (1-z/z_c) [ 1+O(1-z/z_c)^\eps ]
	\qquad (z\to z_c),
\end{equation}
which agrees with \cite{HS92a,HS92b}.

\item
By Corollary~\ref{cor:crossover_weak},
\begin{equation}
\label{eq:SAW-asymp}
    G_z(x) \asymp
    \frac{    \max\{ 1 , \mz  |x|_z \} ^{(d-3)/2} } { \abs x_z^{d-2} }
	e^{- m_z| x|_z}
	\qquad (z \in [z_c-\delta,z_c),\ |x|_z \ge R).
\end{equation}
Such near-critical bounds have been studied previously for the self-avoiding walk \cite{Slad23_wsaw, Liu25EJP, DP25a},
the Ising model \cite{DP25-Ising}, and percolation \cite{DP25b, HMS23}.  The bound \eqref{eq:SAW-asymp} improves on previous bounds
of this type by having the correct exponential decay rate
in both the upper and lower bounds, rather than having different
(undetermined) constants in the exponents.

\item
By Theorem~\ref{thm:xiphi} and \eqref{eq:chimSAW}, for any $\phi>0$ the correlation length of order $\phi$ is given asymptotically by
\begin{equation}
\label{eq:xiphiSAW}
    \xi_\phi(z)
    \sim
    \Bigl(\frac{A_\phi}{A_0}\Bigr)^{1/\phi} \frac 1 \mz
    \sim
    \Bigl(\frac{A_\phi}{2}\Bigr)^{1/\phi}
    \frac{\sigma_{z_c} }{(2d \, \partial_z \hat J_{z_c}(0))^{1/2}} \frac{1}{(1-z/z_c)^{1/2}}.
\end{equation}
This extends the result proved in \cite{HS92a,HS92b} for the special case $\phi=2$ to all $\phi>0$, and exposes the limiting universal ratio for $\xi_\phi/\xi_\psi$.
A variant of \eqref{eq:xiphiSAW} was proved
for $4$-dimensional
continuous-time weakly self-avoiding walk and $n$-component $|\varphi|^4$ spin models (all $n \ge 1$) in \cite{BSTW-clp}.  Despite the logarithmic corrections present for $d=4$,
the same limiting universal ratio for $\xi_\phi/\xi_\psi$ occurs.
\end{itemize}

We restate Theorem~\ref{thm:SAW-intro}  in more explicit detail as Theorem~\ref{thm:SAW}.

\begin{theorem}
\label{thm:SAW}
(i)
The function $G_z$ obeys Assumption~\ref{ass:Omega}
for all $d \ge 2$ and $z \in (0,z_c)$.
\\
(ii)
For $d \ge 5$, and for any fixed $\zeta \in (0,d-4]$.
there exist $\delta,M,\KIR >0$ such that,
uniformly in $z\in [z_c-\delta,z_c)$ and $\mu\in\overline\Omega_z$,
\begin{gather}
\label{eq:Q_moments-SAW}
\norm{ J_z\supmu(y) }_{1} ,\,
\bignorm{ \abs y ^{2+\zeta}  J_z\supmu(y) }_{1}
	\le M ,
\\
\label{eq:Q_infrared-SAW}
\Re[ \hat J_z\supmu(0)  - \hat J_z\supmu(k) ]
	\ge \KIR \abs k ^2 	
	\qquad (k\in \Td) ,
\\
\label{eq:pnorm-SAW-ii}
\bignorm{ \abs y^{d-1} J_z\supmu(y) }_{ \frac d 3 \wedge 2} \le M .
\end{gather}
\end{theorem}

The easy proof of Theorem~\ref{thm:SAW}(i) is given in Section~\ref{sec:SAWi}.
The proof of Theorem~\ref{thm:SAW}(ii) requires more effort.  We prove the bounds
\eqref{eq:Q_moments-SAW}--\eqref{eq:Q_infrared-SAW} via minor extensions of the proofs of
\cite[Theorem~2.5]{HS92a} and \cite[Lemma~3.10]{HS92a}.
Those results involve tilting only by on-axis $\mu$, and
our task is to extend to general $\mu$.
 We do so in Section~\ref{sec:SAWboot}.

On the other hand, the bound \eqref{eq:pnorm-SAW-ii} has not appeared previously
in the literature.  Its proof requires a new diagrammatic estimate, which we
carry out in Section~\ref{sec:SAW-Pip}.
In fact, a stronger bound on the $L^p$ norm of $\abs y^{d-1} J_z\supmu(y)$ is proved in Proposition~\ref{prop:SAW-Pip}:
it is finite whenever $p\ge 1$ satisfies $p > \frac{2d}{3d - 6}$
(this is stronger because $\frac{2d}{3d - 6} < \frac d3 \wedge 2$ for all $d>4$ and $L^p(\Zd)$ norms are decreasing in $p$).

\subsection{Proof of Theorem~\ref{thm:SAW}(i)}
\label{sec:SAWi}

\begin{proof}[Proof of Theorem~\ref{thm:SAW}(i)]
Let $d \ge 2$ and $z \in (0,z_c)$.
By definition, $G_z$ is a $\Zd$-symmetric function.
As discussed below \eqref{eq:tildemSAW},
$\chi^{(te_1)}(z)$ is finite for $t<m_z$ ($=\tilde m_z$) and is infinite
for $t>m_z$.

To complete the verification of Assumption~\ref{ass:Omega}, it remains to show that $\Omega_z$ is an open set.  As in the proof of Theorem~\ref{thm:RW},
we apply a Simon--Lieb argument.  Let $B_R =\{y\in\Zd : \|y\|_\infty \le R\}$, and
let $\partial B_R=\{y\in\Zd : \|y\|_\infty = R\}$ be the boundary
of the box $B_R$.
For any $R \ge 1$, any $x\not\in B_R$, and any $\mu \in \R^d$, we have
(as in \cite[(6.5.8)]{MS93})
\begin{equation}
\label{eq:SAW-LS}
    G\supmu_z(x) \le \sum_{y\in \partial B_R}G_z\supmu(y)G_z\supmu(x-y).
\end{equation}
Suppose now that $\mu\in\Omega_z$.
Since $\chi\supmu (z) <\infty$, it follows from \eqref{eq:SAW-LS} and
\cite[Lemma~A.1]{MS93} that there are constants $c,C>0$ (depending on $z$ and $\mu$)
such that
$G_z\supmu(x) \le C e^{-c\|x\|_2}$ for all $x$.  Therefore, by the Cauchy--Schwarz inequality,
\begin{equation} \label{eqref:SL_pf}
    G_z^{(\mu+\nu)}(x)
    = G_z^{(\mu)}(x) e^{\nu \cdot x}
    \le C e^{-(c-\|\nu\|_2)\|x\|_2}.
\end{equation}
Summation then shows that $\chi^{(\mu+\nu)}(z) <\infty$ when $\|\nu\|_2<c$, so $\Omega_z$ is an open set.
\end{proof}

\subsection{Bootstrap argument: proof of \eqref{eq:Q_moments-SAW}--\eqref{eq:Q_infrared-SAW}}
\label{sec:SAWboot}

The (open) \emph{bubble diagram} is defined by
$
    \bubble(z) = \sum_{x\neq 0} G_z(x)^2
$.
We first prove a version of \eqref{eq:Q_moments-SAW} which is conditional on a bound on the tilted bubble diagram
\begin{equation} \label{eq:def_bubble}
    \Bsupmu(z) = \sum_{x\neq 0} G\supmu_z(x)^2
    = \sum_{x\neq 0} \bigl( G_z(x) e^{\mu\cdot x} \bigr)^2
    .
\end{equation}
Although the proof of the bound is initially conditional,
it will be established unconditionally in due course via a
bootstrap argument varying $\mu \in \Omega_z$.
The fact that the bubble tilted by $\mu \in \overline\Omega_z$ can be finite
may be understood intuitively from the fact (which we ultimately prove) that
$G_z\supmu(x)$ has exponential decay $\exp[-(\muxz-\mu)\cdot x]$, and due to
the optimality of $\muxz$ the decay rate is
positive except for the rare $x$ with $\muxz=\mu$.

For $\mu = 0$, it is proved in \cite[Theorem~1.1]{HS92b}
that $\bubble(z_c) \le 0.493$ for all $d\ge 5$.
This is important for the results of \cite{HS92a}, e.g., it
proves the statement of \cite[Theorem~2.5]{HS92a} that there is a positive constant $\kappa_0$, satisfying $\kappa_0(1+\kappa_0) < 1$, such that
\begin{equation} \label{eq:bubble^0}
    \B^{(0)}(z_c) = \B(z_c) \le \kappa_0.
\end{equation}

\begin{lemma} \label{lem:SAW_diagram-p1}
Let $d \ge 5$, $z \le z_c$, $\mu \in \Rd$, and $a\in [0,d-2]$.
Suppose there is a constant $\kappa > 0$, with $\kappa(1+\kappa) < 1$,
such that $\Bsupmu(z)\le \kappa$.
Then
there is a constant $K_{a,1}$ (depending on $\kappa$) such that
\begin{equation} \label{eq:Pi_moments}
\bignorm{ \abs x^a  \Pi_z \supmu(x) }_1  \le  K_{a,1} .
\end{equation}
\end{lemma}

\begin{proof}
A proof for the case $a > 0$ and $\mu$ on-axis is given in \cite[Proposition~3.2]{Liu25EJP}.
The same proof works for general tilt under our assumption that $\Bsupmu(z)\le \kappa$, with the only difference being that we now distribute $e^{\mu \cdot x}$ instead of $e^{m x_1}$ along one side of the diagram.
For $a=0$, we need to additionally estimate the $N=1$ term of the lace expansion. It is given explicitly as $\Pi_{z,1}(x)=\delta_{0,x} 2dz G_z(e_1)$, and we simply use the bound
\begin{equation}
\bignorm{ \abs x^a \Pi_{z,1}(x)e^{\mu\cdot x} }_1
\le \1_{a=0} 2dz_c G_{z_c}(e_1) ,
\end{equation}
which is finite by \cite{HS92a}.
This concludes the proof.
\end{proof}

Since $J_z\supmu = 2dz \Dnn\supmu + \Pi_z \supmu$, an immediate consequence
of Lemma~\ref{lem:SAW_diagram-p1} is that
\begin{equation} \label{eq:J_moments}
\bignorm{ \abs x^a  J_z \supmu(x)  }_1
\le 2dz_c e^{\abs \mu} + K_{a,1}
\end{equation}
under the same hypotheses.

\begin{lemma} \label{lem:SAW_infrared}
Under the hypotheses of Lemma~\ref{lem:SAW_diagram-p1},
for any $\eps \in (0,1]$.
there exists $C_\eps > 0$
for which
\begin{equation}
\bigabs{ \Re[ \hat J_z\supmu (0) - \hat J_z\supmu(k) ]
-[ \hat J_z (0) - \hat J_z(k)  ] }
\le  C_\eps \abs \mu^\eps e^{\abs \mu} ( 1 - \hat \Dnn(k) )
	\qquad (k\in \Td) .
\end{equation}
\end{lemma}

Since $J_z = 2dz \Dnn + \Pi_z$ satisfies the infrared bound \cite[Theorem~2.8]{HS92a}
\begin{equation} \label{eq:SAW_infrared}
\hat J_z (0) - \hat J_z(k)
\ge ( 2d z - C_5 ) ( 1 - \hat \Dnn(k) )
\end{equation}
with $2d z_c - C_5 >0$, an immediate consequence of Lemma~\ref{lem:SAW_infrared} is that
\begin{equation} \label{eq:SAW_tilted_infrared}
\Re[ \hat J_z\supmu (0) - \hat J_z\supmu(k) ]
\ge (2d z - C_5 - C_\eps  \abs \mu^\eps e^{\abs \mu} ) ( 1 - \hat \Dnn(k) )
\gtrsim 1 - \hat \Dnn(k)
\end{equation}
when $z_c - z$ and $\abs \mu$ are sufficiently small.

\begin{proof}[Proof of Lemma~\ref{lem:SAW_infrared}]
Using $\Zd$ symmetry, we can write
\begin{equation} \label{eq:SAW_boot_pf00}
\Re [ \hat J_z\supmu(0) - \hat J_z\supmu (k) ]
- [ \hat J_z(0) - \hat J_z(k) ]
=  \sum_{x\in \Z^d} ( 1 - \cos (k\cdot x) ) J_z(x) (\cosh (\mu\cdot x) - 1) .
\end{equation}
Since $1 - \hat \Dnn(k) \asymp \abs k^2$,
it suffices to prove that the above is bounded in absolute value by $O(\abs \mu^\eps e^{\abs \mu} \abs k^2)$.
We use the elementary inequalities
\begin{equation}
\abs{ 1 - \cos (k\cdot x)  } \lesssim \abs{ k \cdot x }^2 ,
\qquad
\abs{ \cosh (\mu\cdot x) - 1 } \lesssim \abs{ \mu\cdot x}^\eps \cosh (\mu\cdot x)
\quad (\eps \le 2)
\end{equation}
to bound the right-hand side of \eqref{eq:SAW_boot_pf00} in absolute value by
\begin{equation}
O(\abs \mu^\eps) \abs k^2 \sum_{x\in \Z^d} \abs x^{2+\eps}
	\abs{ J_z(x) } \cosh(\mu \cdot x)
= O(\abs \mu^\eps) \abs k^2 	
	\sum_{x\in \Z^d} \abs x^{2+\eps}  \abs{ J_z \supmu (x) } .
\end{equation}
For any $\eps \le 2$ with $2 + \eps \le d-2$,
it follows from \eqref{eq:J_moments} that
$\sum_{x\in \Zd} \abs x^{2+\eps}  \abs{ J_z \supmu (x) }
\lesssim e^{\abs \mu } $.
This gives the desired result with $\eps \le 1$
in dimensions $d \ge 5$.
\end{proof}

Our choice of $\delta>0$ occurs in the next lemma.
We will apply Lemma~\ref{lem:SAW_bubble} with a small value of $\gamma$.

\begin{lemma} \label{lem:SAW_bubble}
Let $d \ge 5$.
For any $\gamma > 0$, there is a $\delta = \delta(\gamma) > 0$ such that for $z \in [z_c - \delta, z_c)$ and $\mu \in \Omega_z$,
\begin{equation}
\Bsupmu(z) \le \B^{(0)}(z_c) + \gamma .
\end{equation}
\end{lemma}

\begin{proof}
The proof uses a bootstrap argument varying $\mu \in \Omega_z$.
We begin by noting that $\Bsupmu(z)$ is a continuous function in $\mu \in \Omega_z$ when $z < z_c$ is fixed, since $G_z\supmu(x)$ decays exponentially as noted above \eqref{eqref:SL_pf}.

Let $\gamma > 0$, $\eps \in ( 0, 1]$, and $z < z_c$.
Since $\Bsupmu(z)$ is continuous in $\mu \in \Omega_z$, and since $\Omega_z$ is connected (by convexity from Lemma~\ref{lem:Omega_convex}), the image of the map $\mu \mapsto \Bsupmu(z)$ must be connected.
Since $\B\supzero(z) \le \B\supzero(z_c)$ also, it suffices to prove that there is a $\delta > 0$ such that, if $z \in [z_c - \delta, z_c)$ and $\mu \in \Omega_z$, then
\begin{equation} \label{eq:bootstrap_claim}
\Bsupmu(z) \le \B\supzero(z_c) + 2 \gamma
\quad \implies \quad
\Bsupmu(z) \le \B\supzero(z_c) + \gamma .
\end{equation}

We write $F_z = \delta_0 - J_z = \delta_0 - 2dz \Dnn - \Pi_z$.
Given $\Bsupmu(z) < \B\supzero(z_c) + 2 \gamma$ with $\gamma$ sufficiently small, we can use Lemmas~\ref{lem:SAW_diagram-p1}--\ref{lem:SAW_infrared}
with $\kappa = \kappa_0 + 2 \gamma$
to estimate $\Pi_z$ and $J_z$.
Since $\hat F_z \supmu(0) = 1/  \chi\supmu(z)  \ge 0$,
we have
\begin{align} \label{eq:SAW_boot_pf0}
\abs{ \hat G_z \supmu (k)}
&=
\frac{ 1 }{ \abs{ \hat F_z \supmu (k) }}
\le  \frac{ 1}{ \Re [ \hat F_z\supmu(k) - \hat F_z\supmu (0) ]  } 	\nl
&=
\frac{ 1 }{ \hat F_z  (k)  } +
\Biggl[
\frac 1{ \hat F_z(k) - \hat F_z(0) } - \frac{ 1 }{ \hat F_z  (k)  }
\Biggr]
	+ \frac{ [\hat F_z(k) - \hat F_z(0) ] - \Re [ \hat F_z\supmu(k) - \hat F_z\supmu (0) ] }
		{  [\hat F_z(k) - \hat F_z(0) ] \Re [ \hat F_z\supmu(k) - \hat F_z\supmu (0) ] } .
\end{align}
We take the $L^2$ norm and apply the triangle inequality.
By Parseval's identity,
the $L^2$ norm of the left-hand side is $\sqrt{ 1 + \Bsupmu(z)}$, and the $L^2$ norm of the first term on the right-hand side is $\sqrt{ 1 + \B\supzero(z) }$.
For the second term on the right-hand side, by the infrared bound \eqref{eq:SAW_infrared} we have
\begin{equation} \label{eq:SAW_bubble_pf}
\biggnorm{ \frac 1{ \hat F_z(k) - \hat F_z(0) }
	- \frac 1{ \hat F_z(k) } }_2
\lesssim \hat F_z(0)
	\biggnorm{ \frac 1{ \abs k^2 (\abs k^2 + \hat F_z(0) ) } }_2
\lesssim
\begin{cases}
\hat F_z(0) ^{1 - \frac{8-d}4}
	&( 4 < d <  8)
\\
\hat F_z(0) \abs{\log \hat F_z(0)} & (d=8)
\\
\hat F_z(0) & (d>8),
\end{cases}
\end{equation}
where the divergence of the integral for $d \le 8$ can
be computed via the change of variables $k \mapsto \hat F_z(0)^{1/2} k$
(see \cite[Lemma~4.2]{HS00a}).
For the last term on the right-hand side, using Lemma~\ref{lem:SAW_infrared}, we know
its numerator is bounded by $O(\abs \mu^\eps e^{\abs \mu} (1 - \hat \Dnn(k)))$,
and its denominator is $\gtrsim (1 - \hat \Dnn(k))^2$ when $\mu$ is sufficiently small by \eqref{eq:SAW_tilted_infrared}.
Since $\hat F_z (0) \to 0$ and
 $\abs \mu \lesssim \mz  \to 0$ as $z \to z_c$ by Lemma~\ref{lem:Omega_convex} and \eqref{eq:mass_asymp},
we obtain
\begin{equation}
\sqrt{ 1 + \Bsupmu(z)}
\le \sqrt{ 1 + \B\supzero(z)} + o(1)
	+ O(\abs \mu^\eps e^{\abs \mu})
		\biggnorm{ \frac 1{ 1 - \hat P(k) } }_2 ,
\end{equation}
which implies $\Bsupmu(z) \le \B\supzero(z_c) + o(1)$ as $z \to z_c$.
This gives \eqref{eq:bootstrap_claim} and completes the proof.
\end{proof}

\begin{corollary}
\label{cor:SAW-bubble}
Let
$d \ge 5$.  There exist $\delta > 0$ and  $\kappa > 0$,
satisfying $\kappa(1+ \kappa) <1$,
such that,
uniformly in $z \in [z_c - \delta, z_c)$ and in $\mu \in \overline \Omega_z$,
\begin{equation}
\Bsupmu(z)
\le  \kappa.
\end{equation}
\end{corollary}

\begin{proof}
Recall the constant $\kappa_0$ from \eqref{eq:bubble^0}.
We take $\gamma >0$ sufficiently small so that $\kappa = \kappa_0 + \gamma$ satisfies $\kappa(1+ \kappa) < 1$ still.
Then Lemma~\ref{lem:SAW_bubble} produces a $\delta > 0$ for which
\begin{equation}
\sup_{ z \in [z_c - \delta, z_c) } \sup_{\mu \in \Omega_z}
	\Bsupmu(z)
\le \B^{(0)}(z_c) + \gamma
\le \kappa_0 + \gamma
= \kappa .
\end{equation}
This uniform estimate can be extended to all $\mu \in \bar \Omega_z$, by Fatou's Lemma.
\end{proof}

\begin{proof}[Proof of \eqref{eq:Q_moments-SAW}--\eqref{eq:Q_infrared-SAW}]
Let $\delta$ and $\kappa$ be given by Corollary~\ref{cor:SAW-bubble}.
We can now use Lemmas~\ref{lem:SAW_diagram-p1}--\ref{lem:SAW_infrared} with this $\kappa$.
The norm estimates \eqref{eq:Q_moments-SAW} follow immediately from Lemma~\ref{lem:SAW_diagram-p1} and \eqref{eq:J_moments}, since $2+\zeta \le d-2$
under our assumption that $\zeta \in (0,d-4]$.
For the infrared bound,
since $\abs \mu \lesssim \mz  \to 0$ as $z \to z_c$ by Lemma~\ref{lem:Omega_convex} and \eqref{eq:mass_asymp},
the desired \eqref{eq:Q_infrared-SAW} follows immediately from Lemma~\ref{lem:SAW_infrared} and \eqref{eq:SAW_tilted_infrared}, by taking a smaller $\delta > 0$ if necessary.
\end{proof}

\subsection{Diagrammatic estimate: proof of \eqref{eq:pnorm-SAW-ii}}
\label{sec:SAW-Pip}

The following proposition involves a new diagrammatic estimate for self-avoiding walk.
Compared to Lemma~\ref{lem:SAW_diagram-p1}, it allows us to take $a > d-2$, at the cost of having the $L^p$ norm instead of the $L^1$ norm.
The constant $\delta >0$ is given by Corollary~\ref{cor:SAW-bubble}.

\begin{proposition} \label{prop:SAW-Pip}
Let $d \ge 5$ and $a \in [0,2d-4]$.
For any choice of $p \in [1,\infty]$ satisfying
\begin{equation} \label{eq:pa-restriction}
\frac 1 p < 1 + \frac{ \frac 3 2 d - 4 - a } d ,
\end{equation}
there is a constant $K_{a,p} < \infty$ such that
\begin{equation}
\label{eq:Pi_moments-pnorm}
\bignorm{ \abs x^a  \Pi_z \supmu(x)  }_p  \le  K_{a,p}
\end{equation}
uniformly in $z \in [z_c - \delta, z_c)$ and in $\mu \in \overline \Omega_z$.
For $a = 2d-4$, \eqref{eq:Pi_moments-pnorm} also holds with $p = 2$.
\end{proposition}

\begin{proof}[Proof of \eqref{eq:pnorm-SAW-ii} assuming Proposition~\ref{prop:SAW-Pip}]
Let $z \in [z_c - \delta, z_c)$ and $\mu \in \overline \Omega_z$.
Fix $p\ge 1$ satisfying $p > \frac{2d}{3d - 6}$.  As discussed below
the statement of Theorem~\ref{thm:SAW}, it suffices to show that
$\norm{\abs x^{d-1}  J_z \supmu(x)}_p$ is uniformly bounded.
Since $J_z\supmu = 2dz \Dnn\supmu + \Pi_z \supmu$, the triangle inequality gives
\begin{equation} \label{eq:SAW_pnorm_pf}
    \bignorm{ \abs x^{d-1}  J_z \supmu(x)  }_{p}
    \le
    2dz_c \bignorm{ \abs x^{d-1}  P \supmu(x)  }_{p}
    +
    \bignorm{ \abs x^{d-1}  \Pi_z \supmu(x)  }_{p} .
\end{equation}
The first term on the right-hand side is bounded by $2dz_c \norm{ \abs x^{d-1}  P \supmu(x)  }_1 \le 2dz_c e^{\norm \mu_\infty}$,
which is finite because $\norm \mu_\infty \le \mz  \le \tilde m_{z_c - \delta}$ by \eqref{eq:mass_asymp}.
The second term is also bounded, because our condition on $p$ is exactly the hypothesis on $p$ in Proposition~\ref{prop:SAW-Pip} when $a = d-1$.
This shows that both terms of \eqref{eq:SAW_pnorm_pf} are uniformly bounded and concludes the proof.
\end{proof}

\begin{proof}[Proof of Proposition~\ref{prop:SAW-Pip}]
We assume familiarity with
diagrammatic estimates for  the lace expansion \cite{MS93,Slad06}.
Since $L^p(\Zd)$ norms are decreasing in $p$, the $a \le d-2$ case follows directly from Lemma~\ref{lem:SAW_diagram-p1} and Corollary~\ref{cor:SAW-bubble}.
We therefore assume $a > d-2$, in which case we can write $a = d-2 + s$ with $s \in (0, d-2]$.
Also, since the $1$-loop diagram for $\Pi_{z,1}(x)$ contains a factor $\delta_{0,x}$,
it does not contribute to \eqref{eq:Pi_moments-pnorm}  when $a > 0$.

For the $2$-loop diagram $\Pi_{z,2}(x)$, using H\"older's inequality,
its weighted and tilted version obeys the bound
\begin{align}
    \bignorm{|x|^a e^{\mu\cdot x}\Pi_{z,2}(x)}_p
    & \le
    \bignorm{ \1_{x\neq 0} |x|^a e^{\mu\cdot x} G_z(x)^3 }_p
    \nl & =
    \bignorm { |x|^{d-2}G_z(x) \cdot |x|^{s}G_z(x) \cdot \1_{x\neq 0} G\supmu_z(x) }_p
    \nl & \le
    \bignorm {|x|^{d-2}G_z(x)}_\infty \bignorm{|x|^{s}G_z(x)}_q \bignorm{\1_{x\neq 0} G\supmu_z(x) }_2 ,
\label{eq:SAW2loop}
\end{align}
where $p\inv = \infty\inv + q \inv + 2 \inv$.
Since $G_z(x) \le G_{z_c}(x) \lesssim \nnnorm x^{-(d-2)}$ by \cite{Hara08},
the above is finite (in fact, $O(\sqrt{\kappa})$) when $q$ satisfies
\begin{equation}
\frac 1 q < \frac{ d - 2 - s } d ,
\end{equation}
with the equality allowed when $s = d-2$ (in which case $q=\infty$).
This translates to the restriction
\begin{equation}
\frac 1 p
< \frac 1 2 +  \frac{ d - 2 - s } d
= 1 +  \frac{ \half d - 2 - s } d
=1 + \frac{ \frac 3 2 d - 4 - a } d
\end{equation}
on $p$, with the equality allowed when $a = 2d-4$,
since $a = d-2+s$.
This is exactly condition \eqref{eq:pa-restriction}.

For the remaining $N$-loop diagrams with $N \ge 3$,
the diagram for $\Pi_{z,N}(x)$
has three edge-disjoint paths from $0$ to $x$.
We distribute the exponential tilt $e^{\mu \cdot x}$ multiplicatively along the top path.
Also, as explained in Step~1 of \cite[Section~3.2]{Hara08}, we can weight the bottom path in
the diagram with $\abs x^{d-2}$ and weight an adjacent diagonal line in the same loop with $\abs x^{s}$.
Using the natural ordering of the lines given by the walk, two cases can happen for the relative locations of the two weighted edges, as follows.
Let $\ell_s$ and $\ell_{d-2}$ denote
the line number of the $|x|^s$-weighted line and
the $|x|^{d-2}$-weighted line, respectively.
These two line numbers necessarily differ by $1$.
If $\ell_s=\ell_{d-2}-1$ then we group the two lines with the line numbered $\ell_s-1$.
If instead $\ell_s=\ell_{d-2}+1$ then we group the two lines with line numbered $\ell_s+1$.
In either case, the third line in the group lies on the top of the diagram.
Since all diagrams have an odd number of lines in total, there are an even number of lines
preceding (to the left of)
the group, and there are also an even number of lines following
(to the right of) the group.  In this way, the diagram decomposes as
\begin{equation}
\label{eq:Dcal-decomp}
    \Dcal(x) =\sum_{a,b\in \Zd}\sum_{c,d\in \Zd} \Lcal(a,b) \Zcal(a,b,c,d) \Rcal(c,d,x).
\end{equation}
See Figure~\ref{fig:SAWdiagram}.

\begin{figure}[h]
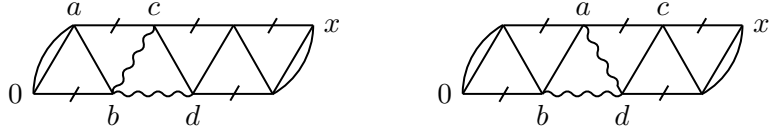
 
\center{
\SAWOne   \qquad   \SAWTwo
\caption{
Two scenarios in decomposing $\abs x^{d-2+s} \Pi_z^{(8)}(x)$.
The wavy lines are weighted by $\abs x^{d-2}$ or $\abs x^{s}$.
In the first scenario, $\Zcal(a,b,c,d)$ consists of the three lines
$ac$, $cb$, $bd$ forming a ``Z'' shape.  In the second scenario,
$\Zcal(a,b,c,d)$ consists of the three lines
$ca$, $ad$, $db$ forming a ``backwards Z'' shape.
}
\label{fig:SAWdiagram}
}
\end{figure}

By the triangle inequality, we have
\begin{align}
    \| \Dcal(x)\|_{L^p_x}
    & \le \sum_{a,b} \Bignorm{ \Lcal(a,b) \sum_{c,d}\Zcal(a,b,c,d) \Rcal(c,d,x) }_{L^p_x}
    \nnb & \le
    \Bigl(\sum_{a,b} \Lcal(a,b)\Bigr)
    \sup_{a,b}\,\Bignorm{ \sum_{c,d}\Zcal(a,b,c,d) \Rcal(c,d,x) }_{L^p_x}.
\end{align}
By translation invariance,
\begin{equation}
    \Zcal(a,b,c,d) \Rcal(c,d,x)
    =
    \Zcal(a-x,b-x,c-x,d-x) \Rcal(c-x,d-x,0),
\end{equation}
and by the change of summation variables $c'=c-x$ and $d'=d-x$, we have
\begin{align}
    \Bignorm{ \sum_{c',d'}\Zcal(a-x,b-x,c',d') \Rcal(c',d',0) }_{L^p_x}
    & \le
    \sup_{c',d'}\,
    \norm{\Zcal(a-x,b-x,c',d') }_{L^p_x}
    \sum_{c',d'} \Rcal(c',d',0) .
\end{align}
Altogether, the above gives
\begin{equation}
    \| \Dcal(x)\|_{L^p_x}
    \le
    \Bigl(\sum_{a,b} \Lcal(a,b)\Bigr)
    \Bigl(\sup_{a,b,c',d'}
    \norm{  \Zcal(a-x,b-x,c',d') }_{L^p_x} \Bigr)
    \sum_{c',d'} \Rcal(c',d',0) .
\end{equation}
The first and last factors on the right-hand side are bounded by standard methods with $L^2$ norms on each line.
Each line contributes either $\sqrt \kappa$ or $\sqrt{ 1 + \kappa}$, as in
the proof of \cite[Proposition~3.2]{Liu25EJP}.
Overall, this argument gives
\begin{equation}
\bignorm{ \abs x^{d-2+s} \Pi_{z,N}(x) e^{\mu\cdot x} }_{L^p_x}
\le
C_a
N^{a+1} (\sqrt \kappa)^{N-1} (\sqrt{1+\kappa})^{N-3}
	\sup_{a,b,c',d'}
    \norm{  \Zcal(a-x,b-x,c',d') }_{L^p_x},
\end{equation}
for some $C_a>0$ and for all $N\ge 3$ (the power of $N$ is computed as in Step~3
of \cite[Section~3.2]{Hara08}).
Since $\kappa(1+\kappa)<1$, summation of the above over $N$ would give the desired \eqref{eq:Pi_moments-pnorm},
once we show that the term involving $\Zcal$ is finite for
$p\ge 1$ satisfying \eqref{eq:pa-restriction}.
We prove this next.

Depending on the two possible orderings of the weighted edges, there are the following two possibilities for $\Zcal$ corresponding to the left/right diagram of Figure~\ref{fig:SAWdiagram}:
\begin{equation}
\begin{aligned}
\Zcal_1 &=
	G_z \supmu(a-x-c')
	\cdot \abs{ b-x-c' }^s G_z( b-x-c' )
	\cdot \abs{ b-x-d' }^{d-2} G_z( b-x-d' ) ,
\\
\Zcal_2 &=
	G_z \supmu(a-x-c')
	\cdot \abs{ a-x-d' }^s G_z( a-x-d' )
	\cdot \abs{ b-x-d' }^{d-2} G_z( b-x-d' ) .
\end{aligned}
\end{equation}
We define $q\in [1,\infty]$ by $p\inv = 2 \inv + q \inv$.
Using H\"older's inequality and translation invariance, for both $i = 1,2$ we get
\begin{align}
\norm{ \Zcal_i(a-x,b-x,c',d') }_{L^p_x}
&\le
\bignorm{ G_z \supmu(x) }_2
\bignorm{ \abs x^s G_z(x) }_q
\bignorm{ \abs x^{d-2} G_z(x) }_\infty  \nl
&\le
\sqrt{1+\kappa}\,
\bignorm{ \abs x^s G_z(x) }_q
\bignorm{ \abs x^{d-2} G_z(x) }_\infty .
\end{align}
Apart from the appearance of $\sqrt{1+\kappa}$ here, rather than $\sqrt  \kappa$,
this is bounded as in \eqref{eq:SAW2loop} for the $2$-loop diagram.
We therefore have the same restrictions on $p$ and $q$,
and the proof is complete.
\end{proof}

\section{Percolation}
\label{sec:perc}

We now verify the hypotheses and apply the results of Sections~\ref{sec:introduction}--\ref{sec:critical}
to the nearest-neighbour percolation on $\Zd$ in dimensions $d \ge 15$.

\subsection{Main result}
\label{sec:perc-mr}

Let $d \ge 2$ and $p\in[0,1]$.
We consider nearest-neighbour Bernoulli bond percolation
on $\Zd$ with
occupation probability $p$: bonds are independently \emph{occupied}
with probability $p$ and \emph{vacant} with probability $1-p$.
The \emph{two-point function} is defined by
\begin{equation}
    \tau_p(x) = \P_p(0 \leftrightarrow x).
\end{equation}
There is a critical value $p_c\in (0,1)$ such that the \emph{susceptibility}
$\chi(p) = \sum_{x\in\Zd}\tau_p(x)$
is finite if and only if $p<p_c$.  For background on percolation, see \cite{Grim99}.

For high dimensions and for $p \in [0,p_c]$,
the lace expansion \cite{HS90a,Slad06,HH17book} provides an OZ equation
\begin{equation}
\label{eq:OZperc}
    \tau_p = \g_p + J_p*\tau_p
\end{equation}
with
\begin{equation}
    \g_p=\delta_0 + \Pi_p, \qquad J_p = 2dp P * \g_p
    = 2dp P + 2dp P * \Pi_p.
\end{equation}
The function $P(x) = \frac{1}{2d} \1\{ \|x\|_1=1 \}$ is the nearest-neighbour kernel, and $\Pi_p(x)$ is an infinite-range
$\Zd$-symmetric function which takes both positive and negative values.
For
dimensions $d \ge 11$, control of $\Pi_p$ uniformly in $p \in [0,p_c]$
(including the critical point $p_c$) is obtained in \cite{FH17,HS90a,Hara08}.

Our main result for percolation is the following theorem.
The proof of Theorem~\ref{thm:perc-intro}(i) exactly parallels the
proof for self-avoiding walk in Section~\ref{sec:SAWi}, and we do not comment
further on it.

\begin{theorem}
\label{thm:perc-intro}
(i) The function $\tau_p$ obeys Assumption~\ref{ass:Omega}
for all $d \ge 2$ and $p \in (0,p_c)$.
\\
(ii)
For nearest-neighbour percolation in dimensions
$d \ge 15$,
and for any fixed $\zeta \in (0,\frac 12 (d-6))$,
there exist $\delta,M,\KIR >0$ such that
$(J\supmu_p,\g\supmu_p)\in \Qcal_{M,\KIR,\zeta}$
uniformly in $p\in [p_c-\delta,p_c)$ and $\mu\in\overline\Omega_p$.
\end{theorem}

We expect Theorem~\ref{thm:perc-intro}(ii) to remain valid for all
$d>6$, but our use of the lace expansion for the nearest-neighbour model
requires us to take $d \ge 15$, in order to apply the numerical results proved in \cite{FH17}.
It would be of interest to prove Theorem~\ref{thm:perc-intro} for spread-out models in dimensions $d>6$.  However, our proof of Theorem~\ref{thm:perc-intro}(ii) uses diagrams that can converge only for $d>8$, so new ideas would be needed.
One example of a diagram requiring $d>8$ is the weighted tilted bubble diagram in Lemma~\ref{lem:Gmu1}.

It follows from
elementary observations (e.g., \cite[Proposition 3.2]{ACC90})
that there is a norm $|\cdot|_p$ on $\Rd$, satisfying $\|v\|_\infty \le |v|_p \le \|v\|_1$ for every $v\in \Rd$, such that the limit
\begin{equation} \label{eq:tildemperc}
    \tilde m_p = \lim_{|x|_p\to \infty}\frac{-\log \tau_p(x)}{|x|_p}
\end{equation}
exists in $(0,\infty)$, and
\begin{equation}
    \tau_p(x) \le e^{-\tilde m_p|x|_p}.
\end{equation}
Consequently, $\chi^{(te_1)}(p) < \infty$ for $t<\tilde m_p$, and
$\chi^{(te_1)}(p) > \infty$ for $t>\tilde m_p$.  Therefore, $\tilde m_p$ coincides with the quantity $m_{\tau_p}$ defined in \eqref{eq:def_mS}.
We denote this quantity $\tilde m_p=m_{\tau_p}$ henceforth by $m_p$.
It is a standard fact (see \cite[Chapter~6]{Grim99}) that
\begin{equation}
\label{eq:mass_asymp-perc}
    \lim_{p\to p_c}m_p = 0,
\end{equation}
so all results of Section~\ref{sec:critical} also apply.
By comparing the OZ decay of Corollary~\ref{cor:OZ} with \eqref{eq:tildemperc}, we see that $| x |_p$ coincides with the norm $|x|_{\tau_p}$ of Corollary~\ref{cor:norm}.

In the following, we consider nearest-neighbour percolation in dimensions $d \ge 15$.
All limits $p\to p_c$ are left limits from the subcritical side $p<p_c$.
We list the consequences of Theorem~\ref{thm:perc-intro} with comparison to previous literature:
\begin{itemize}
\item
Precise asymptotics are given in Theorem~\ref{thm:crossover}:
Uniformly in  $p\in[p_c-\delta,p_c)$, for some $\eps >0$,
\begin{equation} \label{eq:perc-asymp-intro}
\tau_p(x)
= \C(x; \eta_{\hat x,p} , \Lambda_{\hat x,p}) e^{-m_p  |x|_p}
	[  \hat h_p^{(\mu_{\hatx,p})}(0)  + O(|x|^{-\eps})]
\end{equation}
as $\abs x \to \infty$,
with $|\eta_{\hat x,p}| \asymp m_p$ and $\hat x \cdot \Lam_{\hatx,p}\inv \hat x\asymp 1$.
Aspects of \eqref{eq:perc-asymp-intro} also appear in the forthcoming work \cite{Blan27}.

\item
The error term in \eqref{eq:perc-asymp-intro} remains uniform in the limit $p\to p_c$.  In this limit, $m_p\to 0$, $\eta_{\hat x,p}\to 0$,
and $\Lambda_{\hat x,p} \to (\sigma_{p_c}^2/d){\rm Id}$,
with $\sigma_{p_c}^2 = \sum_{x\in \Zd}|x|^2 J_{p_c}(x)$.  It then follows from
\eqref{eq:Ccrit} that
\begin{equation}
\label{eq:perccrit}
\tau_{p_c}(x) =
    \frac{d\Gamma(\frac{d-2}{2})}{2\pi^{d/2}}
    \frac{1}{\sigma_{p_c}^2|x|^{d-2}}
    [ \hat h_{p_c}(0) + O(|x|^{-\eps})].
\end{equation}
This critical decay was proved previously in \cite{Hara08,LS24a}.

\item
By Corollary~\ref{cor:OZ}, $\tau_p$ exhibits the OZ decay
\begin{equation}
\label{eq:perc-OZ}
\tau_p(x)
=
\frac 1 { (2\pi)^{(d-1)/2} \sqrt{\det\Lambda_{\hat x,p}} }
	\frac{1}{ (\hat x \cdot \Lam_{\hat x,p}\inv \hat x)^{1/2} }
    \frac{ \abs {\eta_{\hat x,p}}^{(d-3)/2}}{ \abs  x^{(d-1)/2}}
    e^{-\mp  \abs x_p}[ \hat h_p^{(\mu_{\hatx,p})}(0)  +o(1)].
\end{equation}
It was proved earlier that, for all $d \ge 2$ and all
\emph{fixed} $p\in (0,p_c)$, there exists  $c_{d,p,\hat x}>0$ such that
\begin{equation}
\label{eq:perc-CCC}
    \tau_p(x) \sim c_{d,p,\hat x} \frac{1}{|x|^{(d-1)/2}}e^{-m_p |x|_p}
\end{equation}
as $|x| \to\infty$,
first for $x$ on axis \cite{CCC91} and later for all $x \in \Rd$ \cite{CI02,CIV08}.
However, the results of \cite{CCC91,CI02,CIV08} do not control the precise $p$-dependence of the prefactor
$c_{d,p,\hat x}$ present in \eqref{eq:perc-OZ}, do not include the uniformity in $p$ present in \eqref{eq:perc-asymp-intro}, and do not include the critical decay for $p=p_c$.

\item
By Theorem~\ref{thm:critical_limit}, Euclidean invariance of the norm emerges at the critical point:
\begin{equation} \label{eq:crit_limit-perc}
\abs x_p = \norm x_2 [ 1 + O(m_p^{ 2 } ) ]
	\qquad (p\to p_c).
\end{equation}

\item
It follows from Corollary~\ref{cor:crossover_weak} that
\begin{equation}
\label{eq:SAW-asymp}
    \tau_p(x) \asymp
    \frac{    \max\{ 1 , m_p  |x|_p \} ^{(d-3)/2} } { \abs x_p^{d-2} }
	e^{- m_p| x|_p}
	\qquad (p \in [p_c-\delta,p_c),\ |x|_p \ge R).
\end{equation}
This improves on previous bounds \cite{DP25b, HMS23}
of this type by having the correct exponential decay rate
in both the upper and lower bounds, rather than having different
(undetermined) constants in the exponents.

\item
In dimensions $d\ge11$ (or for spread-out models with $d>6$),
it was recently proved in \cite{Blan27} (see also \cite{BRN26})
that there are constants $C_\chi,C_\xi \in (0,\infty)$ such that,
as $p\to p_c$ and for any $\eps < \frac{d-6}{2}\wedge 1$,
\begin{equation} \label{eq:chim_perc}
    \chi(p)
    =
    \frac{C_\chi }{1-p/p_c}[1+O(1-p/p_c)^\eps],
    	\qquad
    \xi_2(p)^2
    =
    \frac{C_{\xi_2} }{1-p/p_c} [1+O(1-p/p_c)^\eps].
\end{equation}
This refines $\chi(p)\asymp (1-p/p_c)^{-1} \asymp \xi_2(p)^2$ from \cite{HS90a,FH17} (see also \cite{DP25b,PS27}), by providing an asymptotic formula with error estimate.
By Lemma~\ref{lem:xi2-chi-ratio}, \emph{the} correlation length $\xi(p)=m_p^{-1}$ therefore satisfies
\begin{equation} \label{eq:mpchi}
\xi(p) = \frac { \xi_2(p)} {\sqrt{2d}}
\Bigl[ 1 +O(1-p/p_c)^{\varepsilon'} \Bigr]
	\qquad (p\to p_c)
\end{equation}
for $d\ge 15$
and for any $\varepsilon' < \frac{d-6}{4}\wedge 1$.
With \eqref{eq:chim_perc}, this
refines the bounds $\xi(p)^2 \asymp (1-p/p_c)\inv$ from \cite{Hara90} to an asymptotic formula with error estimate.
It also follows from Lemma~\ref{lem:xi2-chi-ratio} that
$    \hat h_{p_c}(0)C_{\xi_2} = \sigma_{p_c}^2 C_\chi $,
which computes the ratio of the
amplitudes in \eqref{eq:chim_perc}.
More generally, we also know from Theorem~\ref{thm:xiphi} that
for any $\phi>0$ the correlation length of order $\phi$ is given asymptotically by
\begin{equation}
\label{eq:xiphi-perc}
    \xi_\phi(p)
    \sim
    \Bigl(\frac{A_\phi}{A_0}\Bigr)^{1/\phi} \xi(p) 
    	\qquad (p\to p_c),
\end{equation}
which reveals the limiting universal ratios for all correlation lengths.
\end{itemize}

Next, we show that Theorem~\ref{thm:perc-intro}(ii) reduces to the following theorem.
By definition,
\begin{equation}
\Omega_p = \Bigl\{ \mu \in \Rd : \chi\supmu(p)=
	 \sum_{y\in \Zd} \tau_p(y) e^{\mu\cdot y} < \infty \Bigr\} .
\end{equation}

\begin{theorem} \label{thm:percPi}
For $d \ge 15$
and for any fixed $\zeta \in (0,\half(d-6))$,
there exist $\delta,M,\KIR >0$ such that,
uniformly in $p\in [p_c-\delta,p_c)$ and $\mu\in\overline\Omega_p$,
\begin{align}
\label{eq:Q_moments-perc-a}
\bignorm{ \abs y ^{u}  \Pi_p\supmu(y) }_{1}
	&\le M \qquad (u=0,2+\zeta),
\\
\label{eq:Q_infrared-perc-a}
\Re[ \hat J_p\supmu(0)  - \hat J_p\supmu(k) ]
	&\ge \KIR \abs k ^2 	
	\qquad (k\in \Td) ,
\\
\label{eq:pnorm-perc-ii-a}
    \bignorm{ \abs y^{d-1} \Pi_p\supmu(y) }_{ 2} &\le M .
\end{align}
\end{theorem}

\begin{proof}[Proof of Theorem~\ref{thm:perc-intro}(ii)]
By Definition~\ref{def:A}, uniform bounds sufficient for $(J\supmu_p,\g\supmu_p)\in \Qcal_{M,\KIR,\zeta}$
  are:
\begin{gather}
\label{eq:Q_moments-perc-pf}
\norm{ J_p\supmu(y) }_{1} ,\,
\bignorm{ \abs y ^{2+\zeta}  J_p\supmu(y) }_{1}
	\le M ,
\\
\label{eq:g_moments-perc-pf}
\norm{ \g_p\supmu(y) }_{1} ,\,
	\bignorm{ \abs y ^{\zeta} \g_p\supmu(y) }_{1} ,\,
    \bignorm{ \abs y ^{2}  \g_p\supmu(y) }_{1} \,
	\le M ,
\\
\label{eq:Q_infrared-perc-pf}
\Re[ \hat J_p\supmu(0)  - \hat J_p\supmu(k) ]
	\ge \KIR \abs k ^2 	
	\qquad (k\in \Td) ,
\\
\label{eq:pnorm-perc-ii-pf}
\bignorm{ \abs y^{d-1} J_p\supmu(y) }_{ 2},\,
    \bignorm{ \abs y^{d-1} \g_p\supmu(y) }_{ 2} \le M .
\end{gather}
The infrared bound \eqref{eq:Q_infrared-perc-pf} is repeated in \eqref{eq:Q_infrared-perc-a}.
The four bounds on $\g\supmu_p=\delta_0+\Pi_p\supmu$ in \eqref{eq:g_moments-perc-pf} and \eqref{eq:pnorm-perc-ii-pf} follow immediately from the
bounds on $\Pi_p\supmu$ in \eqref{eq:Q_moments-perc-a} and \eqref{eq:pnorm-perc-ii-a}.

For the $L^1$ bounds on $J\supmu_p = 2dp\Dnn\supmu * h_p\supmu$ in \eqref{eq:Q_moments-perc-pf}, we apply $\abs x^u \le 2^u ( \abs {x-y}^u + \abs y^u )$ to obtain
\begin{align} \label{eq:J_moments-perc}
\bignorm{ \abs x^a  J_p\supmu (x)  }_1
&\le 2dpe^{\norm \mu_\infty} \bignorm{ \abs x^u (P * h_p \supmu)(x) }_1 \nl
&\le 2dpe^{\norm \mu_\infty} 2^u
\Bigl(
	\bignorm{ \abs x^u P (x) }_1 \norm{ h_p \supmu }_1
	+ \norm{ P }_1  \bignorm{ \abs x^u h_p \supmu (x) }_1	\Bigr).
\end{align}
This is finite because
the norms involving $\Dnn$ are both $1$, the norms of $\g\supmu$ have been bounded in the previous paragraph, $p \le p_c$, and $\|\mu\|_\infty \le m_{p_c-\delta} < \infty$.

Finally, the $L^2$ bound on $\g\supmu_p$ in \eqref{eq:pnorm-perc-ii-pf} follows from the $L^2$ norm on $\Pi\supmu_p$ in \eqref{eq:pnorm-perc-ii-a}.
The $L^2$ bound on $J\supmu_p$ in \eqref{eq:pnorm-perc-ii-pf} then follows via a small modification of \eqref{eq:J_moments-perc}, using $\|f*g\|_2 \le \|f\|_1 \|g\|_2$.
This completes the proof.
\end{proof}

It remains to prove Theorem~\ref{thm:percPi}.
We prove the $L^1$ bounds \eqref{eq:Q_moments-perc-a}
for $\zeta < \half(d-8)$
and the infrared bound  \eqref{eq:Q_infrared-perc-a}  in Section~\ref{sec:perc-tilt},
by adapting the bootstrap argument in Section~\ref{sec:SAWboot} for the self-avoiding walk.
This argument is somewhat different from \cite{Hara90} which obtained the bounds for $\zeta=0$ and on-axis $\mu$.
The improvement of the $L^1$ bound \eqref{eq:Q_moments-perc-a} to
$\zeta < \half(d-8)$, and
the $L^2$ bound \eqref{eq:pnorm-perc-ii-a}, require a new diagrammatic estimate for percolation, which we
prove in Section~\ref{sec:perc-Pip}
following the approach of \cite{DL26}.

\subsection{Bootstrap argument: proof of \eqref{eq:Q_moments-perc-a}--\eqref{eq:Q_infrared-perc-a}}
\label{sec:perc-tilt}

We first prove \eqref{eq:Q_moments-perc-a} for
$\zeta < \half(d-8)$. At the end of Section~\ref{sec:perc-tilt},
we discuss the improvement to
$\zeta < \half(d-6)$.

We define the tilted triangle diagrams
\begin{equation}
\Triangle \supmu(p)
	= \norm{ \tau_p\supmu * \tau_p * \tau_p - \delta_0 }_\infty ,
\qquad
\tilde \Triangle \supmu(p)
	= 2dp \norm{ P * \tau_{p}\supmu * \tau_{p} * \tau_{p} }_\infty .
\end{equation}
For $\mu = 0$, it is proved in \cite{HS90a} that $\Triangle \supzero(p_c) = O(1/d)$ and $p_c = 1/(2d) + O(1/d^2)$ as $d \to \infty$.
Since
\begin{equation}
\tilde \Triangle \supzero(p)
\le 2dp \bigl( \norm P_1 \Triangle \supzero(p)  + \norm{ P }_\infty \bigr)
\le O(1/d)
\end{equation}
too, we can take $d_0$ sufficiently large so that
\begin{equation} \label{eq:triangle^0}
\Triangle\supzero(p_c) \le 0.43 ,
\qquad
\tilde \Triangle \supzero(p_c) \le 0.18
\end{equation}
for all $d \ge d_0$.
By the Mathematica notebook of \cite{FH17}, $d_0=15$ is sufficient.
In particular, we have
\begin{equation}
\tilde \Triangle \supzero(p_c)
(1 + 2\Triangle \supzero(p_c) )
\le (0.18)(1.86)
= 0.3348 .
\end{equation}
Another important ingredient is the pointwise bound
\begin{equation}
\label{eq:taupc}
    \tau_{p_c}(x) \lesssim |x|^{-(d-2)},
\end{equation}
which has been proved for $d \ge 11$ \cite{Hara90,FH17,LS24a}.
We also define the tilted (open) bubble diagram
\begin{equation}
\label{eq:perc-bubble}
    \Bsupmu(p) = \sum_{x\neq 0} (\tau_p(x) e^{\mu\cdot x})^2.
\end{equation}

\begin{lemma} \label{lem:perc_diagram-p1}
Let $d \ge 11$, 
$p \le p_c$, $\mu \in \Rd$, and $a\in [0,\frac{d-4}2)$.
Suppose there are constants $r < 1$ and $\kappa >0$ such that
\begin{equation} \label{eq:perc_r}
e^{\norm \mu_\infty}
\tilde \Triangle\supmu(p)
	(1 + 2 \Triangle\supmu(p) )
	\le r < 1
\end{equation}
Then
there is a constant $K_{a,1}$ (depending on $r, \kappa$) such that
\begin{equation} \label{eq:Pi_moments-perc}
\bignorm{ \abs x^a  \Pi_p \supmu(x) }_1  \; \le  K_{a,1} .
\end{equation}
\end{lemma}

\begin{proof}
The proof is a slight modification of the proof of \cite[Lemma~3.2]{Hara90}.
Our hypothesis \eqref{eq:perc_r} is slightly weaker than the original hypothesis
\begin{equation}
 e^{\norm \mu_\infty}
	\Bigl( p + 2dp \Triangle\supmu(p) \Bigr)
	(1 + 2 \Triangle\supmu(p) )
	\le r < 1
\end{equation}
in \cite[(3.2)]{Hara90};
the adequacy of the weaker hypothesis was observed
by Hara and is described in \cite[(7.2)]{FH17}.
We assume familiarity with diagrammatic estimates for the percolation lace
expansion (see, e.g., \cite{HS90a,Hara90,MS93,Slad06}).
The lace expansion gives a representation of $\Pi_p(x)$ as an alternating sum of nonnegative terms $\Pi_{p,N}(x)$, summed over $N \ge 0$.

We first consider the simplest contribution
$\Pi_{p,0}(x)$, which is bounded above by $\1_{x\neq 0}\tau_p(x)^2$.  Its weighted version therefore obeys
\begin{equation}\label{eq:percN0}
    \bignorm{ \abs x^a  \Pi_{p,0} \supmu(x) }_1
     \le
    \bignorm{ \abs x^a \tau_p(x) }_2
    \bignorm{ \1_{x\neq 0}\tau_p\supmu(x) }_2.
\end{equation}
The second factor on the right-hand side is bounded above by $\sqrt{\kappa}$.  Since
$\tau_{p_c}(x) \lesssim |x|^{-(d-2)}$ by \eqref{eq:taupc}
and our assumption that $d \ge 11$, the first factor on the right-hand side is finite exactly when $a$ satisfies our assumption
that $a < \frac{d-4}{2}$.

\begin{figure}[h]
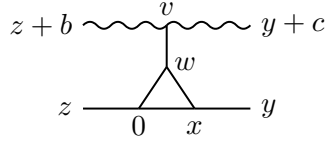
 
\center{
\PercH
\caption{
Diagrammatic representation for $H_p^{(\mu),a}$, before multiplying by $|x|^a$, summing over $u,w,x,y,z$, and taking the supremum over $b,c$.
The wavy lines are tilted.
}
\label{fig:perc_diagram}
}
\end{figure}

We omit the details for $N=1$ and
sketch the argument for $\Pi_{p,N}$ for general $N \ge 2$.
In the diagrams for $\Pi_{p,N}(x)$, there are always two edge-disjoint paths from $0$ to $x$. We distribute $e^{\mu\cdot x}$ multiplicatively along the top path, and distribute $\abs x^a$ additively along the bottom path.
Factors $p e^{\|\mu\|_\infty}$ arising from $2dp P\supmu$ are bounded by $r < 1$.
As in \cite[(3.3)]{Hara08},
for all $N\ge 2$, we obtain a bound
\begin{equation} \label{eq:perc_diagram_pf}
\bignorm{ \abs x^a  \Pi_{p,N} (x) e^{\mu\cdot x} }_1
\lesssim N^{a+1} r^{N-2} 	\Bigl(
	\bignorm{ [\abs x^a \tau_p (x) ]  * \tau_p \supmu }_\infty
	+ H^{(\mu),a}_p \Bigr) ,
\end{equation}
where, as depicted in Figure~\ref{fig:perc_diagram},
\begin{align} \label{eq:percH}
H^{(\mu),a}_p & = \sup_{b,c\in \Zd} \sum_{v,w,x,y,z\in \Zd}
	\abs x^a
    \Big( \tau_p(z) \tau_p(x)\tau_p(w)\tau_p(x-w)\tau_p(y-x)
    \tau_p(v-w)
    \nl & \hspace{40mm} \times \tau\supmu_p(z+b-v)\tau\supmu_p(y+c-v) \Big).	
\end{align}
Since $r < 1$ by our hypotheses, we can sum \eqref{eq:perc_diagram_pf} over $N$ to get the desired \eqref{eq:Pi_moments-perc}, provided that the two terms in the parentheses in \eqref{eq:perc_diagram_pf} are finite.
The first term is bounded just as in \eqref{eq:percN0}.
For
the other term, a standard diagrammatic estimate gives
\begin{align}
H^{(\mu),a}_p
&\le   \norm{ \tau_p \supmu * \tau_p  }_\infty
\norm{ \tau_p \supmu * \tau_p * \tau_p }_\infty
	\bigl( [\abs x^a \tau_p(x) ]  * \tau_p * \tau_p \bigr)(0)   .
\end{align}
The first two factors on the right-hand side are bounded by
\begin{equation}
\norm{ \tau_p \supmu }_2 \norm{ \tau_p * \tau_p }_2
\le \sqrt{1+\kappa}  \norm{ \tau_p * \tau_p }_2 ,
\end{equation}
which is finite in dimensions $d \ge 11 > 8$ by the critical decay \eqref{eq:taupc}.
For the final factor,
the three-fold convolution of $\abs x^{a-(d-2)}$ with two factors $\abs x^{-(d-2)}$ is bounded when $a<d-6$.
This gives the desired result for all $a < \frac {d-4}2  \wedge (d-6) =  \frac {d-4}2$, using $d \ge 11$. 
This completes the proof.
\end{proof}

\begin{lemma} \label{lem:perc_infrared}
Suppose the hypotheses of Lemma~\ref{lem:perc_diagram-p1} hold.
Then
for any $\eps \in (0,2\wedge \frac{d-8}{2})$
there exists $C_\eps > 0$
such that
\begin{equation}
\bigabs{ \Re[ \hat J_p\supmu (0) - \hat J_p\supmu(k) ]
-[ \hat J_p (0) - \hat J_p(k)  ] }
\le  C_\eps \abs \mu^\eps e^{\abs \mu} ( 1 - \hat \Dnn(k) )
	\qquad (k\in \Td) .
\end{equation}
\end{lemma}

\begin{proof}
The proof is the same as the proof of Lemma~\ref{lem:SAW_infrared}.
We can use Lemma~\ref{lem:perc_diagram-p1} with any $\eps \le 2$ such that $2 + \eps < \frac{d-4}2 = 2 + \frac{d-8}2$.
\end{proof}

For $p \in [\frac{1}{2d}, p_c)$, we know from \cite[Lemma~4.5]{HS90a}
that $J_p = 2dp P * ( \delta_0 + \Pi_p )$ satisfies the infrared bound
\begin{equation} \label{eq:perc_infrared}
\hat J_p (0) - \hat J_p(k)
\ge (1 - O(d\inv)) ( 1 - \hat \Dnn(k) ) .
\end{equation}
Therefore, Lemma~\ref{lem:perc_infrared} implies that
(under the hypotheses of Lemma~\ref{lem:perc_diagram-p1})
\begin{equation} \label{eq:perc_tilted_infrared}
\Re[ \hat J_p\supmu (0) - \hat J_p\supmu(k) ]
\ge [1 - O(d\inv) - O( \abs \mu^\eps e^{\abs \mu}) ] ( 1 - \hat \Dnn(k) )
\gtrsim 1 - \hat \Dnn(k)
\end{equation}
when $\abs \mu$ is sufficiently small.  Since $\lim_{p\to p_c}m_p=0$,
and since $\|\mu\|_\infty \le m_p$ when $\mu \in \Omega_p$, we can achieve
small $\abs \mu$ by taking $p$ close to $p_c$.
Also, when also $\mu \in \Omega_p$ we have $\chi \supmu(p) < \infty$, so
\begin{equation} \label{eq:Jz^mu}
1 - \hat J_p \supmu(0)
= \frac{ 1 + \hat \Pi_p\supmu(0) }{ \chi\supmu(p)  }
= \frac{ 1 + \hat \Pi_p(0) - O(\abs \mu^2) }{ \chi\supmu(p)  }
\ge 0 .
\end{equation}
Since $\chi\supmu(p) \ge \chi(p) \to \infty$ as $p \to p_c$, it follows that
$1 - \hat J_p \supmu(0)$ can also be made as small as desired by taking $p$ close to $p_c$.

For the next lemma, we define
\begin{equation}
\tau_\sym \supmu (x) = \half ( \tau_p\supmu(x) +  \tau_p\supmu(-x) )
= \tau_p(x) \cosh( \mu\cdot x).
\end{equation}
The lemma states that $\hat \tau_\sym \supmu$ is close to $\hat \tau_p$ in the Fourier space. We are not aware of a previous estimate of this form.

\begin{lemma} \label{lem:Gsym}
Suppose the hypotheses of Lemma~\ref{lem:perc_diagram-p1} hold.
Let $p\in [p_c-\delta_1, p_c)$ and $\mu \in \Omega_p$, with $\delta_1$ chosen
small enough that \eqref{eq:perc_tilted_infrared} holds.
Then there are (explicit) functions $f_j: [0,\infty) \to [0,\infty)$, satisfying $f_j(t) \to 0$ as $t \to 0$, for which
\begin{equation}
\norm{ \hat \tau_\sym \supmu - \hat \tau_p }_3
\le f_1(\abs \mu) + f_2( 1 - \hat J_p \supmu(0) )
	+  f_3( 1 - \hat J_p(0) ) .
\end{equation}
\end{lemma}

\begin{proof}
Since $\mu \in \Omega_p$, by tilting the OZ equation \eqref{eq:OZperc} we have
\begin{equation}
\hat \tau_\sym \supmu (k)
= \Re \hat \tau_p \supmu (k)
= \Re \biggl( \frac{ 1 + \hat \Pi_p \supmu (k) } { 1 - \hat J_p \supmu (k) } \biggr) .
\end{equation}
With $ \hat \Pi_p \supmu = \hat \Pi_p + ( \hat \Pi_p \supmu - \hat \Pi_p )$
and by Lemma~\ref{lem:perc_diagram-p1} and \eqref{eq:perc_tilted_infrared}, this gives
\begin{align}
\abs{ \hat \tau_\sym \supmu(k) - \hat \tau_p(k) }
&\le  \biggabs{ \Re \biggl( \frac{ 1 + \hat \Pi_p (k) } { 1 - \hat J_p \supmu (k) } \biggr)
	- \hat \tau_p(k) }
	+ \frac{  \norm{ \hat \Pi_p \supmu - \hat \Pi_p }_\infty }
		{ \abs{ 1 - \hat J_p \supmu (k) } } 		\nl
&\le \biggabs{ \Re \biggl( \frac{ 1 + \hat \Pi_p (k) } { 1 - \hat J_p \supmu (k) } \biggr)
	- \hat \tau_p(k) }
	+ \frac{  O( \abs \mu^2 )  } { \abs k^2 } .
\end{align}

The $L^3$ norm of the last term is finite and of order $\abs \mu^2$ in dimensions $d > 6$; we put this as part of the function $f_1(\abs \mu)$.
For the first term, since $\hat \Pi_p (k)$ is real (because $\Pi_p(x)$ is symmetric), we have
\begin{equation}
\biggnorm{ \Re \Biggl( \frac{ 1 + \hat \Pi_p } { 1 - \hat J_p \supmu } \biggr)
	- \hat \tau_p } _3
\le \norm{ 1 + \hat \Pi_p }_\infty
	\biggnorm{ \Re \Biggl( \frac1 { 1 - \hat J_p \supmu } \biggr)
	- \frac1 { 1 - \hat J_p } } _3 .
\end{equation}
Since $\norm{ \hat \Pi_p(k) }_\infty \le \norm{ \Pi_p (x) }_1 \lesssim 1$,
it suffices to bound each of the norms
\begin{equation} \label{eq:Gsym_pf1}
\biggnorm{ \Re \biggl( \frac{ 1 } { 1 - \hat J_p \supmu } \biggr)
	- \frac{ 1 } { \Re[ \hat J_p \supmu(0) - \hat J_p \supmu ] } }_3 ,
\qquad
\biggnorm{ \frac{ 1 } { 1 - \hat J_p  }
	- \frac{ 1 } { \Re[ \hat J_p \supmu(0) - \hat J_p \supmu ] } }_3
\end{equation}
in terms of $\abs \mu$, $1 - \hat J_p \supmu(0)$, and $1 - \hat J_p(0)$.
For the second norm, the required estimate has essentially appeared in \eqref{eq:SAW_boot_pf0} for the self-avoiding walk, with the only difference being the $L^3$ norm here versus the $L^2$ norm there.
We therefore omit its details, which lead to $f_3$ in the desired upper bound.

For the first norm of \eqref{eq:Gsym_pf1},
we use the elementary identity, for $w=u+iv$ and $z=x+iy$, that
\begin{equation}
    \Re \frac 1w  - \frac{1}{\Re z}
    = \frac{u}{\abs w^2} - \frac 1 x
    =
   - \frac{(u-x)u +v^2}{x|w|^2}.
\end{equation}
With the infrared bound \eqref{eq:perc_tilted_infrared}, this gives
\begin{multline} \label{eq:Gsym_pf2}
\biggabs{
\Re \biggl( \frac{ 1 } { 1 - \hat J_p \supmu(k) } \biggr)
-
\frac{ 1 } { \Re[ \hat J_p \supmu(0) - \hat J_p \supmu(k) ] }
}
= \biggabs{
	\frac{ [ 1 - \hat J_p \supmu(0)] \Re [1 -  \hat J_p \supmu (k)]
		+ \abs{\Im \hat J_p\supmu(k) }^2 }
	{ \Re[ \hat J_p \supmu(0) - \hat J_p \supmu(k) ]
		\abs{ 1 - \hat J_p \supmu (k) }^2}
	} 	\\
\lesssim  \frac{ 1 - \hat J_p \supmu(0) }{ \abs k^2 ( 1 - \hat J_p \supmu(0) + \abs k^2 )}
	+ \frac{  \abs{\Im \hat J_p\supmu(k) }^2 }
		{ \abs k^2 ( \abs k^4 +  \abs{\Im \hat J_p\supmu(k) }^2 ) } .
\end{multline}
We take the norm of the two terms separately.
Similarly to \eqref{eq:SAW_bubble_pf} for the self-avoiding walk, the norm of the first term obeys
\begin{equation}
\biggnorm{ \frac{ 1 - \hat J_p \supmu(0) }
	{ \abs k^2 ( \abs k^2 +  [1 - \hat J_p \supmu(0)] )} }_3
\lesssim
\begin{cases}
(1 - \hat J_p \supmu(0)) ^{1 - \frac{12-d}6}
	&( 6 < d < 12)
\\
(1 - \hat J_p \supmu(0)) \abs{\log (1 - \hat J_p \supmu(0))} & (d=12)
\\
(1 - \hat J_p \supmu(0)) & (d>12),
\end{cases}
\end{equation}
which defines the function $f_2(1 - \hat J_p \supmu(0))$.
For the norm of the last term of \eqref{eq:Gsym_pf2},
we use that the function $t \mapsto t / (\abs k^4 + t)$ is increasing in $t\ge0$ for each $k$,
with $t = \abs{\Im \hat J_p\supmu(k) }^2 \lesssim \abs \mu^2$, to get
\begin{equation}
\biggnorm{ \frac{  \abs{\Im \hat J_p\supmu(k) }^2 }
		{ \abs k^2 ( \abs k^4 +  \abs{\Im \hat J_p\supmu(k) }^2 ) } }_3
\lesssim \biggnorm{ \frac{ \abs \mu^2 }
		{ \abs k^2 ( \abs k^4 +   \abs \mu^2 ) } }_3
\lesssim
\begin{cases}
\abs \mu^{2 ( 1 - \frac{18-d}{12} ) }
	&( 6 < d < 18)
\\
\abs \mu^2 \abs{\log \abs \mu^2   } & (d=18)
\\
\abs \mu^2 & (d>18) .
\end{cases}
\end{equation}
This
gives a part of the function $f_1(\abs \mu)$ and
concludes the proof.
\end{proof}

Our choice of $\delta>0$,
and the need to take $d\ge 15$ (rather than $d\ge11$),
occur in the next lemma.
We will apply Lemma~\ref{lem:perc_triangle} with a small value of $\gamma$.

\begin{lemma} \label{lem:perc_triangle}
Let $d \ge 15$.
For any $\gamma > 0$, there is a $\delta = \delta(\gamma) \in (0,\delta_1)$ such that for $p \in [p_c - \delta, p_c)$ and $\mu \in \Omega_p$,
\begin{equation}
\begin{aligned}
\label{eq:perc_triangle}
\Triangle\supmu(p) &\le 2 \Triangle^{(0)}(p_c) + \gamma , \\
\tilde\Triangle\supmu(p) &\le 2 \tilde \Triangle^{(0)}(p_c) + \gamma , \\
\Bsupmu(p) &\le \B^{(0)}(p_c) + \gamma .
\end{aligned}
\end{equation}
\end{lemma}

\begin{proof}
The proof uses a bootstrap argument varying $\mu \in \Omega_p$.
Let $\gamma > 0$, $\eps \in ( 0, 2\wedge \frac{d-8}{2} )$, and $p \in [p_c-\delta_1, p_c)$.
All of $\Triangle \supmu(p)$, $\tilde \Triangle \supmu(p)$, and $\Bsupmu(p)$ are continuous function in $\mu \in \Omega_p$ when $p < p_c$ is fixed, since $\tau_p\supmu(x)$ decays exponentially.
Since $\Omega_p$ is connected (by convexity from Lemma~\ref{lem:Omega_convex}), the image of the map $\mu \mapsto (\Triangle \supmu(p),\tilde \Triangle \supmu(p), \Bsupmu(p))$ must be a connected subset of $\R^3$.
Since $\Triangle \supzero(p) \le \Triangle \supzero(p_c)$,
$\tilde\Triangle \supzero(p) \le \tilde\Triangle \supzero(p_c)$, and
$\B\supzero(p) \le \B\supzero(p_c)$ also, it suffices to prove that there is a $\delta > 0$ for which if $p \in [p_c - \delta, p_c)$ and $\mu \in \Omega_p$ then
\begin{equation} \label{eq:perc_boot}
\begin{cases}
\Triangle\supmu(p) \le 2 \Triangle^{(0)}(p_c) + 2\gamma  \\
\tilde\Triangle\supmu(p) \le 2 \tilde \Triangle^{(0)}(p_c) + 2\gamma \\
\Bsupmu(p) \le \B^{(0)}(p_c) + 2\gamma
\end{cases}
\implies \quad
\begin{cases}
\Triangle\supmu(p) \le 2 \Triangle^{(0)}(p_c) + \gamma  \\
\tilde\Triangle\supmu(p) \le 2 \tilde \Triangle^{(0)}(p_c) + \gamma \\
\Bsupmu(p) \le \B^{(0)}(p_c) + \gamma .
\end{cases}
\end{equation}

Fix  $r \in [0.98,1)$.
Given the bootstrap hypothesis (and $d \ge 15$),
we can use \eqref{eq:triangle^0} to estimate \eqref{eq:perc_r} by
\begin{align}  \label{eq:perc_boot_pf}
e^{\norm \mu_\infty}
\tilde \Triangle\supmu(p)
	(1 + 2 \Triangle\supmu(p) )
&\le  e^{\norm \mu_\infty}
	(2\tilde \Triangle\supzero(p_c) + 2\gamma )
	(1 + 4 \Triangle\supzero(p_c) + 4\gamma ) 	\nl
&\le  e^{\norm \mu_\infty} [
	2(0.18)(1+4(0.43) )
	+O(\gamma) +8\gamma^2 ] \nl
&=  e^{\norm \mu_\infty} [
	0.9792
	+O(\gamma) +8\gamma^2 ] .
\end{align}
Since $\norm \mu_\infty \le m_p  \to 0$ as $p \to p_c$,
the above is less than $r$ when $\gamma$ and $\delta$ are sufficiently small.
We therefore can use Lemmas~\ref{lem:perc_diagram-p1}--\ref{lem:Gsym} with this $r$ and with $\kappa = \B^{(0)}(p_c) + 2\gamma$.
By taking a smaller $\delta$, we can apply
Lemma~\ref{lem:Gsym} to  all $p \in [p_c - \delta, p_c)$ and all $\mu \in \Omega_p$.

To estimate $\Triangle\supmu(p) $,
we first use
\begin{equation}
(\tau_\sym\supmu * \tau_p * \tau_p)(x)
= \half \Bigl( (\tau_p\supmu * \tau_p * \tau_p)(x)
	+ (\tau_p\supmu * \tau_p * \tau_p)(-x) \Bigr)
\end{equation}
to see that $\Triangle\supmu(p) \le 2 \norm{ \tau_\sym\supmu * \tau_p * \tau_p - \delta_0 }_\infty$.
By the inverse Fourier transform,
\begin{equation}
( \tau_\sym\supmu * \tau_p * \tau_p - \delta_0 )(x)
=  ( \tau_p * \tau_p * \tau_p - \delta_0 )(x)
	+  \int_\Td \frac{\D k }{(2\pi)^d}	e^{ik\cdot x}
		[\hat \tau_\sym\supmu (k) - \hat \tau_p(k)] [\hat \tau_p(k)]^2 .
\end{equation}
Taking the supremum over $x\in \Zd$ then gives
\begin{equation}
\norm{ \tau_\sym\supmu * \tau_p * \tau_p - \delta_0 }_\infty
\le  \Triangle  \supzero(p)
	+ \norm{ \hat \tau_\sym \supmu - \hat \tau_p }_3  \norm{\hat \tau_p}_3^2 .	
\end{equation}
We observed below \eqref{eq:Jz^mu} that $1 - \hat J_p \supmu(0) \to 0$ as $p \to p_c$.
Therefore, using Lemma~\ref{lem:Gsym} we obtain
\begin{equation}
\Triangle \supmu(p)
\le 2 \Triangle \supzero(p) + o(1)
	\qquad(p\to p_c).
\end{equation}
This gives the $\Triangle \supmu(p)$ part of \eqref{eq:perc_boot}, since $\Triangle \supzero(p) \le \Triangle \supzero(p_c)$.

The $\tilde \Triangle \supmu(p)$ part of \eqref{eq:perc_boot} is analogous.
We first use $\norm{ P * \tau_{p}\supmu * \tau_{p} * \tau_{p} }_\infty
\le 2\norm{ P * \tau_\sym\supmu * \tau_{p} * \tau_{p} }_\infty$,
and then use the inverse Fourier transform to get
\begin{equation}
\norm{ P * \tau_\sym\supmu * \tau_{p} * \tau_{p} }_\infty
\le \norm{ P * \tau_p * \tau_{p} * \tau_{p} }_\infty + \norm {\hat P}_\infty \norm{\hat\tau_\sym\supmu - \hat\tau_p}_3 \norm {\hat \tau_p}_3^2.
\end{equation}
Since $\norm {\hat P}_\infty \le \norm{P(x)}_1  = 1$,
by Lemma~\ref{lem:Gsym} we get
$\tilde\Triangle \supmu(p)
\le 2 \tilde \Triangle \supzero(p) + o(1)$ as $p \to p_c$.
The $\Bsupmu(p)$ part of \eqref{eq:perc_boot} similarly follows from Lemma~\ref{lem:Gsym} and the fact that $L^2(\Td)$ norm is at most the $L^3(\Td)$ norm.
This completes the proof.
\end{proof}

\begin{corollary} \label{cor:perc-triangle}
Let $d \ge 15$.
There exist $\delta, \kappa > 0$ and  $r < 1$ such that \eqref{eq:perc_r}
and the bound $\Bsupmu(p) \le  \kappa$
hold
uniformly in $p \in [p_c - \delta, p_c)$ and in $\mu \in \overline \Omega_p$.
\end{corollary}

\begin{proof}
We fix $\gamma > 0$ sufficiently small and take $\delta > 0$ from Lemma~\ref{lem:perc_triangle}.
The uniform estimates \eqref{eq:perc_triangle} on $\Triangle \supmu(p)$, $\tilde\Triangle \supmu(p)$, and $\Bsupmu(p)$ extend to all $\mu \in \overline \Omega_p$ by Fatou's lemma, so we can take $\kappa = \B\supzero(p_c) + 2\gamma$.
The paragraph containing \eqref{eq:perc_boot_pf} shows that we can take any $r \in [0.98,1)$.
\end{proof}

\begin{proof}[Proof of \eqref{eq:Q_moments-perc-a} with $\zeta < \half(d-8)$ and \eqref{eq:Q_infrared-perc-a}]

Let $\delta, \kappa, r$ be given by Corollary~\ref{cor:perc-triangle}.
We can now use Lemmas~\ref{lem:perc_diagram-p1}--\ref{lem:perc_infrared} with these $\kappa,r$.
The norm estimates \eqref{eq:Q_moments-perc-a} follow immediately from Lemma~\ref{lem:perc_diagram-p1}, since $2+\zeta < \frac{d-4}2$
under our assumption that $\zeta \in (0,\frac{d-8}2)$.
For the infrared bound,
since $\abs \mu \lesssim \mp  \to 0$ as $p \to p_c$ by Lemma~\ref{lem:Omega_convex} and \eqref{eq:mass_asymp-perc},
the desired \eqref{eq:Q_infrared-perc-a} follows immediately from
\eqref{eq:perc_tilted_infrared}.
\end{proof}

The fact that \eqref{eq:Q_moments-perc-a} actually holds for any
$\zeta<\half(d-6)$ is proved above Proposition~\ref{prop:perc-Pip}.

\subsection{Diagrammatic estimate: proof of \eqref{eq:pnorm-perc-ii-a}}
\label{sec:perc-Pip}

We now prove the $L^2$ bound \eqref{eq:pnorm-perc-ii-a}, beginning with an easy lemma.

\begin{lemma} \label{lem:Gmu1}
Let $d \ge 15$.
There is a constant such that
\begin{equation}
\bignorm{ \abs x \tau_p \supmu(x) }_2 \lesssim 1
\end{equation}
uniformly in $p \in [p_c - \delta, p_c)$ and in $\mu \in \overline \Omega_p$.
\end{lemma}

\begin{proof}
By $\abs x \le \norm x_1 = \sum_{i=1}^d \abs{x_i}$ and  Parseval's identity, we have
\begin{equation}
\bignorm{ \abs x \tau_p \supmu(x) }_2
\le \sum_{i=1}^d \bignorm{ \grad_i \hat \tau_p \supmu(k) }_2
= \sum_{i=1}^d \biggnorm{
	\frac{ \grad_i \Pi_p \supmu } { 1 - \hat J_p \supmu }
	+ \frac{ 1+ \Pi_p \supmu} { 1 - \hat J_p \supmu }
		\frac{ \grad_i J_p \supmu } { 1 - \hat J_p \supmu }
	}_2 .
\end{equation}
For each term in the right-hand side,
we use Corollary~\ref{cor:perc-triangle}, Lemmas~\ref{lem:perc_diagram-p1}, \eqref{eq:perc_tilted_infrared}, and \eqref{eq:Jz^mu} to bound it by a constant multiple of
\begin{equation}
\biggnorm{ \frac{ 1} { \abs k^2 }
	+ \frac{ 1 } { \abs k^4 }	}_2 ,
\end{equation}
which is finite in dimensions $d > 8$. Summation over components then gives the desired result.
\end{proof}

The case $a=d-1$ and $q=2$ of the next proposition proves \eqref{eq:pnorm-perc-ii-a}.
Also, we can take $q = 1$ when $a < \half d - 1$.
For $a=2+\zeta$, this gives $\zeta < \frac 12 (d-6)$, which proves \eqref{eq:Q_moments-perc-a} for $\zeta < \frac 12 (d-6)$.

\begin{proposition} \label{prop:perc-Pip}
Let $d \ge 15$
and $a \in [0,d-1]$.
For any choice of $q \in [1,\infty]$ satisfying
\begin{equation} \label{eq:pa-restriction-perc}
\frac 1 q < 1 + \frac{ \frac 1 2 d - 1 - a } d ,
\end{equation}
there is a constant $K_{a,q} < \infty$ such that
\begin{equation} \label{eq:perc_moments-pnorm}
\bignorm{ \abs x^a  \Pi_p \supmu(x)  }_q  \le  K_{a,q}
\end{equation}
uniformly in $p \in [p_c - \delta, p_c)$ and in $\mu \in \overline \Omega_p$.
For $a = d-1$, \eqref{eq:perc_moments-pnorm} also holds with $q = 2$.
\end{proposition}

\begin{proof}
Since the $a \le 1$ case is covered by Corollary~\ref{cor:perc-triangle} and Lemma~\ref{lem:perc_diagram-p1}, we assume $a > 1$ and write $a = \phi + 1$  with $\phi \in (0,d-2]$.
As in \eqref{eq:percN0},
\begin{equation}\label{eq:percN0-p}
    \bignorm{ \abs x^a  \Pi_{p,0} \supmu(x) }_q
     \le
    \bignorm{ \abs x^\phi \tau_p(x) }_{q'}
    \bignorm{ \abs x^1 \tau_p\supmu(x) }_2,
\end{equation}
with $q\inv =(q')\inv + 2\inv$.
The last factor is bounded using Lemma~\ref{lem:Gmu1}, and
by \eqref{eq:taupc} the other factor
is bounded when
$q' \in [1,\infty]$ satisfies
\begin{equation}
\frac 1 {q'} < \frac{ d - 2 - \phi } d ,
\end{equation}
with the equality allowed when $\phi = d-2$ (in which case $q=\infty$).
Since $a = \phi + 1$,
this translates to the restriction
\begin{equation}
\frac 1 q
< \frac 1 2 +  \frac{ d - 2 - \phi } d
= 1 +  \frac{ \half d - 1 - a } d
\end{equation}
on $q$, with the equality allowed when $a = d-1$.
This agrees with the restriction on $q$ in \eqref{eq:pa-restriction-perc}.

For $\Pi\supmu_{p,N}$ with $N \ge 1$, standard but more involved diagrammatic
estimates can be used.
We apply a simplified version of the method of \cite[Section~4]{DL26}.
In a diagram for $\Pi_{p,N}(x)$,
we distribute $e^{\mu\cdot x}$ multiplicatively along the top path,
distribute $\abs x^1$ additively also along the top path,
and distribute $\abs x^\phi$ additively along the bottom path.
This can produce a line that is both tilted and weighted by $\abs x^1$, for which we use Lemma~\ref{lem:Gmu1}.
As in Proposition~\ref{prop:SAW-Pip} for the self-avoiding walk, we use the triangle inequality
to decompose the diagram.
In brief, we always put the $L^q$ norm on a part of the diagram that contains the $\abs x^\phi$-weighted line, and we estimate the left/right part of the diagram (which might contain the $\abs x^1$-weighted line) using usual methods.
Overall, if the ``$L^q$ bubble,'' the ``$L^q$ martini,'' and the ``$\abs x^1$-weighted triangle'' (defined below) are all $\lesssim 1$, then,
with $r<1$ from Corollary~\ref{cor:perc-triangle}, we get
\begin{equation}
\bignorm{ \abs x^a  \Pi_{p,N} (x) e^{\mu\cdot x} }_q
\lesssim (2N+1)^{a+2} r^{(N-3)\vee0 }
\end{equation}
for all $N\ge 1$.
Since $r < 1$, we can sum this over $N$ to get the desired \eqref{eq:perc_moments-pnorm}.

We now provide more details on the new nonstandard
diagram pieces that need to be bounded.
The $\abs x^1$-weighted triangle is
\begin{equation} \label{eq:weighted_triangle_pf}
\bignorm{ [ \abs x^1 \tau_p \supmu(x) ] * \tau_p * \tau_p }_\infty
\le \bignorm{ \abs x^1 \tau_p \supmu(x) }_2
	\norm{ \tau_p * \tau_p }_2 .
\end{equation}
This is $\lesssim 1$ for $d > 8$, by Lemma~\ref{lem:Gmu1}
and the pointwise estimate \eqref{eq:taupc}.
The two kinds of ``$L^q$ bubbles'' (with or without an $\abs x^1$-weighted line) can be summarised as
\begin{equation}
\sup_{u\in\Zd} \bignorm{
	(1+\abs{ x}^1) \tau_p \supmu(x)
	\abs{ x-u }^\phi \tau_p( x-u )
	}_{L^q_x} .
\end{equation}
This is bounded using H\"older's inequality as in \eqref{eq:percN0-p}.

Finally, for the ``$L^q$ martini'' diagram, we simply use the monotonicity of the  $L^q(\Zd)$ norms in $q$ to bound them by the following ($L^1$) version of \eqref{eq:percH} with extra weight:
\begin{align}
\label{eq:percH-p}
\tilde H^{(\mu),a}_p & = \sup_{b,c\in \Zd} \sum_{v,w,x,y,z\in \Zd}
	\Bigl( \tau_p(z) [ \abs x^\phi
    \tau_p(x)] \tau_p(w) \tau_p(x-w)\tau_p(y-x)
    \tau_p(v-w)
    \nl & \hspace{40mm} \times [(1+ |z+b-v|^1)
    \tau\supmu_p(z+b-v)]\tau\supmu_p(y+c-v) \Bigr).	
\end{align}
This can be decomposed and bounded by
\begin{equation} \label{eq:perc_p-pf}
\bignorm{ \abs x^\phi \tau_p(x) }_\infty
\bignorm{ [ (1+\abs x^1) \tau_p \supmu(x) ] * \tau_p * \tau_p }_\infty
\bignorm{ \tau_p \supmu * \tau_p * \tau_p * \tau_p  }_\infty .
\end{equation}
The first factor is finite since $\phi \le d-2$.
The middle factor
can be bounded by the sum of the
$\abs x^1$-weighted triangle bounded previously in \eqref{eq:weighted_triangle_pf}
and the last factor.
For the last factor of \eqref{eq:perc_p-pf},
we use the Fourier transform and the infrared bounds \eqref{eq:perc_infrared} and \eqref{eq:perc_tilted_infrared} to get
\begin{equation}
\bignorm{ \tau_p \supmu * \tau_p * \tau_p * \tau_p  }_\infty
\le   \bignorm{  \hat \tau_p \supmu (\hat \tau_p)^3 }_1
\lesssim 	\biggnorm{ \frac{ 1 }{ \abs k^8 } }_1
< \infty
\end{equation}
in dimensions $d > 8$.  This completes the proof.
\end{proof}

\appendix

\section{Brownian Green function}
\label{app:Brown}

In this appendix, we prove several elementary results regarding the
Brownian Green function on $\Rd$, defined by
$\bg(x; \eta, \Lambda) = \int_0^\infty \D t \, \rho_t(x; \eta, \Lambda)$ with
\begin{equation} \label{eq:def_rho-app}
\rho_t(x; \eta, \Lambda)
= \frac 1 { \sqrt{\det \Lambda } } \frac{ 1 }{(2\pi t)^{d/2} }
	\exp \Bigl\{ - \frac 1 {2t} (x-t\eta) \cdot \Lambda\inv (x - t\eta) \Bigr\} .
\end{equation}
We also prove Corollaries~\ref{cor:OZ} and~\ref{cor:crossover_weak}.

\subsection{Representation as Bessel function}

We use the modified Bessel function $K_\nu(z)$ of the second kind,
with $\nu = (d-2) / 2$.
For $d=1$ or $3$, $K_\nu(z)$ has the explicit formulas \cite[8.469.3]{GR07}\footnote{
This is a special case of \cite[8.468]{GR07}, which
for integers $n \ge 0$ gives the
explicit formula
\begin{equation*}
    K_{n+\frac 12}(z) =
    \Big( \frac{\pi}{2z} \Big)^{1/2} e^{-z}
    \sum_{k=0}^n \frac{(m+k)!}{k!(n-k)! (2z)^k}
    	\qquad (z > 0).
\end{equation*}
}
\begin{equation} \label{eq:K13}
K_{-1/2}(z)=K_{1/2}(z) = \Big( \frac{\pi}{2z} \Big)^{1/2} e^{-z}
	\qquad (z > 0).
\end{equation}
An integral representation for the Bessel function \cite[8.432.6]{GR07} is
\begin{equation} \label{eq:Kint}
K_\nu(z) = \frac 12 \Big(\frac z2\Big)^\nu \int_0^\infty \frac{1}{u^{\nu+1}}e^{-u-z^2/(4u)}\D u
    \qquad    (\nu \in \R,\, z>0).
\end{equation}
The asymptotic behaviour of $K_\nu(z)$ is known to be
\begin{equation}
\label{eq:Kasy-infty}
    K_\nu(z) \sim \Big( \frac{\pi}{2z} \Big)^{1/2} e^{-z}
    \qquad (z \to \infty)
\end{equation}
and
\begin{equation}
\label{eq:Kasy-0}
    K_\nu(z) \sim \frac{\Gamma(\nu)}{2} \Big(\frac{2}{z}\Big)^{\nu}
    \quad (\nu > 0,\,z \to 0),
    \qquad
    K_0(z) \sim \log\frac 1z +c_0
    \quad (z \to 0),
\end{equation}
where $c_0=\log 2 - \gamma >0$ with $\gamma$  the Euler--Mascheroni constant.

Let $x,\eta \in \Rd$, $x\ne 0$, and let $\Lambda \in \R^{d\times d}$ be symmetric and positive-definite. We define
\begin{equation}
a^2 =   \eta \cdot \Lambda\inv \eta,
\qquad
b^2 =  x \cdot \Lambda\inv x .
\end{equation}

\begin{lemma}
\label{lem:CBess}
Let $d\ge 1$ and $\nu = (d-2)/2$.
Let $x,\eta\in\R^d$ with $x\neq 0$,
and also $\eta \neq 0$ if $d\le 2$.
Let $\Lambda \in \R^{d\times d}$ be symmetric and positive-definite.
Then
\begin{equation} \label{eq:CBess}
\bg(x; \eta, \Lambda)
= \frac { 2} { (2\pi)^{d/2}   \sqrt{\det \Lambda }  }
\left( \frac ab \right)^\nu
K_\nu(ab)  e^{ x \cdot \Lambda\inv \eta }.
\end{equation}
In particular, for $d=1$ and $d=3$ we have
\begin{alignat}2
\bg(x; \eta, \Lambda)
&=
\frac{ 1 } { |\eta| }
	e^{ x \cdot \Lambda\inv \eta -ab},
   &&(d=1),
    \\
    \bg(x; \eta, \Lambda)
&= \frac { 1} { 2\pi \sqrt{\det \Lambda } }
\frac{ 1 } {\sqrt{x \cdot \Lambda\inv x} }  e^{ x \cdot \Lambda\inv \eta -ab}
    \qquad&& (d=3).
\end{alignat}
\end{lemma}

\begin{proof}
By inserting the definitions of $a$ and $b$ into \eqref{eq:def_rho-app}, we have
\begin{equation}
\rho_t(x; \eta, \Lambda)
= \frac {e^{ x \cdot \Lam\inv \eta}} {(2\pi )^{d/2}  \sqrt{\det \Lambda } }
 \frac{ 1 }{t^{d/2} }
	e^{-  a^2t/2- b^2/(2t) } .
\end{equation}
We integrate this to get $\bg(x;\eta,\Lam)$.
Let $z=ab$.
The $t$ integral can be computed via the change of variables $u=a^2t/2$ and \eqref{eq:Kint} to be
\begin{equation}
    \int_0^\infty  \frac{ 1 }{t^{d/2} } 	e^{-a^2t/2-b^2/(2t) } \D t
     =
    \Bigl( \frac{a^2}{2} \Bigr)^\nu \int_0^\infty  \frac{ 1 }{u^{\nu+1} } 	e^{-u-z^2/(4u) } \D u
=    2 \Bigl( \frac ab \Bigr)^{\nu} K_\nu(ab) .
\end{equation}
Collecting terms then gives the desired \eqref{eq:CBess}.
The exact formulas for $d=1,3$ follow from \eqref{eq:K13}.
\end{proof}

\subsection{Asymptotic behaviour}

We now consider the case where $x \ne 0$ is in aligned with the drift $\eta$. That is, we assume
\begin{equation}
    \eta=\abs\eta  \hat x,
    \qquad
    \hat x = \frac x { \abs x } ,
\end{equation}
where $\abs x$ is the Euclidean norm.
Define $\w = \sqrt{ \hat x \cdot \Lam\inv \hat x }$, so
\begin{equation}
a = \w \abs \eta  ,
\qquad
 b = \w \abs x  ,
\qquad
x\cdot \Lam\inv \eta = ab = \w^2 \abs \eta \abs x  .
\end{equation}
By inserting the above into \eqref{eq:CBess}, we get
\begin{equation} \label{eq:Brown_pf}
\bg(x; \abs \eta  \hat x, \Lam)
= \frac { 2} { (2\pi)^{d/2}   \sqrt{\det \Lambda }  }
	\Bigl(  \frac{ \abs \eta  } { \abs x } \Bigr)^\nu
	K_\nu( \w^2 \abs \eta   \abs x ) e^{ \w^2 \abs \eta   \abs x }  .
\end{equation}
We note in passing that for $d>2$, in the limit where $\eta \to 0$ and $\Lambda \to \Lambda_0$
(with $\Lambda_0$ positive definite), it follows from
\eqref{eq:Brown_pf} and \eqref{eq:Kasy-0} that, for nonzero $x$,
\begin{equation}
\label{eq:Ccrit}
    \bg(x; \abs \eta  \hat x, \Lam)
    \to
    \frac{\Gamma(\frac{d-2}{2})}{2\pi^{d/2}}\frac{1}{\sqrt{\det\Lam_0}}
    \frac{1}{(x \cdot \Lam_0\inv x)^{(d-2)/2}}.
\end{equation}

\begin{proof}[Proof of Corollary~\ref{cor:OZ}]
By Theorem~\ref{thm:crossover}, as $\abs x \to \infty$ we have
\begin{equation} \label{eq:OZ_pf}
\GL(x) =
\bg(x; \eta_{\hat x} , \Lambda_{\hat x}) e^{-\mS  \abs x_S} [ \hat \g \supmux (0) + o(1) ] .
\end{equation}
Since $x$ is aligned with $\eta_\hatx$ by Proposition~\ref{prop:geom-intro}, we can use \eqref{eq:Brown_pf} and \eqref{eq:Kasy-infty} with $\eta = \eta_\hatx$, to get
\begin{align}
\bg(x; \eta_\hatx, \Lam_\hatx)
&\sim \frac { 2} { (2\pi)^{d/2}   \sqrt{\det \Lambda_\hatx }  }
	\biggl(  \frac{ \abs{ \eta_\hatx } } { \abs x } \biggr)^\nu
	\biggl( \frac \pi {2 (\hat x \cdot \Lam_\hatx \inv \hat x) \abs{ \eta_\hatx}   \abs x  } \biggr)^{1/2}
\nnb &=
\frac 1 { (2\pi)^{(d-1)/2} \sqrt{\det\Lambda_{\hat x}} }
	\frac{1}{ (\hat x \cdot \Lam_{\hat x}\inv \hat x)^{1/2} }
	\frac{ \abs{ \eta_\hatx} ^{(d-3)/2} } { \abs x ^{(d-1)/2} }  .
\end{align}
Inserting this into \eqref{eq:OZ_pf} gives the desired result.
\end{proof}

The following lemma is an estimate on $\bg$ that is uniform in $x,\eta,\Lam$.

\begin{lemma} \label{lem:C_bdds}
Let $d \ge 1$,
$x,\eta\in \R^d$, and let $\Lambda \in \R^{d\times d}$ be symmetric and positive-definite.
Suppose $\w = \sqrt{ \hat x \cdot \Lam\inv \hat x} \asymp 1$ for all $\hat x$.
Then, for $d > 2$ we have
\begin{equation} \label{eq:C-combo}
\bg(x;\abs \eta \hat x,\Lam)    \asymp	\Casy(x;\abs \eta \hat x)
\quad \text{with} \quad
\Casy(x;\eta)=
\frac{ ( \abs \eta  \vee \abs x\inv )^{(d-3)/2} }{ \abs x^{(d-1)/2} }
\end{equation}
uniformly in $\eta$ and in $x\ne 0$.
If $d \le 2$ and $\so > 0$ is a given constant,
\eqref{eq:C-combo} still holds subject to the condition that $\abs \eta \abs x \ge \so$, with constants that depends on $\so$.
\end{lemma}

\begin{proof}
For $d > 2$, we have $\nu = \frac{d-2}{2} > 0$,
so we can use the asymptotics for $K_\nu(z)$ from \eqref{eq:Kasy-infty} and \eqref{eq:Kasy-0} to get
\begin{equation} \label{eq:C_bdds_pf}
K_\nu(z) e^z \asymp  \frac{  ( 1 \vee z )^{\nu - 1/2} }{ z^\nu }
	\qquad (\nu > 0,\, z > 0) .
\end{equation}
We insert this into \eqref{eq:Brown_pf}, and use $\w \asymp 1$, to obtain
\begin{align}
\bg(x; \abs \eta  \hat x, \Lam)
&\asymp  \Bigl(  \frac{ \abs \eta  } { \abs x } \Bigr)^\nu
	\frac{ ( 1 \vee \abs \eta \abs x)^{\nu - 1/2} }{ ( \abs \eta \abs x )^\nu }
\nnb & =   \frac{ ( 1 \vee \abs \eta  \abs x )^{\nu - 1/2} } { \abs x^{2\nu} }
=   \frac{ (\abs \eta  \vee \abs x\inv  )^{( d-3)/2} } { \abs x^{(d-1)/2} }  ,
\end{align}
as desired.
For $d\le 2$,
we restrict to $\abs \eta \abs x \ge \so > 0$,
which implies that $\w^2 \abs \eta \abs x \ge \so'$ for some $\so' > 0$.
In this case, we use \eqref{eq:Kasy-infty} to get \eqref{eq:C_bdds_pf} only for $z \ge \so'$, with constants depending on $\so'$.
This allows us to conclude in the same way.
\end{proof}

\begin{proof}[Proof of Corollary~\ref{cor:crossover_weak}]
We first consider $d >2$.
In Theorem~\ref{thm:crossover}, we choose $R$ such that the error term in \eqref{eq:main_asymp-intro} has absolute value at most $1/(2M)$.
It then follows from
Theorem~\ref{thm:crossover}, from $\hat \g\supmux(0) \in [M\inv, M]$,
and from
Lemma~\ref{lem:C_bdds}, that if $\abs x \ge R$ then
\begin{equation}
\GL(x) \asymp \Casy(x, \eta_\hatx) e^{- \mS  \abs x_S }
= \frac{  (1 \vee \abs{ \eta_\hatx} \abs x )^{(d-3)/2}  } { \abs x^{d-2} }
	e^{- \mS  \abs x_S } .
\end{equation}
Since $\abs{ \etax }  \asymp \mS $ by \eqref{eq:crossover-bds-intro},
and since $\abs x_S \asymp \abs x$ by Corollary~\ref{cor:norm},
this proves that there exists $R' > 0$ such that
\begin{equation}
\GL(x) \asymp
	\frac{  \max\{ 1 , \mS  \abs x_S \} ^{ (d-3)/2} } { \abs x_S^{d-2} }
	e^{- \mS  \abs x_S }
	\qquad (\abs x_S \ge R') ,
\end{equation}
with constants that depend only on $d,M,\KIR,\zeta$.
This concludes the proof for $d>2$.

Now suppose $d\le 2$ and $\so > 0$. The proof for $d>2$ still applies to $x$ with $\abs x \ge \so/ ( m_S \sqrt d)$, so we get
\begin{equation}
\GL(x) \asymp
	\frac{  1  } { \mS   ^{ (3-d)/2}  \abs x_S^{ (d-1)/2} }
	e^{- \mS  \abs x_S }
	\qquad (\abs x_S \ge R \vee \frac{ \so }{ \mS  })
\end{equation}
with constants that depend only on $d,M,\KIR,\zeta,\so$.
This completes the proof.
\end{proof}

\subsection{Massive Green function}

The massive Green function is defined
for $a \ge 0$ if $d>2$, and for $a>0$ for $d \le 2$, by
\begin{equation} \label{eq:def_mG-app}
    \mG_a(x) = \int_0^\infty \D t\, \rho_t(x;0,\mathrm{Id}) e^{-\half ta^2}
	=
    \int_0^\infty \frac{ \D t}{(2\pi t)^{d/2} } e^{-\frac{1}{2t}|x|^2} e^{-\half ta^2} .
\end{equation}
The change of variables $t \mapsto s^2 t'$
shows that it obeys the scaling identity
\begin{equation} \label{eq:Gscaling}
\mG_a(x) = \frac 1 {s^{d-2}} \mG_{as}(x / s)
\qquad (s > 0) .
\end{equation}

\begin{lemma} \label{lem:C=G}
Let $\Lambda$ be symmetric and positive-definite,
and let $x,\eta \in \R^d$.
Define $a^2= \eta\cdot \Lambda\inv \eta$.
Then
\begin{equation} \label{eq:C=G}
\bg(x;\eta,\Lam)
=  \frac 1 {\sqrt{\det \Lam}}
\mG_{a}( \Lam^{-1/2} x) e^{x\cdot \Lam\inv \eta}
=  \frac {1} { \abs x^{d-2} \sqrt{\det \Lam}}
\mG_{a\abs x}(   \Lam^{-1/2} \hat x )
	e^{x\cdot \Lam\inv \eta} .
\end{equation}
\end{lemma}

\begin{proof}
The second equality of \eqref{eq:C=G} follows from \eqref{eq:Gscaling}.
For the first equality, we take the Fourier transform
and observe that
\begin{equation}
\hat \mG_a(k) \inv = \half \abs k^2 +  \half a^2 ,\qquad
\hat \bg(k;\eta,\Lam) \inv = \half k \cdot \Lam k + i k \cdot \eta .
\end{equation}
If $\mu \in \Rd$, we have
\begin{equation}
\hat \mG_a( \Lam^{1/2} k + i \Lam^{1/2} \mu) \inv
= \half k \cdot \Lam k  + i \mu \cdot \Lam k
	- \half \mu \cdot \Lam \mu  + \half a^2 ,
\end{equation}
so the choice $\mu = \Lambda^{-1}\eta$
and $a^2  = \mu \cdot \Lam \mu = \eta \cdot \Lam\inv \eta$ gives
$\hat \bg(k;\eta,\Lam)
=\hat \mG_a( \Lam^{1/2} k + i \Lam^{1/2} \mu) $.
Taking the inverse Fourier transform then gives the first equality of \eqref{eq:C=G}.
\end{proof}

\begin{lemma} \label{lem:G1}
For $d \ge 1$ and $\phi \ge 0$,
\begin{equation}
    \int_{\Rd} |y|^\phi  \mG_{1}(y) \D y
      =
    2^{\phi+1}
    \frac{\Gamma( \frac{\phi+2}{2})\Gamma( \frac{\phi+d}{2} )}{\Gamma(\frac d2)}
    .
\end{equation}
\end{lemma}

\begin{proof}
By definition, Fubini's theorem, and the change of variables $y = t^{1/2} v$,
\begin{align}
    \int_{\Rd} |y|^\phi  \mG_{1}(y) \D y
     &=
    \int_0^\infty e^{-\frac 12 t}
    \int_{\Rd} |y|^\phi \frac{1}{(2\pi t)^{d/2}} e^{-\frac{1}{2t} |y|^2}
    \D y \D t  \nl
    &=
    \int_0^\infty e^{-\frac 12 t} t^{\phi/2}
    \int_{\Rd} |v|^\phi  \frac{1}{(2\pi)^{d/2}} e^{-\frac 12 |v|^2}
    \D v \D t  .
\end{align}
The $v$-integral is the $\phi^{\rm th}$ moment of a chi-squared distribution with $d$ degrees of freedom, so
\begin{equation}
\int_{\Rd} |y|^\phi  \mG_{1}(y) \D y
= 2^{\phi / 2} \frac{ \Gamma( \frac{d+\phi} 2) } { \Gamma( \frac d 2) }
	\int_0^\infty e^{-\frac 12 t} t^{\phi/2}	 \D t
= 2^{\phi / 2} \frac{ \Gamma( \frac{d+\phi} 2) } { \Gamma( \frac d 2) }
	\cdot 2^{(\phi+2) / 2} \Gamma\Bigl(\frac{ \phi+2} 2 \Bigr) ,
\end{equation}
which is the desired result.
\end{proof}

\section{Proof of Theorem~\ref{thm:crossover} for $d\le 2$}
\label{app:1-2D}

We have proved Theorem~\ref{thm:crossover} for dimensions $d> 2$ in Section~\ref{sec:pf_crossover}.
In this appendix, we present the proof for the remaining dimensions $d\le 2$.
The proof differs in the way we take the limit $\theta \to 1$.

\begin{proof}[Proof of Theorem~\ref{thm:crossover} for $d\le 2$]
Suppose Assumptions~\ref{ass:Omega} and \ref{ass:J} hold.
Let $x\ne 0$, $\theta \in [0,1)$, and $\mux\in \del \Omega$ be given by Proposition~\ref{prop:geom-intro}.
As in \eqref{eq:S_int_pf}, we have
\begin{equation}
S(x) e^{\theta \mu_{\hat x}\cdot x}
= \int_\Td \frac{\D k }{(2\pi)^d}
	\frac { \hat \g\supk{\theta \mux} (k)  } { 1 - \hat J\supk{\theta \mux} (k) }   e^{ik\cdot x}
\end{equation}
and want to take $\theta \to 1$, but we lack
a dominating function when $d\le 2$.
Instead, we use the identity $z\inv = \int_0^\infty e^{-t z} \D t$ for $\Re(z) > 0$ and Fubini's theorem (which applies when $\theta<1$) to write
\begin{equation} \label{eq:S_int_low}
S(x) e^{\theta \mu_{\hat x}\cdot x}
= \int_0^\infty \D t    \int_\Td \frac{\D k }{(2\pi)^d} e^{ik\cdot x}
	\hat \g\supk{\theta \mux} (k)
		e^{-t[ 1 - \hat J\supk{\theta \mux} (k) ] }
= \int_0^\infty \D t  \,e^{-t[ 1- \hat J\supk{\theta \mux} (0) ] } I_{t,\theta}(x)
\end{equation}
with
\begin{equation}
I_{t,\theta}(x) =  \int_\Td \frac{\D k }{(2\pi)^d}  e^{ik\cdot x}
	\hat \g\supk{\theta \mux} (k)
	e^{-t[ \hat J\supk{\theta \mux} (0) - \hat J\supk{\theta \mux} (k) ] }  .
\end{equation}
We will prove below that
there exists a $\delta > 0$ for which
\begin{equation} \label{eq:It_theta_claim}
\int_0^\infty \D t
	\sup_{ \theta \in [1- \delta ,1]} \abs{ I_{t,\theta}(x) }
	< \infty.
\end{equation}
Given this, and using $\hat J\supk{\theta \mux} (0) \le 1$ from Proposition~\ref{prop:geom-intro}, we see that $\sup_{ \theta \in [1- \delta ,1]} \abs{ I_{t,\theta}(x) }$ is a dominating function for taking the limit of \eqref{eq:S_int_low} as $\theta \to 1$.
Also, since $\mu \mapsto \hat J\supmu(k)$ and $\mu \mapsto \g\supmu(k)$ are both continuous functions on $\mu \in \overline \Omega$, we have (using Bounded Convergence Theorem for $I_{t,\theta}$)
\begin{equation}
\lim_{\theta \to 1}
	e^{-t[ 1- \hat J\supk{\theta \mux} (0) ] } I_{t,\theta}(x)
= e^{-t[ 1- \hat J\supk{\mux} (0) ] } I_{t,1}(x) .
\end{equation}
Then, since $\hat J\supmux(0) = 1$
and $\eta_\hatx = \abs{ \eta_\hatx} \hat x$ by Proposition~\ref{prop:geom-intro},
it follows from dominated convergence and Theorem~\ref{thm:deconv-intro} that
\begin{equation}
S(x) e^{\mu_{\hat x}\cdot x}
= \int_0^\infty \D t  \, I_{t,1}(x)
= \bg(x; \eta_\hatx, \Lam_\hatx) \Bigl[\hat \g\supk{\mux} (0)  + O\Bigl(\frac 1 {\abs x^\eps}\Bigr) \Bigr]
\end{equation}
as $\abs x\to \infty$, subject to the condition that
$\abs x \abs{ \eta_x } \ge \so' > 0$ for some $\so' > 0$.
Since $\abs { \eta_x } \gtrsim \mS $ by Proposition~\ref{prop:geom-intro}, we do have this condition by the hypothesis that $\mS  \abs x \ge \so > 0$.
This gives the desired result after multiplication by $e^{-\mu_{\hat x}\cdot x} = e^{- \mS  \abs x_S}$, using Corollary~\ref{cor:norm}.

It remains to prove \eqref{eq:It_theta_claim}.
We write $\eta_\theta = \sum_{y\in \Zd} y J\supk{\theta \mux}(y)$.
Since the function $\mu \mapsto  \sum_{y\in \Zd} y J\supk{\mu}(y)$ is continuous on $\mu \in \overline \Omega$, we can take $\delta > 0$ sufficiently small so that $\abs{ \eta_\theta - \eta_1 } \le \half \abs{ \eta_1}$ for all $\theta \in [1-\delta, 1]$.
Using Lemma~\ref{lem:It_decay}, for $n=0,2$ we have
\begin{equation}
\abs{ I_{t,\theta} (x) }
	\lesssim \frac 1 {(1+\sqrt t  )^d}
	\biggl( \frac{ 1 + \sqrt t } { \abs{ x - \integer{t \eta_\theta} } } \biggr)^n,
\end{equation}
with constants depending only on $n,d,M,\KIR$.
Let
\begin{equation}
T = 4 \times \frac{ \abs x + \half \sqrt d } { \abs{ \eta_1 } } .
\end{equation}
We use the $n=0$ bound when $t \le T$ and use the $n=2$ bound when $t\ge T$.
For $t\ge T$,
\begin{equation}
\abs{ x - \integer{t \eta_\theta} }
\ge \abs{ x - {t \eta_\theta} } -  \abs{  t \eta_\theta - \integer{t \eta_\theta} }
\ge t \frac{ \abs {\eta_1} }2 - \abs x - \half \sqrt d
\ge  t \frac{ \abs {\eta_1} }4 .
\end{equation}
Thus, we get
\begin{equation}
\sup_{ \theta \in [1- \delta ,1]} \abs{ I_{t,\theta}(x) }
\lesssim \begin{cases}
1						& (t \le T) \\
\frac 1 {(1+\sqrt t  )^d}
	\bigl( \frac{ 1 + \sqrt t } { t \abs {\eta_1 } } \bigr)^2
						& (t \ge T) ,
\end{cases}
\end{equation}
which is integrable over $t$ in all dimensions $d > 0$.
This completes the proof.
\end{proof}

\section*{Acknowledgements}
The work of YL was partly undertaken at the University of British
Columbia.
The work of YL was supported in part by the National Natural Science Foundation of China (Grant No.~12595284, 12595280).
The work of both authors was supported in part by the Natural Sciences and Engineering Research Council of Canada (NSERC), Grant No.~GR010086.
We thank Romain Panis for comments on a preliminary version and Yvan
Velenik for advice on the literature.


\end{document}